\documentclass[12pt]{article}
\usepackage{amsmath,amssymb,amsthm,graphicx}
\usepackage[subrefformat=parens,labelformat=parens]{subfig}
\usepackage[all]{xy}   
\usepackage{hyperref}
\usepackage{verbatim}
\usepackage[dvipsnames]{xcolor}
\usepackage{tikz}
\usetikzlibrary{positioning,shapes,shadows}
\usetikzlibrary{arrows,automata}

\newtheorem{theorem}{Theorem}
\newtheorem{proposition}[theorem]{Proposition}
\newtheorem{lemma}[theorem]{Lemma}

\newtheorem{corollary}[theorem]{Corollary}
\theoremstyle{definition}
\newtheorem{definition}[theorem]{Definition}
\newtheorem*{definition*}{Definition}

\newtheorem{note}[theorem]{Note}
\newtheorem{remark}[theorem]{Remark}
\newtheorem{notation}[theorem]{Notation}
\newtheorem{notes}[theorem]{Notes}
\newtheorem{example}[theorem]{Example}

\newcommand{\newword}[1]{\textbf{#1}}
\newcommand{\Z}{\mathbb{Z}}
\newcommand{\N}{\mathbb{N}}
\newcommand{\F}{\mkern 3.75mu\overline{\mkern-3.75mu F}}

\newcommand{\R}{\mathcal{R}}

\renewcommand{\S}{\mathcal{S}}
\newcommand{\oS}{\mkern 3mu\overline{\mkern-3mu S}}
\newcommand{\A}{\mathcal{A}}

\newcommand{\Ain}{A_{\scriptscriptstyle\mathrm{in}}}
\newcommand{\Aout}{A_{\scriptscriptstyle\mathrm{out}}}
\newcommand{\tast}{\text{\small$*$}}
\newcommand{\emptystring}{\varepsilon}

\newcommand{\od}{\mkern 3mu\overline{\mkern-3mu d\mkern 1.5mu}\mkern-1.5mu}
\newcommand{\of}{\mkern 3.5mu\overline{\mkern-3.5mu f\mkern0.75mu}\mkern-0.75mu}
\newcommand{\og}{\mkern 2mu\overline{\mkern-2mu g\mkern0.5mu}\mkern-0.5mu}
\newcommand{\oh}{\mkern 2mu\overline{\mkern-2mu h\mkern-0.25mu}\mkern 0.25mu}
\newcommand{\hb}{\partial_h}

\definecolor{myred}{rgb}{0.75,0,0}
\definecolor{myblue}{rgb}{0.2,0.4,1}

\tikzset{initial text={}}
\makeatletter
\tikzstyle{initial by arrow}=   [after node path=
{
  {
    [to path=
    {
      [->,double=none,every initial by arrow]
      ([shift=(\tikz@initial@angle:\tikz@initial@distance)]\tikztostart.\tikz@initial@angle)
          node [shape=rectangle,anchor=\tikz@initial@anchor] {\tikz@initial@text}
        -- (\tikztostart)}]
    edge ()
  }
}]
\makeatother

\begin{document}

\title{Rational Embeddings of Hyperbolic Groups}

\author{James Belk%
\footnote{University of St Andrews, St Andrews, Scotland. \texttt{\href{mailto:jmb42@st-andrews.ac.uk}{jmb42@st-andrews.ac.uk}}}
\and Collin Bleak%
\footnote{University of St Andrews, St Andrews, Scotland. \texttt{\href{mailto:cb211@st-andrews.ac.uk}{cb211@st-andrews.ac.uk}}}
\and
Francesco Matucci\footnote{Universidade Estadual de Campinas (UNICAMP), S\~{a}o Paulo, Brasil. \texttt{\href{mailto:francesco@ime.unicamp.br}{francesco@ime.unicamp.br}}}
\;\footnote{The first and second authors have been partially supported by EPSRC grant EP/R032866/1 during the creation of this paper.  The third author is a member of the Gruppo Nazionale per le Strutture Algebriche, Geometriche e le loro Applicazioni (GNSAGA) of the Istituto Nazionale di Alta Matematica (INdAM) and gratefully acknowledges the support of the 
Funda\c{c}\~ao de Amparo \`a Pesquisa do Estado de S\~ao Paulo 
(FAPESP Jovens Pesquisadores em Centros Emergentes grant 2016/12196-5),
of the Conselho Nacional de Desenvolvimento Cient\'ifico e Tecnol\'ogico (CNPq 
Bolsa de Produtividade em Pesquisa PQ-2 grant 306614/2016-2) and of
the Funda\c{c}\~ao para a Ci\^encia e a Tecnologia  (CEMAT-Ci\^encias FCT project UID/Multi/04621/2013).}
}

\maketitle

\begin{abstract}
We prove that all Gromov hyperbolic groups embed into the asynchronous rational group defined by Grigorchuk, Nekrashevych and Sushchanski\u\i.
The proof involves assigning a system of binary addresses to points in the Gromov boundary of a hyperbolic group~$G$, and proving that elements of $G$ act on these addresses by asynchronous transducers.  These addresses derive from a certain self-similar tree of subsets of~$G$, whose boundary is naturally homeomorphic to the horofunction boundary of~$G$.
\end{abstract}

\tableofcontents 

\section*{Introduction}
 Let $\{0,1\}^\omega$ denote the Cantor set of all infinite binary sequences.  A homeomorphism of $\{0,1\}^\omega$ is said to be \newword{rational} if there exists an asynchronous transducer (i.e.~an asynchronous Mealy machine) that implements the homeomorphism on infinite binary strings.  In~\cite{GNS}, Grigorchuk, Nekrashevych and Sushchanski\u\i\ observe that the set of all rational homeomorphisms of $\{0,1\}^\omega$ forms a group $\mathcal{R}$ under composition, which they refer to as the \newword{rational group}.  They also observe that the group of rational homeomorphisms of $A^\omega$ is isomorphic to~$\R$ for any finite alphabet~$A$ with at least two elements.

Here the word \newword{asynchronous} refers to transducers that can output a finite binary sequence of any length each time they take a digit as input.  This is a generalization of \newword{synchronous} transducers, which are required to output a single binary digit each time they take a digit of input. The asynchronous rational group $\R$ contains the group of synchronous rational homeomorphisms corresponding to any finite alphabet.

Our main focus is on embedding questions for the rational group~$\R$.  We prove:

\begin{theorem}\label{thm:MainTheorem}Every hyperbolic group embeds into~$\R$.
\end{theorem}

Here a \newword{hyperbolic group} is a finitely generated group $G$ whose Cayley graph satisfies Gromov's thin triangles condition (see~\cite{BrHa}).  This is a vast class of finitely presented groups: in a precise sense, ``generic'' finitely presented groups are hyperbolic~\cite{Champetier,Ol'shanskiiYu}.

There are compelling practical features of groups realised as groups of homeomorphisms of Cantor spaces realisable as finite-state transducers. For example, one can directly understand how such group elements act on their respective Cantor spaces and one can study the specific combinatorics of the transducers representing these group elements.  An example of the impact that can be created by realising group elements this way is provided by Grigorchuk and Zuk.  By realising the lamplighter group as a group of synchronous automata, they are able to compute the spectrum of the resulting group~\cite{grigorchuk2001lamplighter}.  This yields a counterexample to a strong form of Atiyah's Conjecture about the range of values of the spectrum of $L^2$-Betti numbers for closed manifolds~\cite{grigorchuk2000question}.

If we consider the case of the groups Aut$(\{0,1,\ldots,n-1\}^\Z,\sigma)$ of automorphisms of (full) shift spaces, then Grigorchuk, Nekrashevych, and Suschanski\u\i\, in \cite{GNS} answer a request by Kitchens for a new combinatorial realisation of elements of those automorphism groups.  Kitchens in \cite{Kitchens} states that a major obstacle in the progression of understanding groups of automorphisms of shift spaces has been a lack of a practical combinatorial description for elements of these groups.  Grigorchuk, Nekrashevych, and Suschanski\u\i\, give an embedding of Aut$(\{0,1,\ldots,n-1\}^\Z,\sigma)$ in $\R$.  There is now a second realisation arising from the recent description of the group Aut$(G_{n,r})$ of automorphisms of the Higman-Thompson group $G_{n,r}$ (for $1\leq r<n$ natural numbers) as a group of transducers acting on a specific Cantor space $\mathfrak{C}_{n,r}$ \cite{BCMNO}, which also exposed an unexpected connection between subgroups of the outer automorphism group of the Higman-Thompson group $G_{n,r}$ and Aut$(\{0,1,\ldots,n-1\}^\Z,\sigma)$. This connection arose through the study in \cite{BCMNO,BCO} of the special combinatorial properties of the transducers representing the group elements of Aut$(G_{n,r})$ for such $n$ and $r$, and leads to an explicit combinatorial realisation of elements of Aut$(\{0,1,\ldots,n-1\}^\Z,\sigma)$ which exposes these groups' structures as non-split extensions over a central $\Z$.

Groups of synchronous transducers have received much attention in the literature, primarily as this class of groups contain numerous `exotic' groups providing examples of unusual or unexpected behaviour (e.g., \cite{GrigBurnside,sidkiBurnside,GrigQuestions,NekrashevychIMG,BartholdiKaimonovichNekrashevychAmenabilty,NekrashevychSidkiTreeAutGrps} provides a very incomplete list of papers).  While these groups do provide counterexamples to various forms of the Burnside conjecture and Milnor's conjecture, they also remain natural in many ways.  Indeed, this class houses well known foundational groups which arise in other circumstances, including free groups~\cite{Vorobets-Vorobets}, $\mathrm{GL}_n(\mathbb{Z})$ and its subgroups~\cite{Brunner-Sidki-1}, the solvable Baumslag-Solitar groups $BS(1, m)$~\cite{Bartholdi-Sunik-1}, and the generalized lamplighter groups~$(\mathbb{Z}/n\mathbb{Z}) \wr \mathbb{Z}$~\cite{Silva-Steinberg-1}.

On the other hand, comparatively little attention has been paid to the more complex class of groups generated by asynchronous transducers, and the full asynchronous rational group~$\mathcal{R}$ of Grigorchuk, Nekrashevych, and Sushchanski\u\i. It is known that $\R$ is simple and not finitely generated~\cite{Belk-Hyde-Matucci-1}.  Also, while the word problem is solvable in finitely generated subgroups of~$\mathcal{R}$~\cite{GNS}, the periodicity problem for elements of $\mathcal{R}$ has no solution~\cite{Belk-Bleak-1}.  Finally, the group $\R$ houses `exotic' groups of another type: the R.\ Thompson groups $F$, $T$, and $V$ all embed into~$\mathcal{R}$~\cite{GNS}, as do the Brin-Thompson groups~$nV$ (see~\cite{Belk-Bleak-1} for the embedding of the group $2V$) and groups such as the R\"over group $V_\Gamma$ (a finitely presented simple group which is marriage of Grigorchuk's group $\Gamma$ with the R.\ Thompson group $V$, see \cite{Roever}).  Any group of synchronous automata embeds into~$\R$, so $\R$ also contains the groups mentioned earlier.

The proof  of Theorem \ref{thm:MainTheorem} is dynamical as opposed to algebraic. Indeed, there is a general dynamical procedure for showing that a group embeds into~$\R$, which can be defined as follows.

\begin{definition}\label{def:RationalAction}Let $G$ be a group acting by homeomorphisms on a compact metrizable space~$X$.  We say that the action of $G$ on $X$ is \newword{rational} if there exists a quotient map $q\colon \{0,1\}^\omega \to X$ and a homomorphism $\varphi\colon G\to\R$ such that the diagram
\[
\xymatrix@R=0.5in{
\{0,1\}^\omega \ar_{q}[d] \ar^{\varphi(g)}[r] & 
\{0,1\}^\omega \ar^{q}[d] \\ 
X\ar_{g}[r] & X }
\]
commutes for all $g\in G$.
\end{definition}

Note that every compact metrizable space is a quotient of the Cantor set~$\{0,1\}^\omega$, so it makes sense to ask whether any action of a countable group on such a space is rational.  A group $G$ that acts faithfully and rationally on a compact metrizable space must embed into~$\R$. The converse holds as well, since any subgroup of $\R$ acts faithfully and rationally on the Cantor set.

Now, every hyperbolic group $G$ has a \newword{Gromov boundary} $\partial G$, which is a compact metrizable space (see \cite{Gromov1987}, and more generally the survey~\cite{KaBe}) on which $G$ acts by homeomorphisms.  Our main theorem is the following:

\begin{theorem}\label{thm:MainTheoremAction}For any hyperbolic group $G$, the action of $G$ on $\partial G$ is rational.
\end{theorem}

This theorem can be generalized as follows.

\begin{corollary}Let $X$ be a geodesic, hyperbolic metric space, and let $G$ be a group acting properly and cocompactly by isometries on~$X$.  Then the action of $G$ on $\partial X$ is rational.
\end{corollary}
\begin{proof}
By the \v{S}varc-Milnor lemma \cite[Proposition~I.8.19]{BrHa}, we know that~$X$ is quasi-isometric to $G$.  It follows that $G$ is hyperbolic, and there exists a $G$-equivariant homeomorphism from $\partial X$ to~$\partial G$ \cite[Propositions III.1.9 and III.1.10]{BrHa}.
\end{proof}

These statements can be viewed as assigning a certain kind of symbolic dynamics to the action of the group $G$ on~$\partial G$ (or~$\partial X$).  Specifically, the quotient map $q\colon \{0,1\}^\omega \to \partial G$ assigns a binary address to each point of~$\partial G$, and elements of $G$ act on $\partial G$ by asynchronous transducers.  Symbolic dynamics for actions of hyperbolic groups on their boundaries have been studied extensively (see~\cite{Coornaert-Papadopoulos-1}), but this particular assignment of binary addresses seems to be new, as is the action by asynchronous transducers.

It follows immediately from Theorem~\ref{thm:MainTheoremAction} that any hyperbolic group $G$ that acts faithfully on $\partial G$ embeds into~$\R$.  Unfortunately, it is possible for the action of $G$ on $\partial G$ to have nontrivial kernel, which is always a finite normal subgroup of~$G$ as long as $G$ is non-elementary (see Proposition~\ref{prop:KernelBoundaryAction} below).  However, as long as $G$ is nontrivial, it is easy to show that the free product $G*\mathbb{Z}$ is a non-elementary hyperbolic group with no finite normal subgroups.  It follows that $G*\mathbb{Z}$ embeds into~$\R$, and hence $G$ does as well, which proves Theorem~\ref{thm:MainTheorem} from Theorem~\ref{thm:MainTheoremAction}.

Our proof of Theorem~\ref{thm:MainTheoremAction} begins by defining a very broad class of trees which have a notion of rational homeomorphisms on their boundaries.  We refer to these as \newword{self-similar trees}, and we prove in Section~\ref{sec:SelfSimilarTrees} that the group of rational homeomorphisms of the boundary of a self-similar tree acts on the boundary in a rational way in the sense of Definition~\ref{def:RationalAction}.  This seems to be a very general tool for proving that actions are rational, and we hope that it will be helpful in other contexts.

Next we define a tree of subsets of any hyperbolic graph~$\Gamma$, which we refer to as \newword{atoms}.  Assuming a group $G$ acts properly and cocompactly on~$\Gamma$, we prove in Section~\ref{sec:HyperbolicRational} that this tree is self-similar and its boundary is naturally homeomorphic to the well-known \newword{horofunction boundary} (or~\newword{metric boundary}) $\hb\Gamma$ of~$\Gamma$. The horofunction boundary is compact and totally disconnected, and has the Gromov boundary $\partial\Gamma$ as a quotient~\cite{WeWi}.  Our construction computes the horofunction boundary of~$\Gamma$ explicitly, and describes the action of $G$ on~$\hb\Gamma$ by asynchronous transducers.  See Section~\ref{sec:AnExample} for an example of this construction.

Perhaps as evidence of the naturality or importance of the general construction, we learned from the authors that the article \cite{LacaRaeburnRamaggeWhittaker} gives a similar construction, with the goal of extending self-similar groups to act on the path space of a graph.  While the constructions given for the self-similar trees (in their language, self-similar groupoids) are similar, the automata groups arising in \cite{LacaRaeburnRamaggeWhittaker} are quite different in nature from ours (e.g., they are synchronous). 

The transducers that arise in our construction appear to have a special flavour: in all examples that we have computed, they act as prefix-exchange maps on a dense open subset of the boundary.  As such, the embeddings we construct are ``almost'' embeddings of hyperbolic groups into Thompson's group~$V$.  In \cite{LehnertSchweitzer}, Lehnert and Schweitzer use a push-down automaton that implements prefix exchanges on a finite set of test points to prove that all finitely-generated subgroups of Thompson's group~$V$ are in the class of CoCF groups introduced by Holt, R\"over, Rees, and Thomas~\cite{HRRT}. If this method can be extended, it may be possible to use our embedding to shed some light on the question of whether hyperbolic groups are~CoCF.

Finally, note that while synchronous automata groups are always residually finite, the same does not hold true for asynchronous automata groups.  In particular, our result does not yield any immediate information about the question of whether all hyperbolic groups are residually finite (see~\cite{GerstenProbs,Lysionok,NibloProbs,KapovichWise}).

\bigskip
\noindent \textbf{Acknowledgements.} We would like to thank 
Michael Whittaker for discussions where we learned that our papers have somewhat similar constructions of the self-similar tree and Volodymyr Nekrashevych 
for interesting discussions about our results in general.

\numberwithin{theorem}{section}
\section{Background}

\subsection{The Rational Group~\texorpdfstring{$\R$}{R}}
\label{sec:BackgroundRational}

In this section we briefly recall the definitions of transducers and rational homeomorphisms from~\cite{GNS}.  We have modified some of the definitions slightly to simplify the terminology.

Throughout this paper, if $S$ is a set, we let $S^\omega$ denote the set of all infinite sequences of elements of~$S$, and we let $S^{\tast}$ denote the set of all finite sequences of elements of~$S$, including the empty sequence~$\emptystring$.

\begin{definition}A \newword{transducer} consists of the following data:
\begin{enumerate}
\item Two finite sets $\Ain$ and $\Aout$ called the \newword{input alphabet} and \newword{output alphabet},
\item A finite set $Q$ whose elements are called \newword{states},
\item An \newword{initial state} $q_0\in Q$,
\item A \newword{transition function} $t\colon Q \times \Ain \to Q$, and
\item An \newword{output function} $o\colon Q\times \Ain \to \Aout^\tast$.
\end{enumerate}
\end{definition}

A transducer is \newword{synchronous} if $o(q,a)$ is a single symbol in $\Aout$ for each $q\in Q$ and $a\in\Ain$, and \newword{asynchronous} otherwise.  We allow both synchronous and asynchronous transducers.

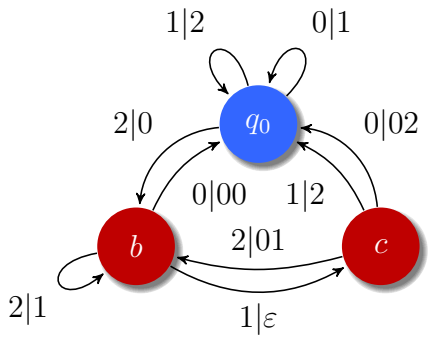
\begin{figure}
\centering
 \begin{tikzpicture}[->,>=stealth',shorten >=1pt,auto,node distance=2.3cm,on grid,semithick,
                    every state/.style={fill=myred,draw=none,circular drop shadow,text=white}]
   \node[state,fill=myblue] (a)   {$q_0$}; 
   \node[state] (b) [below left=of a] {$b$}; 
   \node[state] (c) [below right=of a] {$c$}; 
    \path[->] 
    (a) edge [out=105,in=135,loop] node [swap]{$1|2$} ()
        edge [out=45,in=75,loop] node [swap]{$0|1$} ()
        edge  [out=185,in=85]node [swap]{$2|0$} (b)
    (b) edge [out=190,in=220, loop] node [swap] {$2|1$} ()
        edge[out=65,in=205]  node  [swap]{$0|00$} (a)
        edge [out=330,in=210] node [swap]{$1|\varepsilon$} (c)
    (c) edge[out=95,in=355]  node [swap]{$0|02$} (a)
        edge [out=115,in=335]  node {$1|2$} (a) 
        edge [in=345,out=195] node [swap] {$2|01$} (b);
\end{tikzpicture}
\caption{An asynchronous transducer on a three-letter alphabet $\{0,1,2\}$ with initial state $q_0$.\label{fig-small-trans}}
\end{figure}
We can draw a transducer as a finite directed graph, as shown in Figure~\ref{fig-small-trans}. Each state of this transducer is a node of the graph, and the directed edges indicate the transitions and output.  Specifically, for each $q\in Q$ and $a\in\Ain$, there is a directed edge in from $q$ to $t(q,a)$ in the graph with label $a\mid o(q,a)$.

If $T=(\Ain,\Aout,Q,q_0,t,o)$ is a transducer, an \newword{input string} for $T$ is any infinite string $a_1a_2\cdots\in\Ain^\omega$.  The corresponding \newword{output string} is the concatenation
\[
o(q_0,a_1)\,o(q_1,a_2)\,o(q_2,a_3)\cdots
\]
where $\{q_n\}$ is the sequence of states starting at the initial state $q_0$ defined recursively by $q_n = t(q_{n-1},a_n)$.

Note that the output string may be finite if $o(q_{n-1},a_n) = \emptystring$ for all but finitely many~$n$, but we are interested in transducers whose output strings are always infinite.  Such transducers are called \newword{nondegenerate}.  A nondegenerate transducer defines a function $\Ain^\omega\to\Aout^\omega$ from infinite input strings to infinite output strings.

\begin{definition}\label{def:RationalGNS}Let $\Ain$ and $\Aout$ be finite sets.  We say that a function $f\colon\Ain^\omega\to\Aout^\omega$ is \newword{rational} if there exists a nondegenerate transducer with input alphabet $\Ain$ and output alphabet $\Aout$ whose output string is $f(\psi)$ for each input string $\psi\in\Ain^\omega$.
\end{definition}

The following properties of rational functions are proven in~\cite{GNS}.  We will prove (2) and (3) in a more general setting in Section~\ref{sec:RationalFunctions}.

\begin{proposition}\quad\begin{enumerate}
\item Any rational function $f\colon A^\omega\to B^\omega$ is continuous with respect to the product topologies on $A^\omega$ and $B^\omega$.
\item If $f\colon A^\omega\to B^\omega$ and $g\colon B^\omega\to C^\omega$ are rational, then so is the composition~$g\circ f$.
\item If $f\colon A^\omega\to B^\omega$ is a rational bijection, then the inverse $f^{-1}\colon B^\omega\to A^\omega$ is rational.\hfill\qedsymbol
\end{enumerate}
\end{proposition}

\begin{definition}If $A$ is a finite set with at least two elements, the \newword{rational group} $\R_A$ is the group of all rational homeomorphisms of~$A^\omega$.
\end{definition}

In particular, the \newword{binary rational group} $\R_2$ is the group of all rational homeomorphisms of the Cantor set~$\{0,1\}^\omega$.

\begin{proposition}\label{prop:AlphabetDoesntMatter}For any two finite sets $A$ and $B$ with at least two elements, there exists a rational homeomorphism $A^\omega\to B^\omega$, and therefore the rational groups $\R_A$ and $\R_B$ are isomorphic.
\end{proposition}
\begin{proof}See \cite[Corollary~2.12]{GNS}.
\end{proof}

Thus up to isomorphism there is only one rational group~$\R$, whose simplest form is the binary rational group~$\R_2$.  Other rational groups $\R_A$ are just other manifestations of this group.  We will henceforth use the notation $\R$ for the rational group in cases where the alphabet is unimportant.

\begin{note}In fact, Proposition~\ref{prop:AlphabetDoesntMatter} tells us that the rational groups $\R_A$ and~$\R_B$ corresponding to two alphabets $A$ and~$B$ are actually \newword{conjugate}, in the sense that the action of $\R_A$ on $A^\omega$ is conjugate to the action of $\R_B$ on $B^\omega$ by a homeomorphism~$A^\omega\to B^\omega$.  Being conjugate is the natural geometric notion of equivalence for groups of homeomorphisms, and is stronger than saying that the two groups are algebraically isomorphic.

It follows from this conjugacy that the definition of a rational action given in Definition~\ref{def:RationalAction} does not depend on the alphabet.  That is, if $G$ is any group acting on a compact metrizable space~$X$ and there exists a finite alphabet~$A$, a quotient map $q\colon A^\omega\to X$, and a homomorphism $\varphi\colon G\to \R_A$ such that $q\circ \varphi(g) = g\circ q$ for all $g\in G$, then the action of $G$ on $X$ is rational.
\end{note}

When working with a rational group $\R_A$, it often helps to consider the infinite directed tree $A^\tast$ of all finite strings over~$A$.  The root of $A^\tast$ is the empty string~$\emptystring$, and there is an edge from a string $w_1$ to another string $w_2$ whenever $w_2=w_1 a$ for some letter $a\in A$. The Gromov boundary $\partial A^\tast$ of $A^\tast$ is naturally homeomorphic to~$A^\omega$.

If $\alpha\in A^\tast$ is a finite string, we will let $A^\tast_\alpha$ denote the rooted subtree of~$A^\tast$ with root~$\alpha$, i.e.~the set of all finite strings that have $\alpha$ as a prefix. The boundary $\partial A^\tast_\alpha$ is naturally a subset of~$A^\omega$, consisting of all infinite strings in $A^\omega$ that have $\alpha$ as a prefix.

If $S\subseteq A^\omega$ is nonempty, the \newword{greatest common prefix} of $S$ is the longest string $\alpha$ that is a prefix of all strings in~$S$. If $S$ has at least two points then $\alpha$ must be a finite string. In this case, $\alpha$~is the deepest vertex (i.e.~farthest vertex from the root) in $A^\tast$ with the property that $S\subseteq \partial A^\tast_\alpha$.

\begin{definition}\label{def:Restriction}Let $f\colon A^\omega \to B^\omega$ be a rational map. Let $\alpha\in A^\tast$, and suppose that $f(\partial A^\tast_\alpha)$ has at least two points.  Then the \newword{restriction} of $f$ to~$\alpha$ is the function $f|_\alpha\colon A^\omega\to B^\omega$ defined by
\[
f(\alpha\psi) = \beta\,f|_\alpha(\psi)
\]
for all $\psi\in A^\omega$, where $\beta$ denotes the greatest common prefix of~$f(\partial A^\tast_\alpha)$.
\end{definition}

There is a useful characterization of rational functions based on their restrictions.

\begin{theorem}\label{thm:RestrictionTest}Let $A$ and $B$ be finite sets and let $f\colon A^\omega\to B^\omega$ be a continuous function.  Then $f$ is rational if and only if the following conditions are satisfied:
\begin{enumerate}
\item $f$ has only finitely many different restrictions.
\item For each $\alpha\in A^\tast$ such that $f(\partial A^\tast_\alpha)$ is a one-point set~$\{\psi\}$, the string $\psi \in B^\omega$ is eventually periodic.
\end{enumerate}
\end{theorem}
\begin{proof}See \cite[Theorem~2.5]{GNS}.
\end{proof}

We will use this theorem in Section~\ref{sec:RationalFunctions} to generalize the notion of rational functions to the boundary of arbitrary self-similar trees.

\subsection{Hyperbolic Groups}

In this section we briefly recall relevant facts about hyperbolic graphs and hyperbolic groups.

If $\Gamma$ is a connected graph, a \newword{path} in $\Gamma$ is a sequence $v_0,v_1,\ldots,v_n$ of vertices such that $v_{i-1}$ and $v_i$ are connected by an edge for all $1\leq i\leq n$.  The number $n$ is called the \newword{length} of the path.  A path between two vertices $v$ and $w$ is called a \newword{geodesic} if it has the minimum possible length, and the length of such a path is the \newword{distance} between $v$ and $w$, denoted~$d(v,w)$.  This notion of distance defines a metric on the vertex set of~$\Gamma$, sometimes called the \newword{path metric}.

Throughout this paper, we will regard a graph $\Gamma$ as being the same as its vertex set endowed with the path metric. In particular, we will write $v\in\Gamma$ to mean that $v$ is a vertex of~$\Gamma$.  Note that two such graphs are isomorphic if and only if they are isometric.

For the following definition, a \newword{geodesic triangle} in $\Gamma$ with vertices $v_1,v_2,v_3$ is a triple $\bigl([v_1,v_2],[v_1,v_3],[v_2,v_3]\bigr)$, where each $[v_i,v_j]$ is a geodesic from $v_i$ to~$v_j$.

\begin{definition}Let $\delta\geq 0$.  A connected graph $\Gamma$ is \newword{$\boldsymbol{\delta}$-hyperbolic} if for every geodesic triangle
$\bigl([a,b],[a,c],[b,c]\bigr)$ in~$\Gamma$ and every vertex $v\in[a,b]$, there exists a vertex $w\in[a,c]\cup [b,c]$ so that $d(p,q)\leq \delta$.
\end{definition}

We say that a graph $\Gamma$ is \newword{hyperbolic} if it is $\delta$-hyperbolic for some $\delta\geq 0$.  This definition is due to Gromov~\cite{Gromov1987}, and can be generalized to arbitrary metric spaces.  See~\cite{BrHa} for a general introduction.

There is a natural notion of boundary for a hyperbolic graph, also introduced by Gromov.  If $\Gamma$ is a connected graph, a \newword{geodesic ray} in $\Gamma$ is a sequence $\{v_n\}_{n\geq 0}$ of vertices such that each initial subpath $v_0,\ldots,v_n$ is a geodesic.  Two geodesic rays $R =\{v_n\}$ and $R' = \{v_n'\}$ are said to \newword{fellow travel} if the sequence $\{d(v_n,v_n')\}$ of distances is bounded.  This is clearly an equivalence relation on geodesic rays, and we denote the equivalence class of a geodesic ray $R$ by~$[R]$.

\begin{definition}The \newword{Gromov boundary} of $\Gamma$ is the set
\[
\partial\Gamma  = \{[R] \mid R\text{ is a geodesic ray in }\Gamma\}.
\]
\end{definition}

There is a natural topology on $\partial\Gamma$ which gives it the structure of a compact metrizable space (see~\cite{BrHa}).

\begin{example}Any tree $T$ is $0$-hyperbolic, for if $\bigl([a,b],[a,c],[b,c]\bigr)$ is a geodesic triangle in~$T$ then $[a,b]\subseteq [a,c]\cup [b,c]$.

If we fix a point $p\in T$, then every infinite path of distinct vertices starting at $p$ is a geodesic ray.  No two such rays fellow travel, and the Gromov boundary~$\partial T$ can be identified with the set of all such rays.\qed
\end{example}

\begin{example}\label{ex:SquareTiling}Any graph that is quasi-isometric to the hyperbolic plane $\mathbb{H}^2$ is hyperbolic.  For example, the $1$-skeleton of any tiling of $\mathbb{H}^2$ by congruent polygons is hyperbolic.  The Gromov boundary of such a graph is homeomorphic to a circle, namely the circle of boundary points for the hyperbolic plane.

One such graph is shown in Figure~\ref{fig:SquareTilingVertices}.  This is the $1$-skeleton of the \newword{order five square tiling} of the hyperbolic plane, i.e.~the tiling by congruent regular quadrilaterals (hyperbolic squares) with five squares meeting at every vertex.  We will be using this hyperbolic graph throughout this paper to illustrate our definitions and methods.\qed
\begin{figure}
\centering
\includegraphics{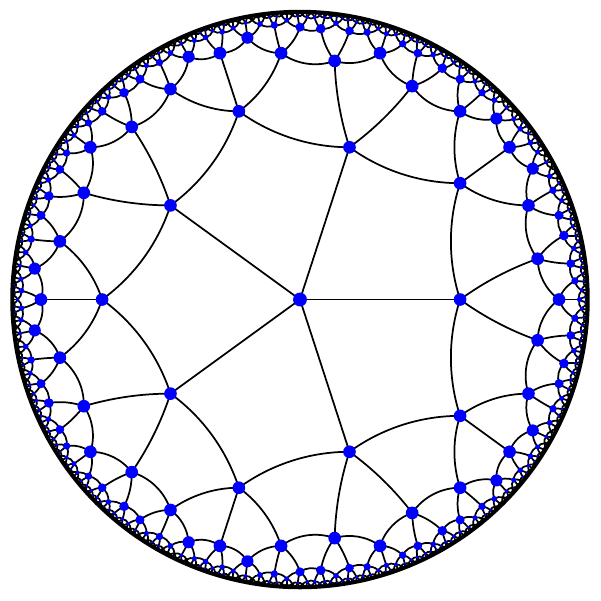}
\caption{The $1$-skeleton of the order five square tiling of the hyperbolic plane.}
\label{fig:SquareTilingVertices}
\end{figure}
\end{example}

An isomorphism between two hyperbolic graphs induces a homeomorphism between their boundaries, and in particular any automorphism of a hyperbolic graph $\Gamma$ induces a self-homeomorphism of~$\partial\Gamma$.  Thus any action of a group~$G$ on~$\Gamma$ by isometries induces an action of~$G$ on~$\partial\Gamma$ by homeomorphisms.

More generally, if $\Gamma$ and $\Gamma'$ are quasi-isometric graphs, then $\Gamma$ is hyperbolic if and only if $\Gamma'$ is hyperbolic, in which case $\partial\Gamma$ is homeomorphic to~$\partial\Gamma'$.

\begin{definition}A \newword{hyperbolic group} is a finitely generated group whose Cayley graph is hyperbolic.
\end{definition}

Of course, a finitely-generated group has many possible Cayley graphs, corresponding to the different finite generating sets.  However, any two such Cayley graphs are quasi-isometric, and hence they are all hyperbolic if any one of them is hyperbolic.  Moreover, the homeomorphism type of the boundary does not depend on the generating set, so it makes sense to talk about \textit{the} Gromov boundary $\partial G$ of a hyperbolic group as a compact metrizable space.

\begin{example}Any finite group or more generally any virtually cyclic group is hyperbolic.  These are the \newword{elementary hyperbolic groups}.\qed
\end{example}

\begin{example}Any finitely generated free group is hyperbolic, since the corresponding Cayley graph is a tree.  More generally, any virtually free group, such as any free product of finite groups, is hyperbolic.\qed
\end{example}

\begin{example}A \newword{hyperbolic surface group} is the fundamental group of a closed surface with Euler characteristic~$\chi<0$.  Such a group is hyperbolic, since its Cayley graph can be viewed as the $1$-skeleton of a tiling of the hyperbolic plane $\mathbb{H}^2$ by congruent $(4-2\chi)$-gons.\qed
\end{example}

\begin{example}Let $\Gamma$ be a hyperbolic graph, and let $G$ be a group acting on $\Gamma$ by isometries.  We say that the action of $G$ on $\Gamma$ is \newword{proper} if every vertex in~$\Gamma$ has finite stabilizer, and \newword{cocompact} if there are finitely many orbits of vertices in~$\Gamma$. If the action of $G$ is proper and cocompact, then the Cayley graph of~$G$ must be quasi-isometric to~$\Gamma$, and therefore $G$ is hyperbolic.  For example, the isometry group of the \mbox{$1$-skeleton} of the order five square tiling of the hyperbolic plane is hyperbolic (see~Example~\ref{ex:SquareTiling}).\qed
\end{example}

We do need a few nontrivial facts from the theory of hyperbolic groups.  The first is the following proposition, which we used in the introduction to prove Theorem~\ref{thm:MainTheorem} from Theorem~\ref{thm:MainTheoremAction}.

\begin{proposition}\label{prop:KernelBoundaryAction}Let $G$ be a non-elementary hyperbolic group.  Then the kernel of the action of $G$ on $\partial G$ is a finite normal subgroup of~$G$.
\end{proposition}
\begin{proof}We will use results from~\cite{KaBe}.  Since $G$ is non-elementary, Theorem~2.28 of \cite{KaBe} ensures that the Gromov boundary $\partial G$ is infinite. By Proposition~4.2(1) of~\cite{KaBe} we have that any element of infinite order has exactly two fixed points in $\partial G$, so no infinite order element is contained in the kernel $K$ of the action of $G$ on $\partial G$. Hence, $K$ is a torsion subgroup of $G$ and, being normal, it is the disjoint union of the conjugacy classes of all of its elements.  By Proposition 6.3(6) in \cite{KaBe}, the number of conjugacy classes of finite order elements is finite and therefore $K$ itself is finite. 
\end{proof}

Next, we will need the fact that hyperbolic groups have finitely many cone types. Recall the following definition.

\begin{definition}Let $\Gamma$ be a graph, let $G$ be a group acting by isometries on~$\Gamma$, and fix a vertex $x_0\in\Gamma$.  For each $x\in \Gamma$, the \newword{cone} on $x$ is the set
\[
C(x) \,=\, \{y\in \Gamma \mid d(x_0,y) = d(x_0,x) + d(x,y)\}.
\]
Two points $x,x'\in\Gamma$ have the \newword{same cone type} if there exists a $g\in G$ so that $gx=x'$ and $g\,C(x) = C(x')$.
\end{definition}

Note that a point $y$ lies in a cone $C(x)$ if and only if there exists a geodesic from $x_0$ to $y$ that goes through~$x$.  Figure~\ref{fig:ConeTypes} shows the cone types for the \mbox{$1$-skeleton} of the order five square tiling of the hyperbolic plane, where $G$ is the full isometry group.
\begin{figure}
\centering
$\underset{\textstyle\text{(a)}}{\includegraphics{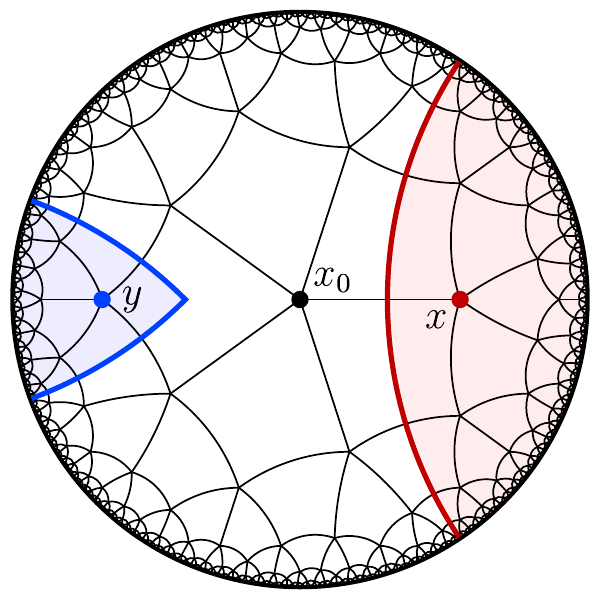}}$\hfill
$\underset{\textstyle\text{(b)}}{\includegraphics{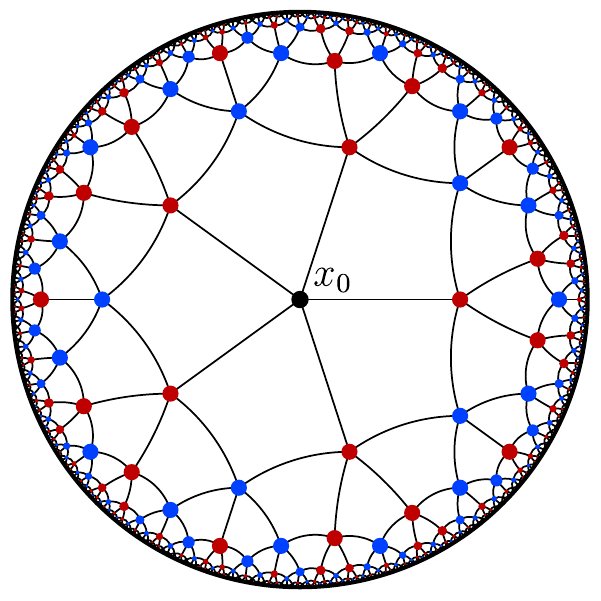}}$
\caption{(a)~Cones $C(x)$ and $C(y)$ for two points $x$ and $y$ in a hyperbolic graph~$\Gamma$, where the cone consists of all vertices in the shaded region.  (b)~Cone types in $\Gamma$ if $G$ is the full isometry group.  Red vertices have the same cone type as~$x$ and blue vertices have the same cone type as~$y$.}
\label{fig:ConeTypes}
\end{figure}

The following theorem is due to Cannon~\cite{Cannon1} (see \cite[Theorem~2.18]{BrHa}).

\begin{theorem}[Cannon]\label{thm:FinitelyManyConeTypes} Let\/ $\Gamma$ be a hyperbolic graph, let $x_0\in\Gamma$, and let~$G$ be a group acting properly and cocompactly by isometries on\/~$\Gamma$.  Then\/~$\Gamma$ has only finitely many cone types with respect to $G$ and~$x_0$.\qed
\end{theorem}

\subsection{The Horofunction Boundary}

Here we review the basic definition and properties of the horofunction boundary (or metric boundary) of a locally finite, connected graph.  The horofunction boundary can actually be defined for any complete metric space---see \cite{BrHa} for a general introduction.

Let $\Gamma$ be a locally finite, connected graph.  As before, we put the path metric on $\Gamma$, and we identify $\Gamma$ with is vertex set. Let $F(\Gamma,\mathbb{Z})$ be the abelian group of all integer-valued functions on~$\Gamma$, and let $\F(\Gamma,\mathbb{Z})$ be the quotient of $F(\Gamma,\mathbb{Z})$ by the subgroup of constant functions. That is, two functions $f,g\in F(\Gamma,\mathbb{Z})$ are identified in $\F(\Gamma,\mathbb{Z})$ if $f-g$ is a constant function. 

\begin{notation}
If $f \in F(\Gamma,\mathbb{Z})$, we will let $\of$ denote the corresponding element of~$\F(\Gamma,\mathbb{Z})$.
\end{notation}

Note that $F(\Gamma,\mathbb{Z}) = \mathbb{Z}^\Gamma$ forms a topological space under the product topology, from which $\F(\Gamma,\mathbb{Z})$ inherits a quotient topology.

\begin{definition}\label{def:CanonicalEmbedding}\quad
\begin{enumerate}
\item If $x\in \Gamma$, the associated \newword{distance function} is the function $d_x\colon \Gamma\to \mathbb{Z}$  defined by
\[
d_x(y) = d(x,y)
\]
for all $y\in\Gamma$.
\item The \newword{canonical embedding} $i\colon \Gamma\to \F(\Gamma,\mathbb{Z})$ is the map defined by
\[
i(x) = \od_x
\]
for all $x\in\Gamma$.
\end{enumerate}
\end{definition}

Note that each $\od_x$ is an isolated point in~$i(\Gamma)$, since $\od_x$ has a global minimum at~$x$.  Thus $i$ really is an embedding.

\begin{definition}\label{def:HorofunctionBoundary}\quad
\begin{enumerate}
\item The \newword{horofunction boundary} of $\Gamma$, denoted $\hb\Gamma$, is the set of limit points of $i(\Gamma)$ in $\F(\Gamma,\mathbb{Z})$.
\item A function $f\colon \Gamma\to\mathbb{Z}$ is called a \newword{horofunction} if $\of\in\hb \Gamma$.
\end{enumerate}
\end{definition}

Note that if $f$ is a horofunction then so is $f+C$ for any constant~$C\in\mathbb{Z}$, and these correspond to the same point in the horofunction boundary.

Note also that $\hb \Gamma$ is empty when $\Gamma$ is finite, since $i(\Gamma)$ cannot have any limit points.  Thus a finite graph has no horofunctions.

\begin{example}If $\Gamma$ is an infinite path with vertex set $\mathbb{N}$, then $\hb \Gamma$ is a single point.  In particular, the only horofunctions on $\Gamma$ are
\[
f(n) = -n + C
\]
for $C\in\mathbb{Z}$ a constant.  Note that $\od_m\to \of$ in $\F(\Gamma,\mathbb{Z})$ as $m\to\infty$.  In particular,
\[
d_m(n) = |m-n| = \begin{cases}-n + m & \text{if }n\leq m, \\ n-m & \text{if }n > m,\end{cases}
\]
for all $m\in\mathbb{N}$, so $\od_m$ agrees with $\of$ on $\mathbb{N}\cap [0,m]$.\qed
\end{example}

\begin{example}If $\Gamma$ is a bi-infinite path with vertex set $\mathbb{Z}$, then $\hb\Gamma$ has two points, corresponding to the horofunctions
\[
f_{-\infty}(n) = n + C\qquad\text{and}\qquad f_\infty(n) = -n+C.
\]
Again $\od_m\to \of_{\mkern-3mu \infty}$ as $m\to\infty$, and $\od_m\to \of_{\mkern-3mu-\infty}$ as~$m\to-\infty$.\qed
\end{example}

\begin{example}Let $\Gamma$ be the infinite square grid in the plane, with vertex set~$\mathbb{Z}^2$.  Then $\hb\Gamma$ is homeomorphic to the union
\[
\bigl(\mathbb{Z}\times \{\pm\infty\}\bigr) \cup \bigl(\{\pm\infty\} \times \mathbb{Z}\bigr) \cup \bigl(\{\pm\infty\}\times\{\pm\infty\}\bigr)
\]
with the obvious topology~\cite{Develin}.  For example, the horofunction
\[
f(x,y) = -x-y+C
\]
corresponds to the point $(\infty,\infty)$, and the horofunction
\[
f(x,y) = |x-5| - y + C
\]
corresponds to the point $(5,\infty)$.\qed
\end{example}

\begin{proposition}\label{prop:HorofunctionTest}Let $f\colon \Gamma\to\mathbb{Z}$.  Then the following are equivalent:
\begin{enumerate}
\item $f$ is a horofunction on\/~$\Gamma$.
\item For every finite set $B\subseteq\Gamma$ there exist infinitely many $x\in \Gamma\setminus B$ for which~$\od_x$ agrees with $\of$ on~$B$.
\end{enumerate}
\end{proposition}
\begin{proof}For any finite set $B\subseteq\Gamma$, let
\[
U_B = \{\og\in \F(\Gamma,\Z) \mid \og\text{ agrees with }\of\text{ on }B\}.
\]
It is easy to check that the sets $U_B$ form a neighborhood base for $\F(\Gamma,\Z)$ at~$\of$, so $f$ satisfies condition (2) if and only if $\of$ is a limit point of~$i(\Gamma)$.\end{proof}

\begin{proposition}If\/ $\Gamma$ is a locally finite, connected graph, then $\hb\Gamma$ is compact and totally disconnected.
\end{proposition}
\begin{proof}\newcommand{\dpx}{{{d_x}\hspace{0.08em}\!\!\!'\,\hspace{0.02em}}}%
\newcommand{\odpx}{{\od_x}\hspace{0.08em}\!\!\!'\,\hspace{0.02em}}%
Fix a point $x_0\in\Gamma$.  For each $x\in\Gamma$, let $\dpx\colon \Gamma\to\Z$ be the function
\[
\dpx(y) = d_x(y) - d_x(x_0).
\]
Note then that $\dpx(x_0) = 0$ and $\odpx = \od_x$.  By the triangle inequality, we know that
\[
|\dpx(y)| \leq d(x_0,y)
\]
for all $y\in\Gamma$, so each $\dpx$ lies in the infinite product
\[
S = \prod_{y\in\Gamma} \bigl(\Z\cap [-d(x_0,y),d(x_0,y)]\bigr).
\]
This is a product of finite sets, which means that $S$ is totally disconnected and compact.  Since $f(x_0) = 0$ for all $f\in S$, the quotient map $F(\Gamma,\Z)\to\F(\Gamma,\Z)$ is one-to-one on~$S$, so the image $\oS$ of $S$ in $\F(\Gamma,\Z)$ is homeomorphic to~$S$.  But $i(\Gamma)\subseteq \S$, so $\hb \Gamma\subseteq \oS$ since $\oS$ is closed, and the result follows.
\end{proof}

Now suppose that $G$ is a group acting by isometries on~$\Gamma$.  There is natural left action of $G$ on $F(\Gamma,\mathbb{Z})$ defined by
\[
(gf)(p) = f\bigl(g^{-1}p\bigr)
\]
for all $g\in G$, $f\in F(\Gamma,\mathbb{Z})$, and $p\in \Gamma$, and this descends to a left action of~$G$ on $\F(\Gamma,\mathbb{Z})$.  It is easy to check that the canonical embedding $i\colon\Gamma\to \F(\Gamma,\mathbb{Z})$ is equivariant with respect to this action.  In particular, $gf$~is a horofunction on~$\Gamma$ for any $g\in G$ and any horofunction $f\in F(\Gamma,\mathbb{Z})$, and this gives us a left action of $G$ on~$\hb\Gamma$.

\begin{theorem}[Webster and Winchester]Let\/ $\Gamma$ be a $\delta$-hyperbolic graph with Gromov boundary $\partial\Gamma$ and horofunction boundary $\hb\Gamma$, and let $G$ be a group acting by isometries on\/~$\Gamma$.  Then there exists a $G$-equivariant quotient map $q\colon \hb\Gamma\to \partial\Gamma$.\hfill\qedsymbol
\end{theorem}
\begin{proof}See \cite{WeWi}.  Note that $\partial\Gamma$ and $\hb\Gamma$ are both compact Hausdorff spaces in this case, so the surjection $\hb\Gamma\to\partial\Gamma$ defined in~\cite{WeWi} is indeed a quotient map.
\end{proof}

Thus, if we wish to show that $G$ acts rationally on $\partial\Gamma$, it suffices to show that $G$ acts rationally on~$\hb\Gamma$.  That is, if we wish to prove Theorem~\ref{thm:MainTheoremAction} (and hence Theorem~\ref{thm:MainTheorem}), it suffices to prove the following:

\begin{theorem}\label{thm:MainTechnicalTheorem}Let\/ $\Gamma$ be a $\delta$-hyperbolic graph with horofunction boundary~$\hb\Gamma$, and let $G$ be a group acting properly and cocompactly by isometries on\/~$\Gamma$.  Then the induced action of $G$ on $\hb\Gamma$ is rational.
\end{theorem}

Sections~\ref{sec:SelfSimilarTrees} and~\ref{sec:HyperbolicRational} are devoted to a proof of this theorem.

\section{Rational Groups for Self-Similar Trees}
\label{sec:SelfSimilarTrees}

As we have seen, there is one version $\R_A$ of the rational group $\R$ for each finite set~$A$.  In this section, we generalize the definition of $\R_A$ to allow for rational homeomorphisms between the boundaries of arbitrary self-similar trees.  We will construct such a tree in Section~\ref{sec:HyperbolicRational} for an arbitrary hyperbolic group.

As defined below, self-similar trees do not necessarily have canonical isomorphisms between subtrees, and therefore infinite descending paths in self-similar trees do not correspond to infinite strings of symbols in a natural way.  As a result, the theory of rational homeomorphisms defined by transducers does not directly apply to self-similar trees.  In this section, we develop the theory of rational homeomorphisms of the boundaries of self-similar trees using a generalization of the notion of having finitely many restrictions, and then prove that the corresponding rational groups embed in~$\R$.

Readers interested primarily in hyperbolic groups may want to skip over some of the technical development in this section.  For such readers, we recommend reading the definition of a self-similar tree (Definition~\ref{def:SelfSimilarTree}), the definition of a rational homeomorphism (Definition~\ref{def:RationalHomeomorphism}), and our most general result on rational actions (Corollary~\ref{cor:MainCorollaryIGuess}) before continuing to Section~\ref{sec:HyperbolicRational}.

We will use the following notation and terminology for trees throughout this section:
\begin{itemize}
\item We will identify each tree $T$ with its vertex set.  In particular, the notation $v\in T$ will mean that $v$ is a vertex of~$T$.
\item The \newword{depth} of a vertex $v\in T$, denoted $|v|$, is the distance from $v$ to the root of $T$.
\item If $v\in T$, we let $T_v$ denote the subtree of $T$ consisting of $v$ and all of its descendants.
\item If $T$ is a locally finite rooted tree, the \newword{boundary} of $T$ is the space $\partial T$ of all infinite descending paths in $T$ starting with the root.
\item If $v\in T$, we will think of $\partial T_v$ as a subset of $\partial T$, namely the set of all infinite descending paths that go through~$v$.  Such subsets are clopen in~$\partial T$, and form a basis for the topology on~$\partial T$.
\item If $S\subseteq \partial T$ has at least two points, the \newword{deepest parent} of $S$ is the deepest vertex~$v$ for which~$S\subseteq T_v$.
\item The \newword{standard ultrametric} on $\partial T$ is the metric $d\colon \partial T\times \partial T\to\mathbb{R}$ defined by
\[
d(p,q) = 2^{-|v|}
\]
for $p\ne q$, where $v$ is the deepest parent of $\{p,q\}$.
\item If $\varphi\colon T_v\to T_w$ is a rooted isomorphism between subtrees of $T$, we will let $\varphi_*$ denote the induced homeomorphism $\partial T_v\to\partial T_w$.  Note that~$\varphi_*$ is a similarity transformation with respect to the standard ultrametric, with
\[
d\bigl(\varphi_*(p),\varphi_*(q)\bigr) = 2^{|v|-|w|}d(p,q)
\]
for all $p,q\in \partial T_v$.
\end{itemize}

\subsection{Self-Similar Trees}

\begin{definition}\label{def:SelfSimilarTree}Let $T$ be a locally finite rooted tree.  A \newword{self-similar structure} on $T$ consists of the following data:
\begin{enumerate}
\item A partition of the vertices of $T$ into finitely many \newword{types}.
\item For every pair $u,v$ of vertices of $T$ of the same type, a nonempty, finite set $\mathrm{Mor}(u,v)$ of (rooted) tree isomorphisms $T_u\to T_v$, each of which maps vertices of $T_u$ to vertices of $T_v$ of the same type.
\end{enumerate}
Elements of $\mathrm{Mor}(u,v)$ are called \newword{morphisms}, and are required to satisfy the following conditions:
\begin{enumerate}
\item[(a)] If $\varphi\in \mathrm{Mor}(u,v)$, then $\varphi^{-1}\in \mathrm{Mor}(v,u)$.
\item[(b)] If $\varphi\in \mathrm{Mor}(u,v)$ and $\psi\in\mathrm{Mor}(v,w)$, then $\psi\varphi\in \mathrm{Mor}(u,w)$.
\item[(c)] If $\varphi \in \mathrm{Mor}(v,w)$ and $u\in T_v$, then the restriction $\varphi|_{T_u}\colon T_u\to T_{\varphi(u)}$ is in $\mathrm{Mor}\bigl(u,\varphi(u)\bigr)$.
\end{enumerate}
A \newword{self-similar tree} is a locally finite rooted tree $T$ together with a self-similar structure on~$T$.
\end{definition}

\begin{note}\label{note:PartialComposition}We will have some use for partial compositions of morphisms.  If $\varphi\in\mathrm{Mor}(v,w)$, $\psi\in\mathrm{Mor}(x,y)$, and $T_w\subseteq T_x$, we will let $\psi\varphi$ denote the composition
\[
T_v \xrightarrow{\varphi} T_w \xrightarrow{\psi'} T_{\psi(w)}
\]
where $\psi'$ is the restriction of~$\psi$.  Similarly, if $\varphi\in\mathrm{Mor}(v,w)$, $\psi\in\mathrm{Mor}(x,y)$, and $T_x\subseteq T_w$, we will let $\psi\varphi$ denote the composition
\[
T_{\varphi^{-1}(x)} \xrightarrow{\varphi'} T_x \xrightarrow{\psi} T_y
\]
where $\varphi'$ is the restriction of~$\varphi$.

Using this notion of composition, conditions (a) through (c) can be summarized by saying that the collection of all morphisms forms an inverse semigroup of isomorphisms between subtrees of~$T$.\qed
\end{note}

\begin{example}\label{ex:BasicTree}If $A$ is a finite alphabet, then we can give $A^\tast$ the structure of a self-similar tree as follows:
\begin{enumerate}
\item There is only one type of vertex in~$A^\tast$.
\item For each $\alpha,\beta\in A^\tast$, there is only one morphism $\varphi_{\alpha\beta}\colon A^\tast_\alpha\to A^\tast_\beta$, namely the \newword{prefix replacement}
\[
\varphi_{\alpha\beta}(\alpha \gamma) = \beta\gamma\qquad(\gamma\in A^\omega).
\]
\end{enumerate}
It is easy to check that this satisfies the required axioms.\qed
\end{example}

In general, we say that a self-similar structure on a tree $T$ is \newword{rigid} if there is at most one morphism between any two vertices of~$T$.  A self-similar tree whose self-similar structure is rigid is said to be a \newword{rigid tree}.

The previous example put a rigid self-similar structure on $A^\tast$, but other self-similar structures are possible.

\begin{example}Let $A$ be a finite alphabet, and let $G$ be a group of automorphisms of $A^\tast$.  Following Nekrashevych~\cite{Nek}, we say that $G$ is \newword{self-similar} if it is closed under restrictions, i.e.~for each $g\in G$ and $\alpha\in A^\tast$ the automorphism $g|_\alpha \colon A^\tast\to A^\tast$ defined by
\[
g(\alpha\gamma) = g(\alpha)\, g|_{\alpha}(\gamma)\qquad(\gamma\in A^\tast)
\]
is again an element of~$G$.

Now suppose that $G$ is a \textit{finite} self-similar subgroup of~$\mathrm{Aut}(A^\tast)$.  For example, if $A$ has $n$ elements, then the symmetric group $S_n$ acts on $A^\tast$ by permuting symbols, and the image of $S_n$ in $\mathrm{Aut}(A^\tast)$ is a finite self-similar group.  Then $G$ induces a self-similar structure on $A^\tast$ as follows:
\begin{enumerate}
\item There is only one type of vertex in~$A^\tast$.
\item For each $g\in G$ and $\alpha,\beta\in A^\tast$, there is a morphism $g_{\alpha\beta}\colon A^\tast_\alpha \to A^\tast_\beta$ defined by
\[
g_{\alpha\beta}(\alpha \gamma) = \beta\,g(\gamma) \qquad (\gamma\in A^\tast).
\]
\end{enumerate}
It is easy to check that this satisfies the required axioms.\qed
\end{example}

Of course, there are also self-similar trees with more than one type of vertex.

\begin{example}Let $\Gamma=(V,E)$ be a finite directed multigraph with edge set~$E$, and fix an \newword{initial vertex} $v_0\in V$.  The corresponding \newword{path language} $\mathcal{L}(\Gamma,v_0) \subseteq E^\tast$ is the set of all finite directed paths $e_1e_2\cdots e_n$ in~$\Gamma$ that begin at~$v_0$.  This set has the natural structure of a locally finite tree, whose root is the empty path~$\emptystring$, and whose boundary $\partial\mathcal{L}(\Gamma,v_0)$ is the set of all infinite directed paths in $\Gamma$ starting at~$v_0$. An example of such a tree is shown in  Figure~\ref{fig:PathLanguage}.
\begin{figure}[t]
\centering
\raisebox{-0.5\height}{\includegraphics{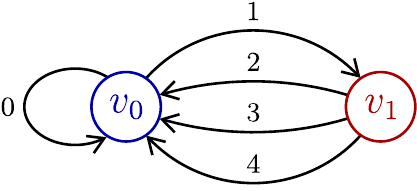}}
\hfill
\raisebox{-0.5\height}{\includegraphics{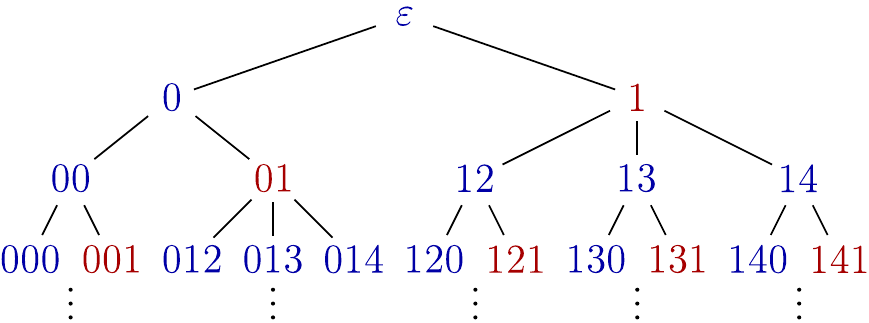}}
\caption{A directed multigraph $\Gamma$ and the corresponding path language tree~$\mathcal{L}(\Gamma,v_0)$.  Vertices of type~$v_0$ are shown in blue, and vertices of type~$v_1$ are shown in red.}
\label{fig:PathLanguage}
\end{figure}

We can define a self-similar structure on $\mathcal{L}(\Gamma,v_0)$ as follows:
\begin{enumerate}
\item Two paths $p,q\in \mathcal{L}(\Gamma,v_0)$ have the same type if and only if $p$ and $q$ end at the same vertex of $\Gamma$.
\item For each pair $p,q$ of paths ending at the same vertex~$w$, we define a single morphism $\varphi_{pq}\colon\mathcal{L}(\Gamma,v_0)_p\to \mathcal{L}(\Gamma,v_0)_q$ by
\[
\varphi_{pq}(pr) = qr
\]
for every finite directed path $r$ in $\Gamma$ starting at~$w$.  We will refer to $\varphi_{pq}$ as a \newword{prefix replacement morphism}.
\end{enumerate}
It is easy to check that this satisfies the required axioms, and gives $\mathcal{L}(\Gamma,v_0)$ the structure of a rigid tree.\qed
\end{example}

It turns out that, for any self-similar tree $T$, the underlying tree is isomorphic to a path language as described above. In particular, define the \newword{type graph} $\Gamma$ of a self-similar tree $T$ as follows:
\begin{enumerate}
\item There is one vertex in $\Gamma$ for each vertex type in~$T$.
\item Given a pair $t_1,t_2$ of vertices in $\Gamma$, the number of directed edges in $\Gamma$ from $t_1$ to $t_2$ is equal to the number of children of type~$t_2$ that each vertex of type~$t_1$ has in~$T$.
\end{enumerate}
Then it is easy to construct an isomorphism of trees $\mathcal{L}(\Gamma,t_0)\to T$, where $t_0$ is the type of the root vertex of~$T$, though the morphisms of $\mathcal{L}(\Gamma,t_0)$ need not be the same as the morphisms of $T$.  However, we will show in Proposition~\ref{prop:IsomorphicToPathLanguage} that, in the case of a rigid tree, it is always possible to construct an isomorphism from a path language that respects the self-similar structure.

\subsection{Rational Homeomorphisms}\label{sec:RationalFunctions}

Our next goal is to define the group of rational homeomorphisms of $\partial T$.

\begin{definition}\label{def:EquivalentRestrictions}Let $T$ and $T'$ be self-similar trees, let $f\colon \partial T\to\partial T'$ be a homeomorphism, and let $v$ and $w$ be vertices of $T$ of the same type.  We say that $f$ has \newword{equivalent restrictions} at $v$ and $w$ if there exist vertices $x$ and $y$ of $T'$ of the same type and morphisms $\varphi\colon T_v\to T_w$ and $\psi\colon T'_x\to T'_y$ such that $f(\partial T_v) \subseteq \partial T'_x$, $f(\partial T_w)\subseteq \partial T'_y$, and the diagram
\[
\xymatrix{
\partial T_v \ar_{f}[d] \ar^{\varphi_\ast}[r] & 
\partial T_w \ar^{f}[d] \\ 
\partial T'_x\ar_{\psi_\ast}[r] & \partial T'_y }
\]
commutes, where $\varphi_*\colon\partial T_v\to\partial T_w$ and $\psi_*\colon \partial T'_x\to\partial T'_y$ are the induced homeomorphisms.
\end{definition}

\begin{proposition}Let $f\colon \partial T\to \partial T'$ be a function.  Then ``$f$ has equivalent restrictions at $v$ and $w$'' is an equivalence relation on vertices of~$T$.
\end{proposition}
\begin{proof}For the reflexive property, if $v$ is a vertex of~$T$, then it suffices to let $\varphi = \mathrm{id}_v$, $x$ and $y$ be the root of $T'$, and $\psi = \mathrm{id}_x=\mathrm{id}_y$.  The symmetric property is clear, by inversion of $\varphi$ and~$\psi$.

For the transitive property, let $u$, $v$, and $w$ be vertices of $T$ of the same type, and suppose $f$ has equivalent restrictions at $u$ and $v$ and also at $v$ and~$w$.  Then there exist vertices $x,y$ of $T'$ of the same type and $y',z$ of $T'$ of the same type with
\[
f(\partial T_u)\subseteq \partial T'_x,\quad
f(\partial T_v)\subseteq \partial T'_y \cap \partial T'_{y'},\quad
f(\partial T_w)\subseteq \partial T'_z
\]
and morphisms $\varphi\in\mathrm{Mor}(u,v)$, $\varphi'\in\mathrm{Mor}(v,w)$, $\psi\in\mathrm{Mor}(x,y)$, and $\psi'\in\mathrm{Mor}(y',z)$ making the following diagrams commute:
\[
\xymatrix{
\partial T_u \ar_{f}[d] \ar^{\varphi_\ast}[r] & 
\partial T_v \ar^{f}[d] \\ 
\partial T'_x\ar_{\psi_\ast}[r] & \partial T'_y }
\qquad
\xymatrix{
\partial T_v \ar_{f}[d] \ar^{\varphi'_\ast}[r] & 
\partial T_w \ar^{f}[d] \\ 
\partial T'_{y'}\ar_{\psi'_\ast}[r] & \partial T'_z }
\]
Note that either $y=y'$, or $y$ is a descendant of $y'$, or $y'$ is a descendant of~$y$.  If $y=y'$ then we are done, since it suffices to use $x$, $z$, and the compositions $\varphi'\circ \varphi\in\mathrm{Mor}(u,w)$ and $\psi'\circ \psi\in\mathrm{Mor}(x,z)$. If $y$ is a descendant of~$y'$, then we can replace $y'$  by $y$, $z$ by $\psi'(y)$, and $\psi'$ by its restriction to arrive in the case where $y=y'$.  Similarly, if $y'$ is a descendant of~$y$, then we can replace $y$ by $y'$, $x$ by $\psi^{-1}(y')$, and $\psi$ by its restriction to arrive in the case where $y=y'$.
\end{proof}

\begin{definition}\label{def:RationalHomeomorphism}Let $T$ and $T'$ be self-similar trees. A homeomorphism $f\colon\partial T\to\partial T'$ is \newword{rational} if it has only finitely many different equivalence classes of restrictions.
\end{definition}

We will show in Proposition~\ref{prop:SameDefinitionRational} that this definition of a rational homeomorphism is a generalization of the definition given in~\cite{GNS}.  First we need the following fundamental proposition.

\begin{proposition}\label{prop:DeepestParentsWork}Let $T$ and $T'$ be self-similar trees, let $f\colon \partial T\to\partial T'$ be a homeomorphism, and let $v$ and $w$ be vertices of $T$ of the same type.  Suppose that $f(\partial T_v)$ and $f(\partial T_w)$ each have at least two points, let $x$ be the deepest parent of $f(\partial T_v)$ in~$T'$, and let $y$ be the deepest parent of $f(\partial T_w)$ in~$T'$.  Then $f$ has equivalent restrictions at $v$ and $w$ if and only if $x$ and $y$ have the same type and there exist morphisms $\varphi\colon T_v\to T_w$ and $\psi\colon T'_x\to T'_y$ making the following diagram commute:
\[
\xymatrix{
\partial T_v \ar_{f}[d] \ar^{\varphi_\ast}[r] & 
\partial T_w \ar^{f}[d] \\ 
\partial T'_x\ar_{\psi_\ast}[r] & \partial T'_y }
\]
\end{proposition}
\begin{proof}Clearly the given condition implies that $f$ has equivalent restrictions at $v$ and~$w$.  For the converse, suppose that $f$ has equivalent restrictions at $v$ and~$w$.  Then there exist vertices $r$ and $s$ of $T'$ and morphisms $\varphi\colon T_v\to T_w$ and $\chi\colon T'_r\to T'_s$ such that $f(\partial T_v) \subseteq \partial T'_r$, $f(\partial T_w)\subseteq \partial T'_s$, and the diagram
\[
\xymatrix{
\partial T_v \ar_{f}[d] \ar^{\varphi_\ast}[r] & 
\partial T_w \ar^{f}[d] \\ 
\partial T'_r\ar_{\chi_\ast}[r] & \partial T'_s }
\]
commutes.  Since $f(\partial T_v) \subseteq \partial T'_r$ and $x$ is the deepest parent of $f(\partial T_v)$, we know that $x$ is a descendant of~$r$, and similarly $y$ is a descendant of~$s$.  We claim that~$\chi(x) = y$.

Since $\varphi_*$ is a homeomorphism, we know that $\varphi_*(\partial T_v) = \partial T_w$, so
\[
f(\partial T_w) = f\bigl(\varphi_*(\partial T_v)\bigr) = \chi_*\bigl(f(\partial T_v)\bigr) \subseteq \chi_*(\partial T_x') = \partial T_{\chi(x)}'.
\]
Since $y$ is the deepest parent of $f(\partial T_w)$, it follows that $y$ is a descendant of~$\chi(x)$.  But similarly
\[
f(\partial T_v) = f\bigl(\varphi_*^{-1}(\partial T_w)\bigr) = \chi_*^{-1}\bigl(f(\partial T_w)\bigr) \subseteq \chi_*^{-1}(\partial T_y) = \partial T_{\chi^{-1}(y)}
\]
so $x$ must be a descendant of $\chi^{-1}(y)$, and therefore~$\chi(x)=y$.  Since $\chi$ is a morphism, it follows that $x$ and $y$ have the same type, and the restriction $\psi\colon T_x'\to T_y'$ of~$\chi$ is a morphism with the desired properties.
\end{proof}

For the following proposition, recall from Example~\ref{ex:BasicTree} that if $A$ is a finite alphabet then $A^\tast$ has the natural structure of a self-similar tree.  The boundary $\partial A^\tast$ of this tree is homeomorphic to~$A^\omega$.

\begin{proposition}\label{prop:SameDefinitionRational}Let $A_1$ and $A_2$ be finite alphabets with at least two symbols.  Then a homeomorphism $f\colon A_1^\omega \to A_2^\omega$ is rational in the sense of Definition~\ref{def:RationalHomeomorphism} if and only if it is rational in the sense of Definition~\ref{def:RationalGNS}.
\end{proposition}
\begin{proof}Let $\alpha,\beta\in A_1^\tast$.  Let $\gamma,\delta\in A_2^\tast$ be the greatest common prefixes (i.e.~deepest parents) of $f\bigl(\partial(A_1^\tast)_\alpha\bigr)$ and $f\bigl(\partial(A_1^\tast)_\beta\bigr)$, respectively. Then the restrictions $f|_\alpha$ and $f|_\beta$ defined in Definition~\ref{def:Restriction} are equal if and only if the diagram
\[
\xymatrix@C=0.5in{
\partial (A_1^\tast)_\alpha \ar_{f}[d] \ar^{(\varphi_{\alpha\beta})_\ast}[r] & 
\partial (A_1^\tast)_\beta \ar^{f}[d] \\ 
\partial (A_2^\tast)_\gamma \ar_{(\varphi_{\gamma\delta})_*}[r] & \partial (A_2^\tast)_\delta }
\]
commutes, where $\varphi_{\alpha\beta}\colon (A_1^\tast)_\alpha\to (A_1^\tast)_\beta$ and $\varphi_{\gamma\delta}\colon (A_2^\tast)_\gamma\to (A_2^\tast)_\delta$ are the prefix replacement morphisms.   By Proposition~\ref{prop:DeepestParentsWork}, the diagram above commutes if and only if $f$ has equivalent restrictions at $\alpha$ and~$\beta$ in the sense of Definition~\ref{def:EquivalentRestrictions}, and therefore $f|_\alpha = f|_\beta$ if and only if $f$ has equivalent restrictions at $\alpha$ and~$\beta$ in the sense of Definition~\ref{def:EquivalentRestrictions}.  The result now follows immediately from Definition~\ref{def:RationalHomeomorphism} and Theorem~\ref{thm:RestrictionTest}.
\end{proof}

\begin{remark}In an effort to simplify the exposition, we are only considering rational homeomorphisms between the boundaries of self-similar trees.  It would be possible to develop a more general theory of rational functions by including a requirement similar to condition~(2) in Theorem~\ref{thm:RestrictionTest}.  In particular, if $T'$ is a self-similar tree, we say that a point $p\in\partial T'$ is a \newword{rational point} if there exist distinct $x,y\in T'$ with $p\in \partial T'_y\subseteq \partial T'_x$ and a morphism $\varphi\colon T'_x\to T'_y$ such that $\varphi_*(p)=p$.  These are the analogs in $\partial T'$ of eventually periodic points in~$A^\tast$ for a finite alphabet~$A$.  We could then define a continuous function $f\colon \partial T\to\partial T'$ to be rational if it satisfies the following conditions:
\begin{enumerate}
\item $f$ has finitely many equivalence classes of restrictions, and
\item For each $v\in T$, if $f(\partial T_v)$ is a single point, then this must be a rational point in $\partial T'$.
\end{enumerate}
This would agree with the existing definition of rational in the case of a continuous function $A_1^\omega\to A_2^\omega$, where $A_1$ and $A_2$ are finite alphabets.\qed
\end{remark}

We now wish to prove that the rational homeomorphisms of $\partial T$ form a group.  To simplify the initial development of our theory, we will restrict our class of self-similar trees.

\begin{definition}A self-similar tree $T$ is \newword{branching} if every vertex in $T$ has at least two children.
\end{definition}

If $T$ is branching, then $\partial T$ has no isolated points, and is therefore homeomorphic to the Cantor set.  Each basic clopen set $\partial T_v$ is also 
homeomorphic to a Cantor set, and has $v$ as its deepest parent.  Moreover, if $T$ and $T'$ are branching and $f\colon \partial T\to\partial T'$ is a rational homeomorphism, then $\partial T_v$ has at least two points for every $v\in V$, so $f(\partial T_v)$ always has a deepest parent.

For now, we will develop our theory only for branching self-similar trees, though we will extend to a larger class of self-similar trees in Section~\ref{sec:PathologicalCases}.

\begin{proposition}\label{prop:CompositionsRational}Let $T$, $T'$, and $T''$ be branching self-similar trees, and let $f\colon \partial T \to \partial T'$ and $g\colon \partial T'\to\partial T''$ be rational homeomorphisms. Then the composition $g\circ f\colon \partial T\to\partial T''$ is rational.
\end{proposition}
\begin{proof}Let $E$ be an equivalence class of vertices of $T$ under the equivalence relation ``$f$ has equivalent restrictions at $v$ and $w$''.  Note that there are only finitely many such $E$, since $f$ is rational.  Therefore, it suffices to prove that $g\circ f$ has finitely many equivalence classes of restrictions on~$E$.

For each $v\in E$, let $r(v)$ denote the deepest parent of $f(\partial T_v)$ in~$T'$, which exists since $T$ and $T'$ are branching.  By Proposition~\ref{prop:DeepestParentsWork}, for every pair of vertices $v,w\in E$, there exist a pair of morphisms $\varphi_{vw}\colon T_v\to T_w$ and $\psi_{vw}\colon T'_{r(v)}\to T'_{r(w)}$ making the following diagram commute:
\[
\xymatrix{
\partial T_v \ar_{f}[d] \ar^{(\varphi_{vw})_\ast}[r] & 
\partial T_w \ar^{f}[d] \\ 
\partial T'_{r(v)}\ar_{(\psi_{vw})_\ast}[r] & \partial T'_{r(w)} }
\]
Put an equivalence on $E$ by $v\sim w$ if $g$ has equivalent restrictions at $r(v)$ and~$r(w)$.  Since $g$ is rational, there are only finitely many equivalence classes.  Let $E'$ be such an equivalence class.  It suffices to show that $g\circ f$ has finitely many equivalence classes of restrictions on~$E'$.

For each $v\in E'$, let $s(v)$ denote the deepest parent of $g\bigl(\partial T'_{r(v)}\bigr)$ in $T''$, which exists since $T'$ and $T''$ are branching. By Proposition~\ref{prop:DeepestParentsWork}, for each $v,w\in E'$, there exist morphisms $\mu_{vw}\colon T'_{r(v)}\to T'_{r(w)}$ and $\nu_{vw}\colon T''_{s(v)} \to T''_{s(w)}$ making the following diagram commute:
\[
\xymatrix{
\partial T'_{r(v)} \ar_{g}[d] \ar^{(\mu_{vw})_\ast}[r] & 
\partial T'_{r(w)} \ar^{g}[d] \\ 
\partial T''_{s(v)}\ar_{(\nu_{vw})_\ast}[r] & \partial T''_{s(w)} }
\]
Fix a vertex $u\in E'$. For each $v\in E'$ let $\pi_v\colon T_{r(u)} \to T_{r(u)}$ be the morphism $\pi_v = \psi_{uv}^{-1} \mu_{uv}$.  We claim that, for $v,w\in E'$, the composition $g\circ f$ has equivalent restrictions at $v$ and $w$ whenever $\pi_v=\pi_w$.  Since $\mathrm{Mor}(r(u),r(u))$ is finite, it will follow immediately from this that $g\circ f$ has only finitely many different restrictions on~$E'$.

Let $v,w\in E'$ and suppose that $\pi_v=\pi_w$.  Then $\psi_{uv}^{-1}\mu_{uv} = \psi_{uw}^{-1}\mu_{uw}$, so the following diagram commutes, including the central square if we allow inverse arrows:
\[
\xymatrix@R=0.25in{
& \partial T_u \ar_{f}[d] \ar^{(\varphi_{uw})_\ast}[rd]\ar_{(\varphi_{uv})^\ast}[ld] & \\ 
\partial T_v\ar_{f}[d]  & \partial T'_{r(u)}\ar_{(\psi_{uw})_\ast}[rd]\ar^{(\psi_{uv})_\ast}[ld] & \partial T_{w}\ar^{f}[d] \\
\partial T'_{r(v)}\ar_{g}[d] & & \partial T'_{r(w)}\ar^{g}[d] \\
\partial T''_{s(v)} & \partial T'_{r(u)}\ar_{g}[d]\ar_{(\mu_{uv})_*}[ul]\ar^{(\mu_{uw})_*}[ur] & \partial T''_{s(w)} \\
& \partial T''_{s(u)}\ar^{(\nu_{uv})_*}[ul]\ar_{(\nu_{uw})_*}[ur] &
}
\]
In particular, the outer octagon gives us the commutative square
\[
\xymatrix@C=0.65in{
\partial T_v \ar_{g\circ f}[d] \ar^{\bigl(\varphi_{uw}\varphi_{uv}^{-1}\bigr)_{\!\ast}}[r] & 
\partial T_w \ar^{g\circ f}[d] \\ 
\partial T''_{s(v)}\ar_{\bigl(\nu_{uw}\nu_{uv}^{-1}\bigr)_{\!\ast}}[r] & \partial T''_{s(w)} }
\]
and therefore $g\circ f$ has equivalent restrictions at $v$ and~$w$.
\end{proof}

\begin{proposition}\label{prop:InversesRational}Let $T$ and $T'$ be branching self-similar trees, and let $f\colon \partial T\to\partial T'$ be a rational homeomorphism. Then the inverse $f^{-1}\colon \partial T'\to\partial T$ is rational.
\end{proposition}
\begin{proof}For each $v$ in $T'$, let $r(v)$ be the deepest parent of $f^{-1}(T'_v)$ in~$T$, which exists since $T$ and $T'$ are branching. Put an equivalence relation $\sim$ on the vertices of $T'$ by $v\sim w$ if $f$ has equivalent restrictions at $\partial T_{r(v)}$ and $\partial T_{r(w)}$.  Since $f$ is rational, there are only finitely many such equivalence classes.  Let $E$ be such an equivalence class.  It suffices to prove that $f^{-1}$ has only finitely many different restrictions on~$E$.

Fix a vertex $u\in E$, and for each $v\in E$ let $s(v)$ be the deepest parent of $f(\partial T_{r(v)})$ in~$T'$, which exists since $T$ and $T'$ are branching. By Proposition~\ref{prop:DeepestParentsWork}, for each $v\in E$ there exist morphisms $\varphi_v\colon T_{r(v)}\to T_{r(u)}$ and $\psi_v\colon T'_{s(v)}\to T'_{s(u)}$ making the following diagram commute:
\[
\xymatrix{
\partial T_{r(v)} \ar_{f}[d] \ar^{(\varphi_v)_\ast}[r] & 
\partial T_{r(u)} \ar^{f}[d] \\ 
\partial T'_{s(v)}\ar_{(\psi_v)_\ast}[r] & \partial T'_{s(u)} }
\]
Now since $f^{-1}(\partial T'_v)\subseteq \partial T_{r(v)}$, we know that
\[
\partial T'_v \subseteq f(\partial T_{r(v)}) \subseteq \partial T'_{s(v)}
\]
for each $v$. Since $T'$ is branching, it follows that $v$ is a descendant of~$s(v)$. Then $\psi_v(v)$ is some descendant of~$s(u)$.  We claim that the set
\[
\{\psi_v(v) \mid v\in E\}
\]
is finite and that $f^{-1}$ has equivalent restrictions at $v,w\in E$ whenever $\psi_v(v)=\psi_w(w)$. This will prove that $f^{-1}$ has only finitely many different restrictions on $E$.

To prove that there are only finitely many possibilities for $\psi_v(v)$, recall that $r(v)$ is the deepest parent of $f^{-1}(\partial T'_v)$ in~$T$.  But
\[
\varphi_v(r(v)) = r(u)
\]
and
\[
(\varphi_v)_*\bigl(f^{-1}(\partial T'_v)\bigr) = f^{-1}\bigl((\psi_v)_*(\partial T'_v)\bigr) =  f^{-1}\bigl(\partial T'_{\psi_v(v)}\bigr)
\]
so $r(u)$ must be the deepest parent of~$f^{-1}\bigl(\partial T'_{\psi_v(v)}\bigr)$ in~$T$, i.e.~$r\bigl(\psi_v(v)\bigr) = r(u)$.  But since $f^{-1}$ is continuous, it is uniformly continuous with respect to the standard ultrametrics on $\partial T'$ and $\partial T$.  In particular, there exists a $k > 0$ such that
\[
d(v,w) \leq \frac{1}{2^k} \quad\Rightarrow\quad d\bigl(f^{-1}(v),f^{-1}(w)\bigr) < \frac{1}{2^{|r(u)|}}
\]
for all $v,w$ in $T'$.  It follows that $|r(v)| > |r(u)|$ whenever $|v|\geq k$, so there are only finitely many vertices $v$ for which $r(v)= r(u)$.  Thus there are only finitely many possibilities for~$\psi_v(v)$.

Now suppose that $v,w\in E$ and $\psi_v(v) = \psi_w(w)$.  Then $\psi_w^{-1}\psi_v$ maps $v$ to~$w$, so let $\chi\colon T_v\to T_w$ be the restriction of this morphism.  Then $\chi_*$ agrees with $(\psi_w)_*^{-1}(\psi_v)_*$ on $\partial T'_v$, and in particular the following diagram commutes:
\[
\xymatrix@C=0.6in{
\partial T_{r(v)} \ar^{(\varphi_w^{-1}\varphi_v)_\ast}[r] & 
\partial T_{r(u)} \\ 
\partial T'_v\ar^{f^{-1}}[u]\ar_{\chi_\ast}[r] & \partial T'_w\ar_{f^{-1}}[u] }
\]
We conclude that $f^{-1}$ has equivalent restrictions at $v$ and~$w$.
\end{proof}

\begin{corollary}\label{cor:RationalGroup}If $T$ is a branching self-similar tree, then the set of all rational homeomorphisms of $\partial T$ forms a group under composition.\hfill\qedsymbol
\end{corollary}

This is the \newword{rational group} associated with $T$, denoted $\R_T$. The next two subsections are devoted to proving the following theorem.

\begin{theorem}\label{thm:RationalGroupIsomorphism}Let $T$ be a branching self-similar tree.  Then the associated rational group $\R_T$ is isomorphic to the binary rational group~$\R_2$. Indeed, the action of $\R_T$ on $\partial T$ is conjugate to the action of $\R_2$ on $\{0,1\}^\omega$.
\end{theorem}

\subsection{Rigid Structures}

\begin{definition}Let $T$ be a self-similar tree.  A \newword{rigid structure} for $T$ is a family $\{\varphi_{vw}\}$ of morphisms, with one morphism $\varphi_{vw}\colon T_v\to T_w$ for each pair $(v,w)$ of vertices of $T$ of the same type, satisfying the following conditions:
\begin{enumerate}
\item $\varphi_{vw}\,\varphi_{uv} = \varphi_{uw}$ for all triples $(u,v,w)$ of the same type.
\[
\xymatrix{
 T_u \ar_{\varphi_{uw}}[dr] \ar^{\varphi_{uv}}[r] & 
 T_v \ar^{\varphi_{vw}}[d] \\ 
 &
 T_w }
\]
\item If $v'$ is a descendant of $v$ and $w' = \varphi_{vw}(v')$, then $\varphi_{v'w'}$ is the restriction of $\varphi_{vw}$ to~$T_{v'}$.
\end{enumerate}
\end{definition}

Note that if $\{\varphi_{vw}\}$ is a rigid structure for $T$ then $\varphi_{vv}\,\varphi_{vv}=\varphi_{vv}$ for any vertex $v$ of~$T$, and hence $\varphi_{vv}$ is the identity isomorphism of~$T_v$.  It follows that $\varphi_{wv}=\varphi_{vw}^{-1}$ for all pairs~$(v,w)$.

\begin{proposition}\label{prop:RigidStructureExists}Every self-similar tree has a rigid structure.
\end{proposition}
\begin{proof}Let $T$ be a self-similar tree. Choose a set $\mathcal{A}$ of vertices of $T$ that contains the root vertex and has exactly one vertex of each type. If $v$ is a vertex of $T$, a \newword{marking} of $v$ will be an element of $\mathrm{Mor}(v,a)$, where $a$ is the vertex in $\mathcal{A}$ having the same type as~$v$.

Let $\mathcal{B}$ be the set of all vertices that are children of vertices in~$\mathcal{A}$. Choose a marking $\tau_b$ for each vertex $b\in\mathcal{B}$.  We now define a marking $\psi_v$ for each vertex $v\in T$ inductively as follows:
\begin{enumerate}
\item If $r$ is the root of $T$, then $\psi_r$ is the identity isomorphism of~$T$.
\item If $v$ is a vertex of $T$ with marking $\psi_v$ and $w$ is a child of~$v$, let $b=\psi_v(w)$ denote the corresponding child of $\psi_v(v)$, and let
\[
\psi_w = \tau_b\psi_v,
\]
where the composition on the right is partial, as described in Note~\ref{note:PartialComposition}
\end{enumerate}
For each pair of vertices $(v,w)$ of $T$ of the same type, let $\varphi_{vw} = \psi_w^{-1}\psi_v$.  We claim that $\{\varphi_{vw}\}$ is a rigid structure for~$T$.

Clearly $\varphi_{vw}\,\varphi_{uv}=\varphi_{uw}$ for every triple $(u,v,w)$ of vertices of the same type.  For restrictions, suppose that $v$ and $w$ have the same type and $v'$ is a child of~$v$.  Let $w'=\varphi_{vw}(v')$.  Let $a = \psi_v(v)=\psi_w(w)$, and let $b=\psi_v(v')=\psi_w(w')$.  Then
\[
\varphi_{v'w'} = \psi_{w'}^{-1}\psi_{v'} = \psi_w^{-1}\tau_b^{-1}\tau_b\psi_v = \psi_w^{-1} i_b\,\psi_v
= \psi_w^{-1}\psi_v\, i_{v'} = \varphi_{vw}\,i_{v'},
\]
where each $i_x$ denotes the identity map on~$T_x$, and we are again using partial compositions. We conclude that $\varphi_{v'w'}$ is the restriction of $\varphi_{vw}$ to~$T_{v'}$, as desired.
\end{proof}

Note that a rigid structure is itself a self-similar structure on~$T$.  If $T$ is a self-similar tree and $\{\varphi_{vw}\}$ is a rigid structure on~$T$, the \newword{corresponding rigid tree} is the self-similar tree having the same underlying graph as $T$ but with $\{\varphi_{vw}\}$ as its self-similar structure.

\begin{proposition}\label{prop:RigidStructureSameRationalHomeomorphisms}Let $T$ be a branching self-similar tree, let $\{\varphi_{vw}\}$ be a rigid structure on~$T$, and let $T'$ be the corresponding rigid tree.  Then the rational homeomorphisms of $\partial T$ are the same as the rational homeomorphisms of $\partial T'$.
\end{proposition}

\begin{proof}Note that the identity map $i\colon \partial T'\to\partial T$ is rational.  In particular, for any vertices $v,w\in T'$ of the same type, we have a commutative diagram
\[
\xymatrix@C=0.5in@R=0.5in{
\partial T'_v \ar_{i}[d] \ar^{(\varphi_{vw})_\ast}[r] & 
\partial T'_w \ar^{i}[d] \\ 
\partial T_v\ar_{(\varphi_{vw})_\ast}[r] & \partial T_w }
\]
and therefore $i$ has equivalent restrictions at $v$ and $w$.  Conjugating by~$i$, we deduce that any rational homeomorphism of $\partial T$ is also a rational homeomorphism of $\partial T'$, and vice versa.
\end{proof}

\subsection{Rigid Trees and \texorpdfstring{$\R_2$}{R2}}

In this subsection we complete the proof of Theorem~\ref{thm:RationalGroupIsomorphism}.

\begin{definition}Let $T$ and $T'$ be self-similar trees and let $\Phi\colon T\to T'$ be an isomorphism of rooted trees.  We say that $\Phi$ is an \newword{isomorphism of self-similar trees} if the following conditions are satisfied:
\begin{enumerate}
\item Two vertices $v,w\in T$ have the same type if and only $\Phi(v)$ and $\Phi(w)$ have the same type in~$T'$.
\item For every pair $v,w$ of vertices of $T$ of the same type, $\Phi$~conjugates $\mathrm{Mor}(v,w)$ to $\mathrm{Mor}\bigl(\Phi(v),\Phi(w)\bigr)$.
\end{enumerate}
\end{definition}

For the following theorem, we say that a directed multigraph $\Gamma$ is \newword{branching} if each vertex in $\Gamma$ has at least two outgoing edges.

\begin{proposition}\label{prop:IsomorphicToPathLanguage}Let $T$ be a branching, rigid self-similar tree. Then the type graph $\Gamma$ for $T$ is branching, and there exists an isomorphism of self-similar trees $\Phi\colon T\to\mathcal{L}(\Gamma,t_0)$, where $t_0$ is the type of the root vertex in~$T$
\end{proposition}
\begin{proof}Since $T$ is branching, $\Gamma$ must be branching as well.  Let $v_0$ be the root of $T$, and let $V$ be a set of vertices in $T$ that contains the root and has exactly one vertex of each type.  Let $\tau\colon T\to V$ be the function that assigns to each vertex $x\in T$ the vertex $\tau(x)\in V$ having the same type as~$x$.  Then we can think of the elements of $V$ as the vertices of $\Gamma$, with one directed edge $e_w$ in $\Gamma$ from $v$ to $\tau(w)$ for each $v\in V$ and each child $w$ of~$v$.

For any two vertices $x,y\in T$ of the same type, let $\psi_{x,y}$ denote the morphism from $x$ to~$y$.  Define a tree isomorphism $\Phi\colon \mathcal{L}(\Gamma,v_0) \to T$ inductively by $\Phi(\emptystring) = v_0$ and
\[
\Phi(pe_w) = \psi_{v,\Phi(p)}(w)
\]
for each path $p$ in $\Gamma$ from $v_0$ to $v$ and each edge $e_w$ starting at~$v$. Note that $\tau(\Phi(p))$ is always the endpoint of~$p$, and therefore two paths $p,q$ in~$\mathcal{L}(\Gamma,v_0)$ have the same type if and only if $\Phi(p)$ and $\Phi(q)$ have the same type in~$T$.

Now let $p$ and $p'$ be paths in $\Gamma$ from $v_0$ to some vertex~$v\in V$, and let $\varphi_{p,p'}\in\mathrm{Mor}(p,p')$ be the prefix replacement, i.e.
\[
\varphi_{p,p'}(pq) = p'q.
\]
We claim that
\[
\psi_{\Phi(p),\Phi(p')}\bigl(\Phi(pq)\bigr) = \Phi\bigl(\varphi_{p,p'}(pq)\bigr)
\]
for every path $q$ in $\Gamma$ starting at~$v$. We proceed by induction on~$q$. Note that the statement is trivially true for $q=\emptystring$.  Now suppose it is true for some path $q$ from $v$ to~$v'$, and let $e_w$ be an edge in $\Gamma$ starting at~$v'$.  Then
\[
\psi_{\Phi(p),\Phi(p')}\bigl(\Phi(pqe_w)\bigr) = \psi_{\Phi(p),\Phi(p')}\bigl(\psi_{v',\Phi(pq)}(w)\bigr)
\]
But $\psi_{\Phi(p),\Phi(p')}\bigl(\Phi(pq) \bigr) = \Phi\bigl(\varphi_{p,p'}(pq)\bigr) = \Phi(p'q)$, so since $T$ is rigid \[
\psi_{\Phi(p),\Phi(p')}\circ \psi_{v',\Phi(pq)}= \psi_{v',\Phi(p'q)}.
\]
Now,
\[
\psi_{\Phi(p),\Phi(p')}\bigl(\Phi(pqe_w)\bigr) = \psi_{v',\Phi(p'q)}(w) = \Phi(p'qe_w) = \Phi\bigl(\varphi_{p,p'}(pqe_w)\bigr).
\]
We conclude that $\Phi$ is an isomorphism of self-similar trees.
\end{proof}

\begin{proposition}\label{prop:RationalHomeomorphismToT2}Let\/ $\Gamma=(V,E)$ be a finite, directed, branching multigraph, and let $v_0\in V$.  Then there exists a rational homeomorphism from $\mathcal{L}(\Gamma,v_0)$ to the Cantor set~$\{0,1\}^\omega$.
\end{proposition}
\begin{proof}For each vertex $v\in V$, choose a complete binary prefix code for the edges of $E$ with initial vertex~$v$. Together, these codes define an encoding function $c\colon E\to\{0,1\}^\tast$, which we can extend to a function $c\colon E^\tast\to \{0,1\}^\tast$ by
\[
c(e_1 e_2\cdots e_n) = c(e_1)\cdot c(e_2) \cdot \cdots \cdot c(e_n).
\]
Let $f\colon \partial\mathcal{L}(\Gamma,v_0)\to \{0,1\}^\omega$ be the function
\[
f(e_1e_2\cdots) = c(e_1)\cdot c(e_2) \cdot \cdots.
\]
By construction $f$ is bijective.  We claim that $f$ is rational.

Let $p$ and $q$ be any two directed paths in $\Gamma$ from $v_0$ to the same vertex~$w$, and let $\varphi_{pq}\colon \mathcal{L}(\Gamma,v_0)_p \to \mathcal{L}(\Gamma,v_0)_q$ be the prefix replacement morphism. Let $\varphi_{c(\alpha),c(\beta)} \colon \{0,1\}^\tast_{c(p)}\to \{0,1\}^\tast_{c(q)}$ be the prefix replacement morphism between the corresponding subtrees of~$\{0,1\}^\tast$.  Then it is easy to check that the diagram
\[
\xymatrix@R=0.5in{
\partial \mathcal{L}(\Gamma,v_0)_p\ar_{f}[d] \ar^{\varphi_*}[r] & 
\partial \mathcal{L}(\Gamma,v_0)_q \ar^{f}[d] \\ 
\partial\{0,1\}^\tast_{c(p)}\ar_{\psi_*}[r] & \partial\{0,1\}^\tast_{c(q)}
}
\]
commutes.  Since $\mathcal{L}(\Gamma,v_0)$ has only finitely many different types, this proves that $f$ is rational.
\end{proof}

\begin{proof}[Proof of Theorem~\ref{thm:RationalGroupIsomorphism}] Let $T$ be a branching self-similar tree.  By Proposition~\ref{prop:RigidStructureExists}, there exists a rigid structure on~$T$, giving us a rigid tree~$T'$. By Proposition~\ref{prop:RigidStructureSameRationalHomeomorphisms}, we know that $\R_T = \R_T'$.  By Proposition~\ref{prop:IsomorphicToPathLanguage}, there exists a finite, directed, branching multigraph $\Gamma = (V,E)$ and a vertex $v_0\in V$ so that~$T'$ is isomorphic to $\mathcal{L}(\Gamma,v_0)$.  Let $\Phi\colon T'\to \mathcal{L}(\Gamma,v_0)$ an isomorphism of self-similar trees, and let $\Phi_*\colon \partial T'\to \partial \mathcal{L}(\Gamma,v_0)$ be the associated rational homeomorphism.  By Proposition~\ref{prop:RationalHomeomorphismToT2}, there exists a rational homeomorphism $h\colon \mathcal{L}(\Gamma,v_0)\to \{0,1\}^\omega$.  Thus, the mapping
\[
f \mapsto h\circ \Phi\circ f\circ \Phi^{-1}\circ h^{-1}
\]
is an isomorphism from $\R_T$ to~$\R_2$.  Indeed, $h\circ\Phi$ conjugates the action of $\R_T$ on $\partial T$ to the action of $\R_2$ on $\{0,1\}^\omega$.
\end{proof}

\subsection{Non-Branching Trees}\label{sec:PathologicalCases}

In this section we deal with the case of non-branching self-similar trees.  Note that such trees may have isolated points in their boundaries.

\begin{definition}A self-similar tree $T$ is \newword{without dead ends} if every vertex of $T$ has at least one child.
\end{definition}

We wish to extend our theory of rational homeomorphisms to self-similar trees without dead ends.  Specifically, we wish to prove that the set $\R_T$ of rational homeomorphisms for such a tree forms a group, and that this group \textit{embeds} into the binary rational group~$\R_2$.

\begin{definition}Let $T$ be a tree without dead ends.  A subtree $T_v\subseteq V$ is called an \newword{isolated branch} if
\begin{enumerate}
\item $\partial T_v$ is a single point, and
\item There does not exist an ancestor $w$ of $v$ such that $\partial T_v=\partial T_w$.
\end{enumerate}
\end{definition}

As long as $v$ is not the root, condition (2) is equivalent to saying that the parent of $v$ has at least two children.  Note that isolated branches of $T$ are in one-to-one correspondence with isolated points in~$\partial T$.

\begin{definition}Let $T$ be a tree without dead ends.  The \newword{expansion} of $T$ is the tree $E[T]$ obtained by replacing each isolated branch $T_v$ of $T$ by a copy of the infinite binary tree $\{0,1\}^\tast$.
\end{definition}

\noindent More formally, each vertex of $E[T]$ is either
\begin{enumerate}
\item A vertex $v\in T$ that does not lie in an isolated branch, or
\item A pair $(v,\alpha)$, where $T_v$ is an isolated branch of $T$ and $\alpha\in \{0,1\}^\tast$.
\end{enumerate}
Vertices of the first type are called \newword{old vertices}, and vertices of the second type are \newword{new vertices}.
Descendants in $E[T]$ are defined as follows:
\begin{enumerate}
\item If $v$ is an old vertex, then the descendants of $v$ in $E[T]$ consist of all old vertices that are descendants of $v$ in~$T$, together with all new vertices $(w,\alpha)$ for which $w$ is a descendant of $v$ in~$T$.
\item The descendants of a new vertex $(v,\alpha)$ are all pairs $(v,\beta)$ for which $\alpha$ is a prefix of $\beta$.
\end{enumerate}
Note then that each isolated branch $T_v$ in $T$ has a corresponding infinite binary tree $E[T]_{(v,\emptystring)}$ of new vertices in~$E[T]$.

We place a self-similar structure on $E[T]$ as follows:
\begin{enumerate}
\item If $v$ and $w$ are old vertices in $E[T]$, then $v$ and $w$ have the same type in $E[T]$ if and only if they have the same type in~$T$.  For each morphism $\varphi\colon T_v\to T_w$, there is a corresponding morphism $E[\varphi]\colon E[T]_v\to E[T]_w$, which maps each old vertex $v'\in E[T]_v$ to~$\varphi(v')$, and maps each new vertex $(v',\alpha)\in E[T]_v$ to $\bigl(\varphi(v'),\alpha\bigr)$.
\item Any two new vertices $(v,\alpha)$ and $(w\beta)$ in $E[T]$ have the same type, with a unique morphism $\varphi_{(v,\alpha),(w,\beta)}\colon E[T]_{(v,\alpha)}\to E[T]_{(w,\beta)}$ defined by
\[
\varphi_{(v,\alpha),(w,\beta)}(v,\alpha\psi) = (w,\beta\psi)
\]
for all $\psi\in\{0,1\}^\tast$.
\end{enumerate}
It is easy to check that this satisfies the axioms for a self-similar structure.

Now, observe that each point in $\partial E[T]$ consists of either:
\begin{enumerate}
\item A non-isolated point in $\partial T$, or
\item A pair $(p,\psi)$, where $p$ is an isolated point in $\partial T$ and $\psi\in\{0,1\}^\omega$.
\end{enumerate}
If $f\colon \partial T\to\partial T'$ is a homeomorphism, then $f$ must map isolated points of $\partial T$ to isolated points of $\partial T'$, and therefore $f$ induces a homeomorphism $E[f]\colon \partial E[T]\to\partial E[T']$ defined by $E[f](p) = f(p)$ if $p$ is a non-isolated point in~$\partial T$, and $E[f](p,\psi) = (f(p),\psi)$ if $p$ is an isolated point in $\partial T$ and $\psi\in\{0,1\}^\omega$.

Note that this operation satisfies $E[f\circ g]=E[f]\circ E[g]$ for any homeomorphisms $g\colon \partial T\to \partial T'$ and $f\colon \partial T'\to\partial T''$.  Similarly, $E[f^{-1}]=E[f]^{-1}$ for any homeomorphism $f\colon\partial T\to\partial T'$.

\begin{lemma}\label{lem:Expanding}Let $T$ and $T'$ be self-similar trees without dead ends, and let $f\colon\partial T\to\partial T'$ be a homeomorphism.  Then $f$ is rational if and only if the induced homeomorphism $E[f]\colon E[T]\to E[T']$ is rational.
\end{lemma}
\begin{proof}Let $v$ and $w$ be old vertices of $E[T]$ of the same type. Since $\partial T_v$ and $\partial T_w$ each have at least two points, the images $f(\partial T_v)$ and $f(\partial T_w)$ each have at least two points, so they have deepest parents $x$ and $y$, respectively.  These must be old vertices of~$E[T']$, and indeed are the deepest parents of $E[f](\partial E[T]_v)$ and $E[f](\partial E[T]_w)$, respectively.  Then for any two morphisms $\varphi\colon T_v\to T_w$ and $\psi\colon T'_x\to T'_y$, the diagram
\[
\xymatrix{
\partial T_v \ar_{f}[d] \ar^{\varphi_\ast}[r] & 
\partial T_w \ar^{f}[d] \\ 
\partial T'_x\ar_{\psi_\ast}[r] & \partial T'_y }
\]
commutes if and only if the diagram
\[
\xymatrix{
\partial E[T]_v \ar_{E[f]}[d] \ar^{E[\varphi]_\ast}[r] & 
\partial E[T]_w \ar^{E[f]}[d] \\ 
\partial E[T']_x\ar_{E[\psi]_\ast}[r] & \partial E[T']_y }
\]
commutes, so $f$ has equivalent restrictions at $v$ and $w$ if and only if $E[f]$ has equivalent restrictions at $v$ and~$w$.

Finally, if $(v,\alpha)$ and $(w,\beta)$ are new vertices of $E[T]$, then the diagram
\[
\xymatrix@C=1.1in@R=0.5in{
\partial E[T]_{(v,\alpha)} \ar_{E[f]}[d] \ar^{(\varphi_{(v,\alpha),(w\beta)})_\ast}[r] & 
\partial E[T]_w \ar^{E[f]}[d] \\ 
\partial E[T']_{(f(v),\alpha)}\ar_{(\varphi_{(f(v),\alpha),(f(w)\beta)})_\ast}[r] & \partial E[T']_{(f(w),\beta)} }
\]
commutes, so $E[f]$ has equivalent restrictions at $(v,\alpha)$ and $(w,\beta)$.  It follows that $f$ is rational if and only if $E[f]$ is rational.
\end{proof}

Lemma~\ref{lem:Expanding} allows us to eliminate isolated branches from any tree without dead ends.  That is, we can now assume that our tree $T$ without dead ends also has no isolated branches.  For such a tree, any vertex with only one child must eventually have a descendant with two or more children.

\begin{definition}Let $T$ be a self-similar tree without dead ends or isolated branches.
\begin{enumerate}
\item A vertex $v\in T$ is \newword{essential} if $v$ has at least two children.
\item The \newword{simplification} of $T$ is the tree $S[T]$ consisting of all essential vertices of~$T$.
\end{enumerate}
\end{definition}

That is, $S[T]$ is the tree of all essential vertices of $T$, where $v$ is a descendant of $w$ in $S[T]$ if and only if $v$ is a descendant of $w$ in~$T$.  Note then that an essential vertex $w$ is a child in of an essential vertex~$v$ in $S[T]$ if and only if $w$ is a descendant of $v$ in $T$ and each intermediate vertex on the path from $v$ to $w$ in $T$ has only one child.  Since each vertex in $S[T]$ has the same number of children as in~$T$, the tree $S[T]$ is branching.

We can put a self similar structure on $S[T]$ by simply restricting morphisms of~$T$.  That is, if $v$ and $w$ are essential vertices in~$T$, then $v$ and $w$ have the same type in $S[T]$ if and only if they have the same type in~$T$.  Given any morphism $\varphi\colon T_v\to T_w$, we define a morphism $S[\varphi]\colon S[T]_v\to S[T]_w$ obtained from $\varphi$ by restricting to the essential vertices.

It is not hard to see that $\partial S[T]$ is naturally homeomorphic to $\partial T$, since an infinite descending path in~$T$ is completely determined by which essential vertices it passes through.  If $f\colon \partial T\to\partial T'$ is a homeomorphism, we let $S[f]\colon \partial S[T]\to\partial S[T']$ be the induced homeomorphism.

\begin{lemma}\label{lem:Simplifying}Let $T$ and $T'$ be self-similar trees without isolated branches, and let $f\colon \partial T\to\partial T'$.  Then $f$ is rational if and only if the induced homeomorphism $S[f]\colon \partial S[T]\to\partial S[T']$ is rational.
\end{lemma}
\begin{proof}We claim first that, if $v$ and $w$ are essential vertices of $T$ of the same type, then $f$ has equivalent restrictions at $v$ and $w$ if and only if $S[f]$ has equivalent restrictions at $v$ and $w$.  To see this, let $x$ and $y$ be the deepest parents of $f(\partial T_v)$ and $f(\partial T_w)$, respectively, which exist since $T$ and $T'$ have no isolated branches.  Note that $x$ and $y$ must be essential, since otherwise they would not be deepest parents.  Moreover, $x$ and $y$ are the deepest parents of $S[f](\partial S[T]_v)$ and $S[f](\partial S[T]_w)$, respectively.  Then for any two morphisms $\varphi\colon T_v\to T_w$ and $\psi\colon T'_x\to T'_y$, the diagram
\[
\xymatrix{
\partial T_v \ar_{f}[d] \ar^{\varphi_\ast}[r] & 
\partial T_w \ar^{f}[d] \\ 
\partial T'_x\ar_{\psi_\ast}[r] & \partial T'_y }
\]
commutes if and only if the diagram
\[
\xymatrix{
\partial S[T]_v \ar_{S[f]}[d] \ar^{S[\varphi]_\ast}[r] & 
\partial S[T]_w \ar^{S[f]}[d] \\ 
\partial S[T']_x\ar_{S[\psi]_\ast}[r] & \partial S[T']_y }
\]
commutes, which proves the claim.

It follows immediately that $S[f]$ is rational whenever $f$ is rational.  For the converse, suppose that $S[f]$ is rational.  For each $v\in T$, let $r(v)$ denote the deepest parent of $\partial T_v$, i.e.~the first descendant of $v$ that is essential.  Note then that $\partial T_{r(v)} = \partial T_v$ for each~$v$, since every infinite descending path in $T$ that passes through $v$ must pass through $r(v)$ as well.  Put an equivalence relation $\sim$ on the vertices of $T$ by $v\sim w$ if $v$ and $w$ have the same type and $f$ has equivalent restrictions at $r(v)$ and $r(w)$.  Since $S[f]$ is rational, we know from the previous paragraph that $f$ has only finitely many equivalence classes of restrictions at essential vertices, which means that $\sim$ has only finitely many equivalence classes.  Let $E$ be such an equivalence class.  It suffices to prove that $f$ has only finitely many equivalence classes of restrictions on~$E$.

Fix a vertex $u\in E$.  For each $v\in E$, let $s(v)$ denote the deepest parent of $f(\partial T_v)$, and let $\varphi_v\colon T_{r(u)}\to T_{r(v)}$ and $\psi_v\colon T_{s(u)}\to T_{s(v)}$ be morphisms that make the following diagram commute:
\[
\xymatrix{
\partial T_{r(u)} \ar_{f}[d] \ar^{(\varphi_v)_*}[r] & 
\partial T_{r(v)} \ar^{f}[d] \\ 
\partial T'_{s(u)}\ar_{(\psi_v)_*}[r] & \partial T'_{s(v)} }
\]
Choose also for each $v\in E$ a morphism $\chi_v\colon T_v\to T_u$.  Since $\chi_v$ is a tree isomorphism, it must map $r(v)$ to~$r(u)$ and hence $T_{r(v)}$ to~$T_{r(u)}$.  Then the composition $\chi_v\varphi_v$ is a morphism $T_{r(u)}\to T_{r(u)}$.  We claim that $f$ has equivalent restrictions at two vertices $v,w\in E$ whenever $\chi_v\varphi_v=\chi_w\varphi_w$.  Since there are only finitely many morphisms $T_{r(u)}\to T_{r(u)}$ in~$T$, it will follow from this that $f$ has only finitely many equivalence classes of restrictions on~$E$, and therefore $f$ is rational.

Let $v,w\in E$ and suppose that $\chi_v\varphi_v=\chi_w\varphi_w$.  Then we have a commutative diagram
\[
\xymatrix{
\partial T_v\ar^{(\chi_v)_*}[r] & \partial T_u & \partial T_w \ar_{(\chi_w)_*}[l] \\
\partial T_{r(v)}\ar^{\mathrm{id}}[u] \ar_{f}[d]&
\partial T_{r(u)} \ar_{f}[d] \ar^{(\varphi_w)_*}[r] \ar_{(\varphi_v)_*}[l] & 
\partial T_{r(w)} \ar^{f}[d] \ar_{\mathrm{id}}[u] \\ 
\partial T'_{s(v)}  &
\partial T'_{s(u)}\ar^{(\psi_v)_*}[l] \ar_{(\psi_w)_*}[r] &
\partial T'_{s(w)} }
\]
In particular, the outer square
\[
\xymatrix@C=0.55in{
\partial T_{v} \ar_{f}[d] \ar^{(\chi_w^{-1}\chi_v)_*}[r] & 
\partial T_{w} \ar^{f}[d] \\ 
\partial T'_{s(v)}\ar_{(\psi_w\psi_v^{-1})_*}[r] & \partial T'_{s(w)} }
\]
commutes, so $f$ has equivalent restrictions at $v$ and~$w$.
\end{proof}

This immediately gives us the following result in the case that there are no isolated branches.

\begin{corollary}\label{cor:NoIsolatedBranchesCase}Let $T$ be a self-similar tree without dead ends or isolated branches.  Then the set $\R_T$ of rational homeomorphisms of~$T$ forms a group under composition, and this group isomorphic to~$\R_2$.  Indeed, the action of $\R_T$ on $\partial T$ is conjugate to the action of $\R_2$ on $\{0,1\}^\omega$.
\end{corollary}
\begin{proof}By Lemma~\ref{lem:Simplifying}, the rational homeomorphisms of $T$ are the same as the rational homeomorphisms of~$S[T]$.  But $S[T]$ is branching, so the conclusion follows from Theorem~\ref{thm:RationalGroupIsomorphism}.
\end{proof}

For a tree that does have isolated branches, we get the following result.

\begin{theorem}\label{thm:RationalGroupEmbedsInR2}Let $T$ be a self-similar tree without dead ends.  Then the set $\R_T$ of rational homeomorphisms of~$T$ forms a group under composition.  Moreover, the action of $\R_T$ on $\partial T$ is rational in the sense of Definition~\ref{def:RationalAction}, and in particular $\R_T$ embeds into~$\R_2$.
\end{theorem}
\begin{proof}From Lemma~\ref{lem:Expanding}, we know that a homeomorphism $f\colon \partial T\to\partial T$ is rational if and only if the induced homeomorphism $E[f]\colon \partial E[T]\to \partial E[T]$ is rational.  Moreover, note that $E[fg]= E[f]\,E[g]$ for all $f,g\in\R_T$ and $E[f^{-1}]=E[f]^{-1}$ for all $f\in\R_T$.  Since the rational homeomorphisms of $E[T]$ form a group by Corollary~\ref{cor:NoIsolatedBranchesCase}, it follows that the rational homeomorphisms of~$T$ form a group as well.

Now clearly the homomorphism $f\mapsto E[f]$ is an embedding of $\R_{E[T]}$ into~$\R_T$, and since $\R_{E[T]}$ is isomorphic to $\R_2$ by Corollary~\ref{cor:NoIsolatedBranchesCase} it follows that $\R_T$ embeds into $\R_2$.  Moreover, Corollary~\ref{cor:NoIsolatedBranchesCase} tells us that the action of $\R_{E[T]}$ on $\partial E[T]$ is conjugate to the action of $\R_2$ on $\{0,1\}^\omega$, so we can use $\R_{E[T]}$ instead of $\R_2$ to prove that the action of $\R_T$ on $\partial T$ is rational.  Let $q\colon \partial E[T]\to \partial T$ be the quotient map which is the identity on the non-isolated points of~$\partial T$ and maps $\{p\}\times\{0,1\}^\omega$ to $p$ for each isolated point $p$ of~$T$.  Then clearly the diagram
\[
\xymatrix{
\partial E[T] \ar_{q}[d] \ar^{E[f]}[r] & 
\partial E[T] \ar^{q}[d] \\ 
\partial T\ar_{f}[r] & \partial T}
\]
commutes for each $f\in \R_T$, and this proves that the action of $\R_T$ on $\partial T$ is rational in the sense of Definition~\ref{def:RationalAction}.
\end{proof}

Note that $\R_T$ need not be isomorphic to $\R_2$ if $T$ has isolated branches.  For example, if $T$ is a self-similar tree whose boundary $\partial T$ consists of finitely many isolated points, then $\R_T$ is a finite symmetric group.

Theorem~\ref{thm:RationalGroupEmbedsInR2} gives us the following test for whether an action is rational, which we use in Section~\ref{sec:HyperbolicRational} to prove Theorem~\ref{thm:MainTechnicalTheorem}, and hence Theorems~\ref{thm:MainTheorem} and~\ref{thm:MainTheoremAction}.

\begin{corollary}\label{cor:MainCorollaryIGuess}Let $G$ be a group acting on a compact metrizable space~$X$, and let $T$ be a self-similar tree without dead ends.  Suppose there exists a quotient map $q\colon \partial T\to X$ and a homomorphism $\varphi\colon G\to\R_T$ such that the diagram
\[
\xymatrix{
\partial T \ar_{q}[d] \ar^{\varphi(g)}[r] & 
\partial T \ar^{q}[d] \\ 
X\ar_{g}[r] & X}
\]
commutes for all $g\in G$.  Then the action of $G$ on $X$ is rational.\qed
\end{corollary}

\section{Rational Actions of Hyperbolic Groups}\label{sec:HyperbolicRational}

Let $\Gamma$ be a locally finite, connected hyperbolic graph, and let $G$ be a group acting by properly and cocompactly by isometries on~$\Gamma$.  For example, $G$ could be any hyperbolic group and $\Gamma$ could be its Cayley graph.  In this section, we construct a self-similar tree $T$ whose boundary $\partial T$ is naturally homeomorphic to the horofunction boundary~$\hb \Gamma$, and we show that the action of $G$ on $\hb\Gamma$ induces a rational action of $G$ on $\partial T$.

As in Section~1, we will identify $\Gamma$ with its set of vertices under the path metric~$d$. We fix a \newword{base vertex} $x_0\in \Gamma$, and define the \newword{length} $\ell(x)$ of any vertex $x\in\Gamma$ to be its distance from the base vertex. Let
\[
B_n = \{x\in\Gamma \mid \ell(x) \leq n\}
\qquad\text{and}\qquad
S_n = \{x\in\Gamma \mid \ell(x) = n\}
\]
denote the $n$-ball and $n$-sphere, respectively, centered at $x_0$.

If $x\in S_n$, a \newword{successor} to $x$ is any element of $S_{n+1}\cap C(x)$, where $C(x)$ denotes the cone on~$x$, and a \newword{predecessor} for $x$ is any element of $S_{n-1}$ for which $x$ is a successor.

\subsection{The Tree of Atoms}

\begin{definition}Let $B$ be a finite set of vertices in~$\Gamma$.
\begin{enumerate}
\item Given any point $x\in \Gamma$, the \newword{corresponding atom} is the set
\[
A = \bigl\{y \in \Gamma \;\bigl|\; \od_y\text{ agrees with }\od_x\text{ on }B\bigr\}.
\]
Let $\A(B)$ be the collection of all such atoms.
\item The \newword{shape algebra} for~$B$, denoted~$\S(B)$, is the algebra of sets generated by~$\A(B)$.
\end{enumerate}
\end{definition}

\begin{notes}\quad
\begin{enumerate}
\item The atoms in $\A(B)$ form a partition of~$\Gamma$, and therefore each element of $\S(B)$ is a disjoint union of atoms.  Thus the elements of $\A(B)$ are precisely the atoms in the Boolean algebra~$\S(B)$.
\item If $B\subseteq B'$ are finite sets, then each atom for $B'$ is a subset of some atom for~$B$. It follows that $\S(B) \subseteq \S(B')$.
\end{enumerate}
\end{notes}

Each atom $A$ comes with a function $\od_A\in \F(B,\mathbb{Z})$, which agrees with $\od_x$ on $B$ for each~$x\in A$.  We refer to this as the \newword{distance function for~$\boldsymbol{A}$.}

For example, Figure~\ref{fig:AtomsExample}(a) shows the four atoms derived from a certain three-point subset of~a hyperbolic graph.  The corresponding shape algebra has $16$~different sets, namely all possible disjoint unions of these four atoms.  Figure~\ref{fig:AtomsExample}(b) shows the distance function for each of the four atoms.
\begin{figure}
$\underset{\textstyle\text{(a)}}{\includegraphics{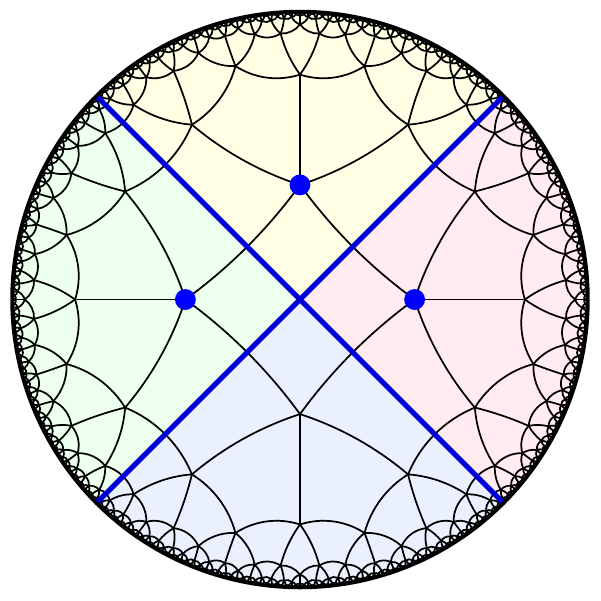}}$
\hfill
$\underset{\textstyle\text{(b)}}{\includegraphics{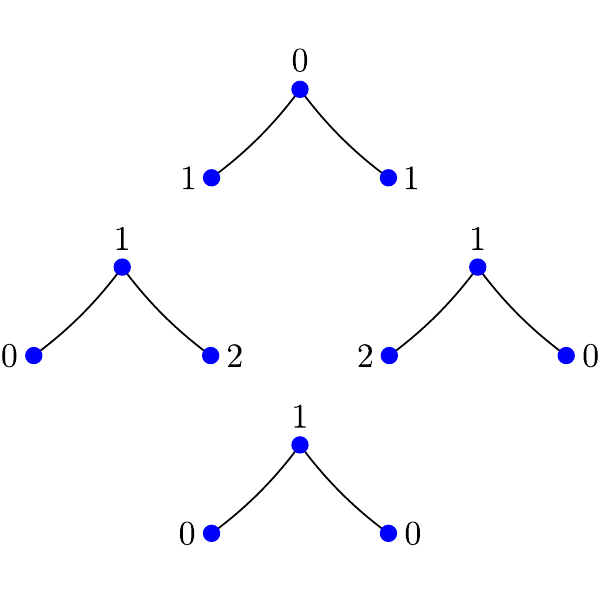}}$
\caption{(a)~Four atoms derived from a three-point subset $B$ of a hyperbolic graph~$\Gamma$. Each atom consists of the vertices of~$\Gamma$ that lie in the shown region. (b)~The distance functions $\od_A$ for~$A\in\A(B)$, with additive constants chosen so that the minimum value is~$0$ for each.}
\label{fig:AtomsExample}
\end{figure}

\begin{proposition}Each finite set $B\subseteq\Gamma$ has only finitely many different atoms.
\end{proposition}
\begin{proof}Fix a point $p\in B$.  By the triangle inequality, we know that
\[
\bigl|\od_x(q) - \od_x(p)\bigr| \leq d(p,q)
\]
for all $x\in\Gamma$ and $q\in B$.  But there are only finitely many different functions $f\colon B\to\mathbb{Z}$ satisfying $f(p) = 0$ and
\[
\bigl|f(q)\bigr| \leq d(p,q)
\]
for all $q\in B$, and therefore there can be only finitely many atoms
\end{proof}

Now if we fix a base vertex~$x_0\in\Gamma$, we get a sequence of balls in~$\Gamma$:
\[
\{x_0\} = B_0 \subseteq B_1 \subseteq B_2 \subseteq \cdots
\]
Taking the corresponding atoms gives us a sequence of partitions of~$\Gamma$:
\[
\A(B_0),\quad \A(B_1),\quad \A(B_2),\quad \ldots
\]
Each of these partitions is a refinement of the previous one, with $\A(B_0)$ having only only one atom, namely the whole graph~$\Gamma$.  For example, Figure~\ref{fig:SquareTilingAtomsUnlabelled} shows the atoms of $\A(B_1)$ and $\A(B_2)$ for the $1$-skeleton of the order five square tiling of the hyperbolic plane.
\begin{figure}
\centering
\includegraphics{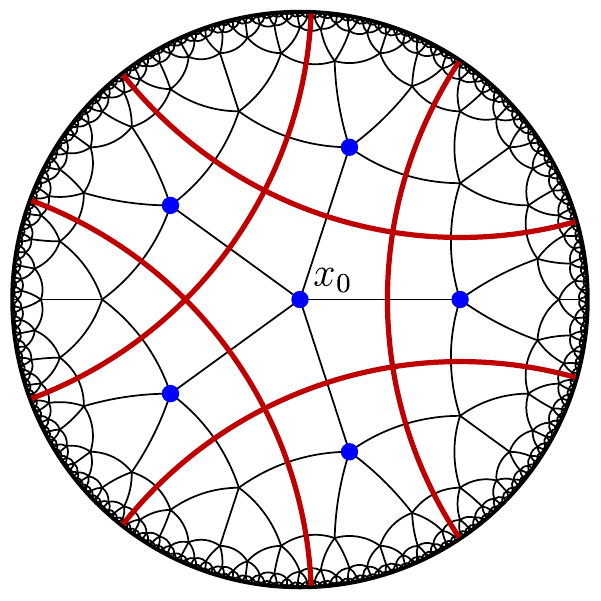}
\hfill
\includegraphics{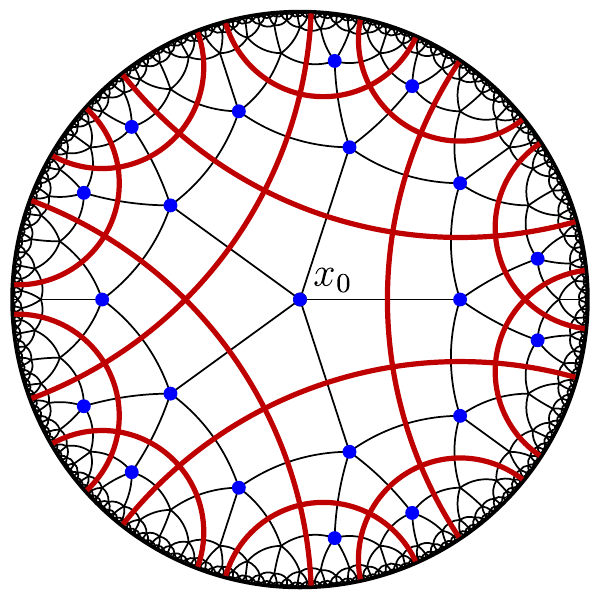}
\caption{The atoms of $\A(B_1)$ and $\A(B_2)$ for the order five square tiling of the hyperbolic plane.  There are $11$~atoms in $\A(B_1)$ and $36$~atoms in~$\A(B_2)$.}
\label{fig:SquareTilingAtomsUnlabelled}
\end{figure}

Because each $\A(B_n)$ is a refinement of the previous, the disjoint union
\[
\coprod_{n=0}^\infty \A(B_n)
\]
has the structure of a rooted tree.  It turns out that the boundary of this tree is naturally homeomorphic to~$\Gamma\cup \hb \Gamma$, where $\hb\Gamma$ is the horofunction boundary of~$\Gamma$.  That is, the boundary is homeomorphic to the closure of~$i(\Gamma)$ in~$\F(\Gamma,\Z)$ (see Definitions~\ref{def:CanonicalEmbedding} and~\ref{def:HorofunctionBoundary}).

Because we are interested in $\hb\Gamma$ specifically, we would like to restrict to a subtree whose boundary is precisely~$\hb\Gamma$.  As we will see, it suffices to consider only the atoms in each~$\A(B_n)$ that have infinite cardinality.  This motivates the following definition.

\begin{definition}For each $n\geq 0$, let $\A_n(\Gamma)$ be the set of infinite atoms in~$\A(B_n)$.  The \newword{tree of atoms} for~$\Gamma$ is the disjoint union.
\[
\A(\Gamma) = \coprod_{n=0}^\infty \A_n(\Gamma).
\]
\end{definition}

Note that, since $\A(\Gamma)$ is defined as a disjoint union, an element of $\A(\Gamma)$ is technically an ordered pair $(n,A)$, where $n\geq 0$ and~$A\in\A_n(\Gamma)$.  This distinction is sometimes relevant, for it is possible for the same set $A$ to be an atom in $\A_n(\Gamma)$ for two different values of~$n$.  However, we will often abuse notation and treat elements of $\A(\Gamma)$ as subsets of~$\Gamma$, with the understanding that each atom $A\in\A(\Gamma)$ knows which $\A_n(\Gamma)$ it comes from.

Note also that the tree $\A(\Gamma)$ has no dead ends.  In particular, since the union of the atoms in $\A_n(\Gamma)$ is the complement of a finite set in~$\Gamma$ and every atom in~$\A_n(\Gamma)$ is infinite, each atom in $\A_n(\Gamma)$ must contain at least one atom from~$\A_{n+1}(\Gamma)$.

The following proposition tells us that the atoms of $\A_n(\Gamma)$ move away from any finite set as~$n\to\infty$.

\begin{proposition}\label{prop:OnePointAtoms}If $n\geq 1$, then each one-point subset of $B_{n-1}$ is an atom in~$\A(B_n)$.  Thus every atom of $\A_n(\Gamma)$ is contained in the complement of~$B_{n-1}$.
\end{proposition}
\begin{proof}We say that a function $f\in F(\Gamma,\Z)$ has a \newword{local minimum} at a point $p\in\Gamma$ if $f(p) < f(q)$ for every vertex~$q$ that is adjacent to~$p$.  Note that adding an arbitrary constant to $f$ does not change the positions of the local minima, so it makes sense to talk about local minima for an element $\of\in\F(\Gamma,\Z)$.

Now, for any $x\in\Gamma$ it is not hard to see that the distance function $\od_x$ has a local minimum of~$x$, and this is the only local minimum for~$\od_x$ on~$\Gamma$.  If $x,y\in\Gamma$ and $x\ne y$, it follows that $\od_x$ and $\od_y$ cannot agree on any subset of~$\Gamma$ that contains~$x$ and all of its neighbors.  In particular, if $x\in B_{n-1}$, then $B_n$ contains $x$ and all of its neighbors, so $\od_x$ does not agree with $\od_y$ on $B_n$ for any $y\ne x$, which proves that $\{x\}$ is an atom in~$\A(B_n)$.
\end{proof}

We will prove the following theorem in the next section.

\begin{theorem}\label{thm:BoundaryIsHoroboundary}The boundary of the tree of atoms $\A(\Gamma)$ is naturally homeomorphic to the horofunction boundary~$\hb\Gamma$ of\/~$\Gamma$.
\end{theorem}

After this is done, the next task is to endow the tree $\A(\Gamma)$ with a self-similar structure.  First, if $A\in\A_n(\Gamma)$ and $m\geq n$, let
\[
\A_m(A) = \{A'\in \A_m(\Gamma) \mid A'\subseteq A\}.
\]
The union
\[
\A(\Gamma)_A = \coprod_{m\geq n} \A_m(A)
\]
is the subtree of atoms rooted at~$A$.

\begin{definition}Let $A_1\in \A_m(\Gamma)$ and $A_2\in\A_n(\Gamma)$ be atoms.  We say that an element $g\in G$ \newword{induces a morphism} from $A_1$ to $A_2$ if the following conditions are satisfied:
\begin{enumerate}
\item $gA_1 = A_2$.
\item $g(A_1\cap B_{m+k}) = A_2\cap B_{n+k}$ for all $k\geq 0$.
\item For each $k>0$ and each $A_1'\in\A_{m+k}(A_1)$, there exists an $A_2'\in \A_{n+k}(A_2)$ such that $gA_1'=A_2'$.
\end{enumerate}
The \newword{corresponding morphism} is the isomorphism $\varphi\colon\A(\Gamma)_{A_1}\to\A(\Gamma)_{A_2}$ of subtrees defined by condition~(3).  We say that $A_1$ and~$A_2$ have \newword{the same type} if there exists a morphism from $\A(\Gamma)_{A_1}$ to $\A(\Gamma)_{A_2}$.
\end{definition}

\begin{notes}\quad
\begin{enumerate}
\item Condition (2) is equivalent to saying that
\[
\ell(gp) - n = \ell(p) - m
\]
for all $p\in A_1$.
\item Clearly the composition of two morphisms is a morphism, and the inverse of a morphism is a morphism.  
\end{enumerate}
\end{notes}

\begin{proposition}For all $A_1,A_2\in \A(\Gamma)$, there are only finitely many morphisms $\A(\Gamma)_{A_1}\to \A(\Gamma)_{A_2}$.
\end{proposition}
\begin{proof}It suffices to prove that there are only finitely many morphisms from an atom $A\in\A(\Gamma)$ to itself.  Let $S$ be the set of elements of $A$ of minimum length (i.e.~minimum distance to the base vertex).  If $\varphi\colon \A(\Gamma)_A\to \A(\Gamma)_A$ is a morphism corresponding to an element~$g\in G$, then it follows from condition~(2) that $gS=S$.  Since $S$ is finite and $G$ acts properly on~$\Gamma$, there are only finitely many such~$g$, and therefore only finitely many morphisms from $A$ to~$A$.
\end{proof}

We will say that two atoms $A_1,A_2\in \A(\Gamma)$ have \newword{the same type} if there exists a morphism $\A(\Gamma)_{A_1}\to \A(\Gamma)_{A_2}$.  Unfortunately, it is not easy to prove that there are only finitely many types of atoms in~$\A(\Gamma)$, and indeed this is the first part of our theory that requires~$\Gamma$ to be hyperbolic.  After developing some geometric machinery in Sections~\ref{subsec:Neighbors and Visibility} and~\ref{subsec:MembershipTest}, we prove in Section~\ref{subsec:FinitelyManyTypes} that $\A(\Gamma)$ has only finitely many types, thereby endowing~$\A(\Gamma)$ with a self-similar structure.

Finally, we prove the following theorem in Section~\ref{subsec:ProofOfRationality}.

\begin{theorem}\label{thm:ActionIsRational}The group $G$ acts on the boundary of $\A(\Gamma)$ by rational homeomorphisms.
\end{theorem}

By Theorem~\ref{thm:BoundaryIsHoroboundary} and Corollary~\ref{cor:MainCorollaryIGuess}, it follows that $G$ acts rationally on the horofunction boundary~$\hb\Gamma$ (Theorem~\ref{thm:MainTechnicalTheorem}) and hence the Gromov boundary~$\partial \Gamma$ (Theorem~\ref{thm:MainTheoremAction}), and if this action is faithful then $G$ embeds in~$\R$ (Theorem~\ref{thm:MainTheorem}).  Therefore, we will have all of our chief results upon completing the proof of Theorem~\ref{thm:ActionIsRational}.

\subsection{Infinite Atoms and the Horofunction Boundary}

We begin by associating to each atom a certain subset of the horofunction boundary.

\begin{definition}Let $B\subseteq\Gamma$ be finite.  Given any atom $A\in\A(B)$, the \newword{shadow} of~$A$ is the set
\[
\partial A = \{\of \in \hb \Gamma \mid \of\text{ agrees with }\od_A\text{ on }B\}.
\]
\end{definition}

\medskip
The following proposition explains our interest in infinite atoms.

\medskip
\begin{proposition}\label{prop:ShadowsNonempty}Let $B\subseteq\Gamma$ be a finite set, and let $A\in\A(B)$.  Then the shadow $\partial A$ is nonempty if and only if $A$ is infinite.
\end{proposition}
\begin{proof}Suppose first that $\partial A$ is nonempty, and let $\of\in\partial A$.  Then by Proposition~\ref{prop:HorofunctionTest}, there exist infinitely many points $x\in\Gamma$ such that $\od_x$ agrees with $\of$ on~$B$, and it follows that $A$ is infinite.

Conversely, suppose that $A$ is infinite.  Let
\[
B = B_0\subseteq B_1\subseteq B_2\subseteq \cdots
\]
be an ascending chain of finite sets whose union is~$\Gamma$.  Then we can find a descending chain
\[
A = A_0 \supseteq A_1 \supseteq A_2 \supseteq \cdots,
\]
where each $A_n$ is an infinite atom in~$\A(B_n)$.  Then each $\od_{A_n}$ is a restriction of the next, and their union is a function $\of\in \F(\Gamma,\Z)$.  Clearly $\of$ agrees with $\od_A$ on~$B$.  We claim that $f$ is a horofunction.

Let $B'\subseteq\Gamma$ be any finite set. Since $\bigcup_n B_n=\Gamma$, there exists an $n\in\N$ so that $B'\subseteq B_n$.  Since $A_n$ is infinite, there exist infinitely many points~$x$ for which $\od_x$ agrees with $\of$ on~$B_n$, and it follows from Proposition~\ref{prop:HorofunctionTest} that $f$ is a horofunction.
\end{proof}

\begin{proposition}\label{prop:AtomsPartition}Let $B\subseteq\Gamma$ be finite.  Then the sets
\[
\{\partial A\mid A\in\A(B)\text{ and }A\text{ is infinite}\}
\]
are a partition of $\hb\Gamma$ into clopen sets.
\end{proposition}
\begin{proof}Since $\od_A$ and $\od_{A'}$ disagree on $B$ for any two $A,A'\in\A(B)$, the corresponding shadows are disjoint, and we know from Proposition~\ref{prop:ShadowsNonempty} that they are all nonempty.  To prove that the union of the shadows is all of $\hb\Gamma$, let $f\colon\Gamma\to\Z$ be a horofunction.  By Proposition~\ref{prop:HorofunctionTest}, there exists an $x\in\Gamma$ such that~$\od_x$ agrees with $\of$ on~$B$.  Let $A\in\A(B)$ be the atom containing~$x$.  Then $\od_x$ agrees with $\od_A$ on~$B$, so $\of$ agrees with $\od_A$ on~$B$, and therefore~$\of\in\partial A$.  Finally, observe that each $\partial A$ is open in~$\hb\Gamma$, since the preimage in $F(\Gamma,\mathbb{Z})$ is open in the product topology.  Since there are only finitely many~$\partial A$, it follows that each~$\partial A$ is also closed.
\end{proof}

\begin{proof}[Proof of Theorem~\ref{thm:BoundaryIsHoroboundary}] To show that $\partial\A(\Gamma)$ is homeomorphic to~$\hb\Gamma$, we begin by defining a function $h\colon\hb\Gamma\to \partial\A(\Gamma)$ as follows.  Given a point $\of\in \hb\Gamma$, we know from Proposition~\ref{prop:AtomsPartition} that there exists for each $n\in\N$ an atom $A_n\in \A_n(\Gamma)$ whose shadow contains~$\of$. Then for each $n$, the distance functions $\od_{A_{n}}$ and $\od_{A_{n+1}}$ must both agree with~$\of$ on~$B_n$, and therefore $\od_{A_n}$ is the restriction to~$B_n$ of~$\od_{A_{n+1}}$. It follows that the sequence $\{A_n\}$ is nested, with $A_0\supseteq A_1\supseteq \cdots$, so it corresponds to an infinite descending path in the tree~$\A(\Gamma)$.  Let $h(\of)$ be this path.  Note that $h$ is continuous since the sets $\{\partial \A(\Gamma)_A \mid A\in\A(\Gamma)\}$ form a basis for the topology on $\partial\A(\Gamma)$, and $h^{-1}(\partial \A(\Gamma)_A) = \partial A$ is open in $\hb\Gamma$ for each~$A$.

To prove that $h$ is bijective, let
\[
A_0 \supseteq A_1 \supseteq A_2 \supseteq \cdots.
\]
be any infinite descending path in~$\A(\Gamma)$, where $A_n\in\A_n(\Gamma)$ for each~$n$.  Then $\od_{A_{n+1}}$ agrees with $\od_{A_n}$ on~$B_n$ for each~$n$, and therefore
\[
\partial A_0 \supseteq \partial A_1 \supseteq \partial A_2 \supseteq \cdots.
\]
Since each $\partial A_n$ is closed and $\hb\Gamma$ is compact, the intersection $\bigcap_{n=0}^\infty \partial A_n$ contains at least one point.  This maps to $\{A_n\}$ under~$h$, and therefore $h$ is surjective.  Moreover, if $f$ and $f'$ are horofunctions such that $\of,\of'\in\bigcap_{n=0}^\infty A_n$,  then for each~$n$ both $\of$ and $\of'$ agree with~$\od_{A_n}$ on~$B_n$, and therefore $\of$ and $\of'$ agree with each other on~$B_n$.  Since this holds for every~$n$, it follows that $\of=\of'$, so the intersection $\bigcap_{n=0}^\infty A_n$ is a single point.  This proves that $h$ is injective and hence bijective.  Since $\hb\Gamma$ and $\partial\A(\Gamma)$ are compact Hausdorff spaces, it follows that $h$ is a homeomorphism.
\end{proof}



\subsection{Nearest Neighbors and Visibility}
\label{subsec:Neighbors and Visibility}

In this section we develop some geometric tools that will be essential in our proofs.

\begin{definition}Let $B\subseteq\Gamma$ be a finite set, and let $x\in\Gamma$.
\begin{enumerate}
\item A point $p\in B$ is called a \newword{nearest neighbor} for~$x$ if
\[
d(p,x)\leq d(q,x)
\]
for all $q\in B$.  We let $N(x,B)$ denote the set of nearest neighbors to $x$ in~$B$.
\item A point $p\in B$ is \newword{visible} from $x$ if
\[
d(p,x) < d(p,q) + d(q,x)
\]
for all $q\in B\setminus\{p\}$.  We let $V(x,B)$ denote the set of points in~$B$ that are visible from~$x$.
\end{enumerate}
\end{definition}

Equivalently, a point $p\in B$ is visible from $x$ if
\[
[p,x]\cap B = \{p\}
\]
for every geodesic $[p,x]$ from $p$ to~$x$.

Clearly $N(x,B) \subseteq V(x,B)$, i.e.~every nearest neighbor to $x$ in $B$ is visible from~$x$.  However, there are sometimes points in $B$ that are visible from $x$ but are not nearest neighbors, as shown in Figure~\ref{fig:NearestNeighborVisible}.
\begin{figure}
\centering
\includegraphics{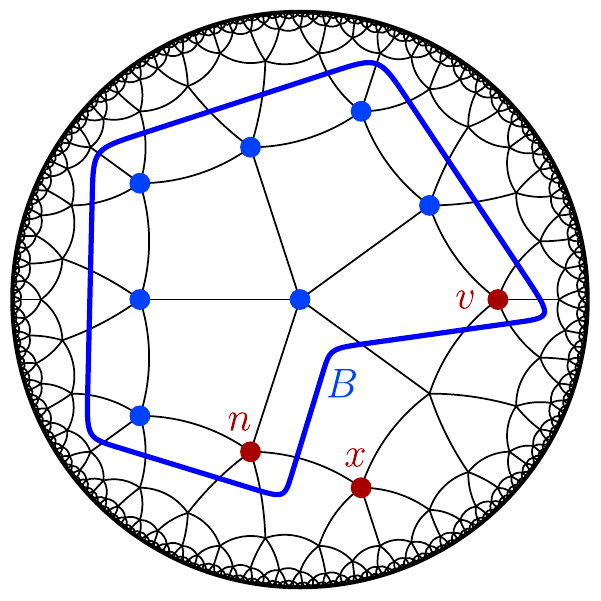}
\caption{A point $x$ and a finite subset $B$ of a hyperbolic graph~$\Gamma$.  Here $N(x,B) = \{n\}$ and $V(x,B) = \{n,v\}$.  In particular, $n$~is a nearest neighbor to~$x$, while $v$ is visible from~$x$ but is not a nearest neighbor.}
\label{fig:NearestNeighborVisible}
\end{figure}

\begin{proposition}\label{prop:SameNearestNeighbors}Let $B\subseteq\Gamma$ be a finite set, and let $x,y\in\Gamma$ be points that lie in the same atom of~$\A(B)$.  Then
\[
N(x,B)=N(y,B)\qquad\text{and}\qquad V(x,B)=V(y,B).
\]
\end{proposition}
\begin{proof}For the first statement, observe that a point $p\in B$ lies in $N(x,B)$ if and only if
\[
d_x(q) - d_x(p) \geq 0
\]
for all $q\in B$.  But $\od_x$ and $\od_y$ agree on~$B$, so
\[
d_x(q) - d_x(p) = d_y(q) - d_y(p)
\]
for all $x,y\in B$, which implies that $N(x,B) = N(y,B)$.

A similar argument holds for $V(x,B)$.  Specifically, a point $p\in B$ lies in $V(x,B)$ if and only if
\[
d_x(p) - d_x(q) < d(p,q)
\]
for all $q\in B$.  Again, the quantity on the left depends only on~$\od_x$, so it follows that $V(x,B) = V(y,B)$.
\end{proof}

If $B\subseteq\Gamma$ is a finite set and $A\in\A(B)$ we will let
\[
N(A,B)\qquad\text{and}\qquad V(A,B)
\]
denote the set of points in $B$ that are nearest neighbors to points in $A$ and visible from points in~$A$, respectively.  That is, $N(A,B) = N(x,B)$ and $V(A,B) = V(x,B)$ for all $x\in A$.  This is well-defined by Proposition~\ref{prop:SameNearestNeighbors}.

Though nearest neighbors are fairly natural, visibility may see like a strange concept.  The following proposition explains our interest in visibility.

\begin{proposition}\label{prop:VisibleGeodesic}Let $B\subseteq \Gamma$ be a finite set, let $p\in B$, and let $x\in\Gamma$.  Then there exists a geodesic from $p$ to $x$ that contains a point of~$V(x,B)$.
\end{proposition}
\begin{proof}Let $(q_1,\ldots,q_n)$ be any geodesic from $q_1=p$ to $q_n=x$.  Let $q_i$ be the last point in this geodesic that lies in~$B$.  If $q_i$ is visible from $x$ we are done.  Otherwise, we can replace $(q_i,q_{i+1},\ldots,q_n)$ with another geodesic from $q_i$ to $x$ that intersects $B$ at a later point.  Continuing in this fashion, we eventually obtain a geodesic that contains a point of~$V(x,B)$.
\end{proof}

The proposition above is actually the defining property of $V(x,B)$.  Specifically, if $B\subseteq \Gamma$ is a finite set, $x\in\Gamma$, and $P\subseteq B$, the following are equivalent:
\begin{enumerate}
\item $V(x,B)\subseteq P$.
\item For every point $p\in B$, every geodesic from $p$ to $x$ contains a point of~$P$.
\end{enumerate}
Condition (1) implies condition (2) by Proposition~\ref{prop:VisibleGeodesic}, while condition~(2) implies condition~(1) since every visible point $p\in B$ has a geodesic $[p,x]$ that intersects $B$ only at~$p$.

We can use Proposition~\ref{prop:VisibleGeodesic} to obtain a useful test for whether a given point lies in a given atom.

\begin{lemma}\label{lem:Reconstructdx}Let $B\subseteq\Gamma$ be finite, let $x\in\Gamma$, and let $P\subseteq B$ be any set that contains $V(x,B)$.  Then for all~$b\in B$,
\[
d(b,x) = \min_{p\in P} \bigl[d(b,p) + d(p,x)\bigr].
\]
\end{lemma}
\begin{proof}Let $b\in B$, and let $m$ be the given minimum. By the triangle in equality, $d(b,x) \leq d(b,p) + d(p,x)$ for all $p\in P$, and therefore $d(b,x)\leq m$.  For the reverse inequality, Proposition~\ref{prop:VisibleGeodesic} tells us that there exists a geodesic $[b,x]$ from $b$ to $x$ that intersects~$V(x,B)$.  Let $p\in [b,x]\cap V(x,B)$.  Then $p\in P$ and $d(b,p) + d(p,x) = d(b,x)$, which proves that $m\leq d(b,x)$.
\end{proof}

\begin{proposition}\label{prop:VisibilityTest}Let $B\subseteq \Gamma$ be a finite set. Let $x\in\Gamma$ and $A\in\A(B)$, and let $P\subseteq B$ be any set containing $V(x,B)\cup V(A,B)$.  Then $x\in A$ if and only if $\od_x$ agrees with $\od_A$ on~$P$.
\end{proposition}
\begin{proof}If $x\in A$, then $\od_x$ and $\od_A$ must agree on all of~$B$, and therefore $\od_x$ agrees with $\od_A$ on~$P$.  For the converse, suppose that $\od_x$ agrees with $\od_A$ on~$P$.  Let~$y\in A$, and note that $V(y,B) = V(A,B) \subseteq P$ and $\od_x$ and $\od_y$ agree on~$P$.  By Lemma~\ref{lem:Reconstructdx}, we know that
\[
d_x(b) = \min_{p\in P}\bigl[d(b,p) + d_x(p)\bigr]
\qquad\text{and}\qquad d_y(b) = \min_{p\in P}\bigl[d(b,p) + d_y(p)\bigr]
\]
for all $p\in B$. Since $\od_x$ and $\od_y$ agree on $P$, it follows that $\od_x$ and $\od_y$ agree on all of~$B$.  Then $\od_x$ agrees with $\od_A$ on~$B$, so $x\in A$.
\end{proof}

\subsection{A Membership Test for Atoms}
\label{subsec:MembershipTest}

In this section we use the notions of nearest neighbors and visibility to obtain a useful test for whether a given point $x\in\Gamma$ lies in a given atom of~$\A(\Gamma)$.

Recall that $B_n$ denotes the $n$-ball in $\Gamma$ centered at the base vertex~$x_0$.  Note that if $x$ is any point in~$\Gamma$ with $\ell(x)\geq n$, then $N(x,B_n)$ and $V(x,B_n)$ are both subsets of the $n$-sphere~$S_n$.  If particular, $N(x,B_n)$ is precisely the set of points in $S_n$ that lie on geodesics from $x_0$ to~$x$.

The following fundamental proposition helps us pin down the locations of the points in~$V(x,B_n)$.  Though we assumed at the beginning of Section~\ref{sec:HyperbolicRational} that $\Gamma$ is $\delta$-hyperbolic, the following proposition is the first time we make use of this fact.

\begin{proposition}\label{prop:MaxAngularDistance}Let $x\in\Gamma$, let $n\leq \ell(x)$, and let $p\in V(x,B_n)$.  Let $x_0,x_1,\ldots,x_{\ell(x)}$ be a geodesic from $x_0$ to $x$ and let $p_0,p_1,\ldots,p_n$ be a geodesic from $x_0$ to~$p$.  Then
\[
d(x_i,p_i)\leq 4\delta+2
\]
for all\/~$0\leq i\leq n$.
\end{proposition}
\begin{proof}Let $0\leq i \leq n$.  If $i\leq 2\delta+1$ we are done, so suppose that $i> 2\delta+1$.  Let $[p,x]$ be a geodesic from $p$ to~$x$. Since $p$ is visible from~$x$, this intersects $B_n$ only at~$p$. Let $j = \min(n-1-\delta,i)$ and note that $j < n-\delta$, which means that the point~$p_j$ does not lie within $\delta$ of any point on $[p,x]$.  Since $\Gamma$ is $\delta$-hyperbolic, it follows that $d(p_j,x_k)\leq \delta$ for some~$k$.  
Note then that $|j-k|\leq \delta$, so
\[
|i-k| \leq |i-j| + |j-k| \leq (\delta + 1) + \delta = 2\delta+1.
\]
Then
\begin{multline*}
\qquad d(x_i,p_i) \leq d(x_i,x_k) + d(x_k,p_j) + d(p_j,p_i) \leq |i-k| + \delta + |i-j| \\[6pt]
\leq (2\delta+1)+\delta + (\delta+1) = 4\delta+2.\tag*{\qedhere}
\end{multline*}
\end{proof}

This proposition motivates the following definition.

\begin{definition}\label{def:Proximal}Let $x\in\Gamma$, let $n\leq \ell(x)$, and let $p\in S_n$.  We say that $p$ is \newword{proximal} to~$x$ if there exists a geodesic $p_0,p_1,\ldots,p_n$ from $x_0$ to $p$ such that
\[
d(x_i,p_i) \leq 4\delta + 2
\]
for all $0\leq i\leq n$ and every geodesic $x_0,x_1,\ldots,x_m$ from $x_0$ to~$x$.  We let $P(x,S_n)$ be the set of points in $S_n$ that are proximal to~$x$.
\end{definition}

By Proposition~\ref{prop:MaxAngularDistance}, every point in $B_n$ that is visible to~$x$ must be proximal to~$x$, i.e.~$V(x,B_n) \subseteq P(x,S_n)$.  

The advantage of $P(x,S_n)$ over $V(x,B_n)$ is that there is a useful inductive test to check whether a point $p\in S_n$ is proximal to~$x$, as shown in the following proposition.

\begin{proposition}\label{prop:ProximalInductive}Let $n\geq 1$, let $x\in \Gamma$ with $\ell(x)\geq n$, and let $p\in S_n$.  Then $p$ is proximal to $x$ if and only if
\begin{enumerate}
\item $p$ has a predecessor that is proximal to~$x$, and
\item $d(p,q)\leq 4\delta+2$ for all $q\in N(x,B_n)$.
\end{enumerate}
\end{proposition}
\begin{proof}Suppose that $p$ is proximal to $x$, and let $p_0,p_1,\ldots,p_n$ be a geodesic from $x_0$ to $p$ satisfying the definition of proximality.  Then $p_{n-1}$ is a predecessor for~$p_n$ and is clearly proximal, with $p_0,p_1,\ldots,p_{n-1}$ being the required geodesic.  Furthermore, if $q$ is a nearest neighbor to $x$ in $B_n$, then there must exist a geodesic $x_0,\ldots,x_{\ell(x)}$ from $x_0$ to $x$ such that $x_n=q$, and it follows that $d(q,p) = d(x_n,p_n) \leq 4\delta+2$.

For the converse, suppose that $p$ satisfies conditions~(1) and~(2).  Let $p_{n-1}$ be a predecessor to $p$ that is proximal to~$x$, and let $p_0,\ldots,p_{n-1}$ be a geodesic from $x_0$ to $p_{n-1}$ satisfying the definition of proximality. Note then that $p_0,\ldots,p_{n-1},p$ is a geodesic from $x_0$ to~$p$.  Let $x_0,\ldots,x_{\ell(x)}$ be a geodesic from $x_0$ to~$x$.  We know that $d(x_i,p_i)\leq 4\delta+2$ for all $0\leq i\leq n-1$.  Moreover, the point $x_n$ must lie in $N(x,B_n)$, and therefore $d(p,x_n)\leq 4\delta+2$, which proves that $p$ is proximal to~$x$.
\end{proof}

\begin{corollary}\label{prop:DiameterOfP}If $x\in\Gamma$ and $n\leq \ell(x)$, then $P(x,S_n)$ has diameter at most\/~$8\delta+4$.\qed
\end{corollary}

The proposition above showed that $P(x,S_n)$ is in some sense determined by nearest neighbors.  The following proposition carries this idea further.

\begin{proposition}\label{prop:ContainmentND}Let $n\geq 0$, let $x,y\in\Gamma$ with $\ell(x)\geq n$ and $\ell(y)\geq n$, and suppose that $N(x,B_n)\subseteq N(y,B_n)$. Then $P(y,S_n)\subseteq P(x,S_n)$.
\end{proposition}
\begin{proof}We use induction on $n$.  The statement is obvious for $n=0$, since $P(x,S_0)=P(y,S_0) = \{x_0\}$.  For $n>0$, suppose that $N(x,B_n)\subseteq N(y,B_n)$, and let $p\in P(y,S_n)$.  We know from Proposition~\ref{prop:ProximalInductive} that $p$ has a predecessor $p'\in P(y,S_{n-1})$.  But $N(x,B_{n-1})\subseteq N(y,B_{n-1})$, since $N(x,B_{n-1})$ consists of all predecessors of elements of $N(x,B_n)$, and similarly for~$y$.  From our induction hypothesis, it follows that $P(y,S_{n-1})\subseteq P(x,S_{n-1})$, so $p'\in P(x,S_{n-1})$.  But $d(p,q)\leq 4\delta+2$ for all $q\in N(y,B_n)$, and hence $d(p,q)\leq 4\delta+2$ for all $q\in N(x,B_n)$, which proves that $p\in P(x,S_n)$ by Proposition~\ref{prop:ProximalInductive}.
\end{proof}

\begin{corollary}\label{Cor:WellDefinedPn}Let $n\geq 0$, and let $x,y\in \Gamma$ with $\ell(x)\geq n$ and $\ell(y)\geq n$.  If $x$ and $y$ lie in the same atom of $\A(B_n)$, then 
\[
P(x,S_n) = P(y,S_n)
\]
\end{corollary}
\begin{proof}By Proposition~\ref{prop:SameNearestNeighbors} we know that $N(x,B_n) = N(y,B_n) = N(A,B_n)$, so the result follows from Proposition~\ref{prop:ContainmentND}.
\end{proof}

If $A\in\A_n(\Gamma)$, we will let $P(A,S_n)$ denote the set of points in $S_n$ that are proximal to points in~$A$, i.e.~$P(A,S_n) = P(x,S_n)$ for any~$x\in A$.  This is well-defined by Corollary~\ref{Cor:WellDefinedPn}.

We now show how to use proximal points to determine whether a given point lies in a given atom.

\begin{proposition}\label{prop:MembershipAtom}Let $x\in\Gamma$ and $A\in\A_n(\Gamma)$. Then $x\in A$ if and only if
\begin{enumerate}
\item $x\in C(p)$ for all $p\in N(A,B_n)$, and
\item $\od_x$ agrees with $\od_A$ on $P(A,S_n)$.
\end{enumerate}
\end{proposition}
\begin{proof}Clearly $x$ satisfies the given conditions if it lies in~$A$.  For the converse, suppose that $x$ satisfies the given conditions.  Since
\[
N(x,B_n) = \{p\in S_n \mid x \in C(p)\},
\]
it follows from condition~(1) that $N(A,B_n) \subseteq N(x,B_n)$, so by Proposition~\ref{prop:ContainmentND} we have $P(x,S_n) \subseteq P(A,S_n)$.  Then
\[
N(x,B_n)\subseteq V(x,B_n) \subseteq P(x,S_n) \subseteq P(A,S_n)
\]
so $N(x,B_n) = N\bigl(x,P(A,S_n)\bigr)$.  (In general $N(x,B) = N(x,B')$ whenever $B'\subseteq B\subseteq\Gamma$ and $N(x,B)\subseteq B'$.) But similarly,
\[
N(A,B_n)\subseteq V(A,B_n) \subseteq P(A,S_n)
\]
so $N(A,B_n) = N\bigl(A,P(A,S_n)\bigr)$. But $\od_x$ agrees with $\od_A$ on $P(A,S_n)$ by condition~(2), which implies that $N\bigl(x,P(A,S_n)\bigr)=N\bigl(A,P(A,S_n)\bigr)$, and hence $N(x,B_n) = N(A,B_n)$. Then $P(x,S_n) = P(A,S_n)$ by Proposition~\ref{prop:ContainmentND}.  By Proposition~\ref{prop:MaxAngularDistance}, this set contains $V(x,B_n) \cup V(A,B_n)$, so it follows from Proposition~\ref{prop:VisibilityTest} that~$x\in A$.
\end{proof}

\subsection{Finitely Many Types}
\label{subsec:FinitelyManyTypes}

In this section we exploit our membership test for atoms (Proposition~\ref{prop:MembershipAtom}) to prove that the tree $\A(\Gamma)$ has only finitely many different types of atoms.  The difficult part here is to construct a sufficient number of morphisms between atoms, so the main technical result for this section is a test for whether a given element of~$G$ induces a morphism between two given atoms.

First we introduce a little notation.  Given a nonempty subset $P\subseteq \Gamma$, a function $f\in F(P,\Z)$, and a group element $g\in G$, let $gf \in F(gP,\mathbb{Z})$ be the function defined by
\[
(gf)(p) = f\bigl(g^{-1}p\bigr)
\]
for all $p\in gP$.  Note that if $f$ and $f'$ differ by a constant, then $gf$ and $gf'$ differ by a constant, so for any $\of\in\F(P,\Z)$ we get a well-defined $g\of\in\F(gP,\Z)$.

\begin{definition}\label{def:GeometricEquivalence}Let $A\in\A_m(\Gamma)$ and $A'\in\A_n(\Gamma)$.  We say that an element $g\in G$ induces a \newword{geometric equivalence} from $A$ to $A'$ if
\begin{enumerate}
\item $g\,P(A,S_m) = P(A',S_n)$,
\item $g\,\od_{A}$ agrees with $\od_{A'}$ on $P(A',S_n)$, and
\item $g\,C(p) = C(gp)$ for all $p\in P(A,P_m)$.
\end{enumerate}
\end{definition}

\medskip Here is the main technical result for this section:

\begin{proposition}\label{prop:MainMorphismTest}Let $A\in \A_m(\Gamma)$ and $A'\in \A_n(\Gamma)$, and let $g\in G$.  If $g$ induces a geometric equivalence from $A$ to~$A'$, then $g$ induces a morphism from $A$ to~$A'$.
\end{proposition}

As we will see in Section~\ref{sec:AnExample} (Note~\ref{note:NotGeometricallyEquivalent}), the converse of this proposition is not true, so this test does not allow us to detect all morphisms of~$\A(\Gamma)$.  However, it does provide us with a sufficient number of morphisms to prove the following result.

\begin{corollary}\label{cor:FinitelyManyTypes}The tree $\A(\Gamma)$ has finitely many types of atoms.
\end{corollary}
\begin{proof}Since $P(A,S_m)$ has diameter at most $8\delta+4$ (Corollary~\ref{cor:FinitelyManyTypes}), there are only finitely many possibilities for $P(A,S_m)$ modulo the action of~$G$.  Moreover, since the action of $G$ is proper and $\Gamma$ has finitely many cone types, there are only finitely many choices for $C(p)$ for each $p\in P(A,S_m)$, and there are only finitely many choices for the restriction of $\od_A$ to $P(A,S_m)$.  By Proposition~\ref{prop:MainMorphismTest}, two atoms with corresponding choices have the same type, and therefore there are only finitely many types of atoms.
\end{proof}

The rest of this section is devoted to a proof of Proposition~\ref{prop:MainMorphismTest}.

\begin{lemma}\label{lem:GeometricEquivalenceMapsAtoms}Let $A\in\A_m(\Gamma)$ and $A'\in\A_n(\Gamma)$, and let $g\in G$.  If $g$ induces a geometric equivalence from $A$ to~$A'$, then:
\begin{enumerate}
\item $g\,N(A,B_m) = N(A',B_n)$,
\item $gA=A'$, and
\item $g(A\cap B_{m+k})=A'\cap B_{n+k}$ for all $k\geq 0$.
\end{enumerate}
\end{lemma}
\begin{proof}For statement~(1), observe that the points of $N(A,B_m)$ are precisely the points at which $\od_A$ achieves its minimum value on~$B_m$.  Indeed, since $N(A,B_m)\subseteq P(A,S_m)$, the points of $N(A,B_m)$ are precisely the points of $P(A,S_m)$ on which the restriction of $\od_A$ achieves its minimum value.  Since $g\,P(A,S_m) = P(A',S_n)$ and $g\,\od_A$ agrees with $\od_{A'}$ on $P(A',S_n)$, it follows that $g\,N(A,B_m)$ is precisely the set of points of $P(A',S_n)$ on which the restriction of $\od_{A'}$ achieves its minimum value, and therefore $g\,N(A,B_m) = N(A',B_n)$.

Statement (2) uses our membership test for atoms (Proposition~\ref{prop:MembershipAtom}).  Let $x\in A$, and observe that $x\in C(p)$ for all $p\in N(A,B_m)$.  Since $g$ is a geometric equivalence and $N(A,B_m) \subseteq P(A,S_m)$, we know that $g\,C(p) = C(gp)$ for all $p\in N(A,B_m)$, so $gx\in C(gp)$ for all $p\in N(A,B_m)$.  By statement~(1), we know that $g\,N(A,B_m) = N(A',B_n)$, and therefore $gx \in C(q)$ for all $q\in N(A',B_n)$.  Moreover, since $\od_x$ agrees with $\od_A$ on $P(A,S_m)$, we see that $\od_{gx} = g\,\od_x$ agrees with $g\,\od_A = \od_{A'}$ on $P(A',S_n)$, and therefore $gx\in A'$ by Proposition~\ref{prop:MembershipAtom}. This proves that $gA\subseteq A'$, and the reverse inclusion follows from the fact that $g^{-1}$ is a geometric equivalence from $A'$ to~$A$.

For statement~(3), let $k> 0$ and fix any point $p\in N(A,B_m)$.  Then $p\in N(x,B_m)$ for all $x\in A$, which means that for each $x\in A$ there exists a geodesic from $x_0$ to $x$ that goes through~$p$.  Since $\ell(p)=m$, it follows that
\[
A\cap B_{m+k} = \{x\in A \mid d(x,p) =k\}.
\]
Moreover, $gp\in N(A',B_n)$ by statement~(1), so similarly
\[
A'\cap B_{n+k} = \{y\in A' \mid d(y,gp) = k\}.
\]
Since $gA=A'$ by statement~(2), it follows that $g(A\cap B_{m+k}) = A'\cap B_{n+k}$.
\end{proof}

\begin{lemma}\label{lem:SuccessorsAndCones}Let $p\in\Gamma$, let $p'$ be a successor to $p$, and let $g\in G$.  If $g\,C(p) = C(gp)$, then $gp'$ is a successor to~$gp$, and $g\,C(p') = C(gp')$.
\end{lemma}
\begin{proof}Note first that, for any point $p\in\Gamma$, the set of successors to $p$ is precisely the set of points in $C(p)$ that are adjacent to~$p$. Moreover, if $p'$ is a successor to $p$, then
\[
C(p') = \{x\in C(p) \mid d(x,p) = d(x,p')+1\}.
\]
Since $G$ acts by isometries, the lemma follows immediately.
\end{proof}

If $x\in \Gamma$ and $n\in\mathbb{N}$, let $A_n(x)$ denote the atom of $\A(B_n)$ that contains~$x$.  Note that
\[
A_0(x) \supseteq A_1(x)\supseteq A_2(x)\supseteq\cdots
\]
and that $A_n(x) \in \A_n(\Gamma)$ if and only if $A_n(x)$ is infinite.

\begin{lemma}\label{lem:InductionStep}Let $x\in\Gamma$, let $m,n\geq 1$, and let $g\in G$.  Suppose that $A_m(x)$ and $A_n(gx)$ are both infinite, and suppose that $g$ induces a geometric equivalence from $A_{m-1}(x)$ to $A_{n-1}(gx)$.  Then $g$ induces a geometric equivalence from $A_m(x)$ to~$A_n(gx)$
\end{lemma}
\begin{proof}Let $y=gx$.  Since $A_m(x)$ and $A_n(y)$ are both infinite, we know from Proposition~\ref{prop:OnePointAtoms} that $\ell(x)\geq m$ and~$\ell(y)\geq n$.  Let $X$ be the set of all successors of elements of~$P(x,S_{m-1})$, and let $Y$~be the set of all successors of elements of~$P(y,S_{n-1})$.

We have $g\,P(x,S_{m-1}) = P(y,S_{n-1})$ as $g$ is a geometric equivalence between $A_{m-1}(x)$ and $A_{n-1}(y)$.  Also, $g\,C(p) = C(gp)$ for every point $p\in P(x,S_{m-1})$.  By Lemma~\ref{lem:SuccessorsAndCones}, it follows that $gX=Y$ and $g\,C(p) = C(gp)$ for all $p\in X$.  In particular, since $P(x,S_m)\subseteq X$ by Proposition~\ref{prop:ProximalInductive}, we know that $g\,C(p)=C(gp)$ for all $p\in P(x,S_m)$.

Next observe that $N(x,B_m) \subseteq V(x,B_m) \subseteq P(x,S_m) \subseteq X$, and similarly for~$y$. It follows that $N(x,B_m) = N(x,X)$ and $N(y,B_n) = N(y,Y)$.  Since $gx=y$ and $gX=Y$, it we conclude that $g\,N(x,B_m) = N(y,B_n)$.  But
\[
P(x,S_m) = \{p\in X \mid d(p,q)\leq 4\delta+2\text{ for all }q\in N(x,B_m)\}
\]
by Proposition~\ref{prop:ProximalInductive}, and similarly for~$y$, so $g\,P(x,S_m) = P(y,S_n)$.

Finally, since $g\,\od_x = \od_y$ and $\od_x$ agrees with $\od_{A_m(x)}$ on $P(x,S_m)$ and $\od_y$ agrees with $\od_{A_n(y)}$ on $P(y,S_n)$, it follows that $g\,\od_{A_m(x)}$ agrees with $\od_{A_n(y)}$ on $P(y,S_n)$, and therefore $g$ is a geometric equivalence from $A_m(x)$ to~$A_n(y)$.
\end{proof}

\begin{proof}[Proof of Proposition~\ref{prop:MainMorphismTest}] Let $A\in\A_m(\Gamma)$, let $A'\in\A_n(\Gamma)$, and let $g\in G$ be a geometric equivalence from $A$ to $A'$.  We must show that $g$ induces a morphism from $\A(\Gamma)_A$ to $\A(\Gamma)_A'$.

Note first that $gA=A'$ and $g(A\cap B_{m+k}) = A'\cap B_{n+k}$ for each $k\geq 0$ by Lemma~\ref{lem:GeometricEquivalenceMapsAtoms}.  We claim that for $k\geq 0$ and each $A_k\in \A_{m+k}(A)$, there exists an $A_k'\in \A_{n+k}(A')$ such that $g$ induces a geometric equivalence from $A_k$ to~$A_k'$.  By Lemma~\ref{lem:GeometricEquivalenceMapsAtoms}, it will follow that $gA_k=A_k'$, which will prove that $g$ induces a morphism.

We proceed by induction on~$k$.  We know the statement holds for $k=0$, since $A_k=A$ and $A_k'=A'$ in this case.  For $k>0$, since $A_k$ is infinite and the union of the atoms in $\A_{n+k}(A')$ is the complement of a finite set in~$A'$, we can choose a point $x\in A_k$ such that $gx$ lies in some infinite atom $A_k'\in\A_{n+k}(A')$.  Let $y=gx$, and let $A_{k-1}=A_{m+k-1}(x)$.  By our induction hypothesis, there exists an atom $A'_{k-1}\in\A_{n+k-1}(A')$ such that $g$ induces a geometric equivalence from $A_{k-1}$ to $A_{k-1}'$.  By Lemma~\ref{lem:GeometricEquivalenceMapsAtoms}, we know that $gA_{k-1}=A_{k-1}'$, so $A_{k-1}'=A_{n+k-1}(y)$.  By Lemma~\ref{lem:InductionStep}, we conclude that that $g$ induces a geometric equivalence from $A_{m+k}(x) = A_k$ to $A_{n+k}(y) = A_k'$, which completes the induction.
\end{proof}

\subsection{Proof of Rationality}
\label{subsec:ProofOfRationality}

Our goal in this section is to prove Theorem~\ref{thm:ActionIsRational}, i.e.~that the action of $G$ on $\partial\A(\Gamma)$ is rational.

If $g\in G$, we define the \newword{magnitude} of $g$, denoted $|g|$, to be the distance from the base vertex $x_0$ to $gx_0$.  For example, if $\Gamma$ is the Cayley graph of~$G$, then $|g|$ is the word length of~$g$.  Note that $|g| = |g^{-1}|$ for all $g\in G$.

\begin{proposition}\label{prop:LipschitzAction}Let $g\in G$ with $|g|=k$, and let $A\in\A_n(\Gamma)$ for some $n\geq k$.  Then there exists a unique $A'\in\A_{n-k}(\Gamma)$ so that~$gA\subseteq A'$.
\end{proposition}
\begin{proof}Observe that
\[
|\ell(gx)-\ell(x)| \leq |g| = k
\]
for all $x\in\Gamma$.  It follows easily that $B_{n-k} \subseteq gB_n$.  Since $A\in\A(B_n)$, we know that $gA\in\A(gB_n)$, and therefore $gA\subseteq A'$ for some $A'\in \A(B_{n-k})$.  This $A'$ is unique since the atoms of $\A(B_{n-k})$ are disjoint.
\end{proof}

\begin{note}It follows immediately from this proposition that the homeomorphism of $\partial\A(\Gamma)$ induced by an element $g\in G$ is Lipschitz with respect to the standard ultrametric, with Lipschitz constant $2^{|g|}$.  Indeed, since $g^{-1}$ is Lipschitz as well, it follows that the homeomorphism induced by $g$ is bilipschitz.\qed
\end{note}

Proposition~\ref{prop:LipschitzAction} prompts the following definition.

\begin{definition}A \newword{mapping triple} is an ordered triple $(g,A,A')$, where
\begin{enumerate}
\item $g\in G$,
\item $A\in\A_n(\Gamma)$ for some $n\geq |g|$, and
\item $A'$ is the atom from $\A_{n-|g|}(\Gamma)$ that contains~$gA$.
\end{enumerate}
\end{definition}

By Proposition~\ref{prop:LipschitzAction}, for every $g\in G$ and every $A\in \A_n(\Gamma)$ where $n\geq |g|$, there exists a unique $A'\in \A_{n-|g|}(\Gamma)$ such that $(g,A,A')$ is a mapping triple.

\begin{proposition}\label{prop:DisksAreNear}Let $(g,A,A')$ be a mapping triple, where $|g|=k$ and $A\in\A_n(\Gamma)$. Then there exist points $p\in P(A,S_n)$ and $q\in P(A',S_{n-k})$ so that $d(gp,q)\leq 2k$.
\end{proposition}
\begin{proof}Fix a point $x\in A$, and note that $gx\in A'$.  Since $n\geq k$, we know that $g^{-1}x_0\in B_n$.  By Proposition~\ref{prop:VisibleGeodesic}, there exists a geodesic $[g^{-1}x_0,x]$ from $g^{-1}x_0$ to $x$ that goes through a point $p\in V(A,B_n)$. Let $[x_0,gx]$ be the image of this geodesic under~$g$, and let $q\in N(A',B_{n-k})$ be the point at which this geodesic crosses~$S_n$. Note that $p\in P(A,S_n)$ and $q\in P(A',S_{n-k})$.  Since $\ell(p) = n$, we know that $n-k \leq \ell(gp) \leq n+k$.  Since $\ell(q)=n-k$ and $gp,q\in [x_0,gx]$, it follows that $d(gp,q)\leq 2k$.
\end{proof}

If $g\in G$, let $L_g\colon \partial\A(\Gamma)\to\partial\A(\Gamma)$ be the homeomorphism induced by~$g$.

\begin{proposition}\label{prop:EquivalentRestrictions}Let $(g,A_1,A_1')$ and $(g,A_2,A_2')$ be mapping triples, and suppose there exists an $h\in G$ so that
\begin{enumerate}
\item $h$ induces a morphism from $A_1'$ to $A_2'$, and
\item $g^{-1}hg$ induces a morphism from $A_1$ to $A_2$.
\end{enumerate}
Then $L_g$ has equivalent restrictions at $A_1$ and $A_2$.
\end{proposition}
\begin{proof}Let $\psi\in\mathrm{Mor}(A_1',A_2')$ and $\varphi\in\mathrm{Mor}(A_1,A_2)$ be the induced morphisms, and note that $\psi_*\colon \partial A_1'\to\partial A_2'$ is a restriction of $L_h$ and $\varphi_*\colon A_1\to A_2$ is a restriction of $L_{g^{-1}hg}$.  Then the diagram
\[
\xymatrix{
\partial A_1 \ar_{L_g}[d] \ar^{\varphi_*}[r] & 
\partial A_2 \ar^{L_g}[d] \\ 
\partial A_1'\ar_{\psi_*}[r] & \partial A_2'}
\]
commutes since $L_gL_{g^{-1}hg} = L_hL_g$, so $L_g$ has equivalent restrictions at $A_1$ and~$A_2$.
\end{proof}

\begin{proof}[Proof of Theorem~\ref{thm:ActionIsRational}]Fix an element $g\in G$ with $|g|=k$.  Given a mapping triple $(g,A,A')$ with $A\in\A_n(\Gamma)$, define the \newword{signature} of $(g,A,A')$ to be the following information:
\begin{enumerate}
\item The sets $g\,P(A,S_n)$ and $P(A',S_{n-k})$,
\item The functions $g\,\od_{A}$ on~$g\,P(A,S_n)$ and $\od_{A'}$ on~$P(A',S_{n-k})$,
\item The set $g\,C(g^{-1}p)$ for each $p\in g\,P(A,S_n)$, and the cone $C(q)$ for each $q\in P(A',S_{n-k})$.
\end{enumerate}
We say that two mapping triples $(g,A_1,A_1')$ and $(g,A_2,A_2')$ with $A_1\in\A_m(\Gamma)$ and $A_2\in\A_n(\Gamma)$ have \newword{equivalent signatures} if there exists an $h\in G$ so that
\begin{enumerate}
\item $hg\,P(A_1,S_m) = g\,P(A_2,S_n)$ and $h\,P(A_1',S_{m-k}) = P(A_2',S_{n-k})$,
\item $hg\,\od_{A_1}$ agrees with $g\,\od_{A_2}$ on~$g\,P(A_2,S_n)$, and $h\,\od_{A_1'}$ agrees with $\od_{A_2'}$ on~$P(A_2',S_{n-k})$, and
\item $hg\,C(g^{-1}p) = g\,C(g^{-1}hp)$ for all $p\in g\,P(A_1,S_m)$ and $h\,C(q) = C(hq)$ for all $q\in P(A_1',S_{m-k})$.
\end{enumerate}
By Propositions~\ref{prop:MainMorphismTest} and~\ref{prop:EquivalentRestrictions}, if $(g,A_1,A_1')$ and $(g,A_2,A_2')$ have equivalent signatures, then $L_g$ has equivalent restrictions at $A_1$ and~$A_2$.  In particular, $h$~clearly induces a morphism from $A_1'$ to~$A_2'$ by Proposition~\ref{prop:MainMorphismTest}, and it is easy to show using Proposition~\ref{prop:MainMorphismTest} that $g^{-1}hg$ induces a morphism from $A_1$ to~$A_2$.

Finally, it is not hard to see that there are only finitely many equivalence classes of signatures for a given $g\in G$.  In particular, each of the sets $g\,P(A,S_n)$ and $P(A',S_{n-k})$ has diameter at most $8\delta+4$, so by Proposition~\ref{prop:DisksAreNear} the union $g\,P(A,S_n)\cup P(A',S_{n-k})$ has diameter at most $16\delta+8+2k$.  Thus there are only finitely many possible pairs $\bigl(g\,P(A,S_n),P(A',S_{n-k})\bigr)$ up to the action of~$G$.  Once such a pair is chosen, there are only finitely many possible choices for parts (2) and~(3) of the signature.  We conclude that $L_g$ has only finitely many restrictions, so $L_g$ is rational.
\end{proof}

\section{An Example}
\label{sec:AnExample}

In this section we work out a specific example of a tree of atoms and the corresponding rational homeomorphisms.  Let $\Gamma$ be the $1$-skeleton of order five square tiling of the hyperbolic plane (see~Example~\ref{ex:SquareTiling}), and fix a vertex $x_0$ of $\Gamma$.  Let~$G$ be the group of orientation-preserving isometries of~$\Gamma$.  We will demonstrate a rational action of~$G$.

\subsection{The Atoms}
\begin{figure}
\centering
\includegraphics{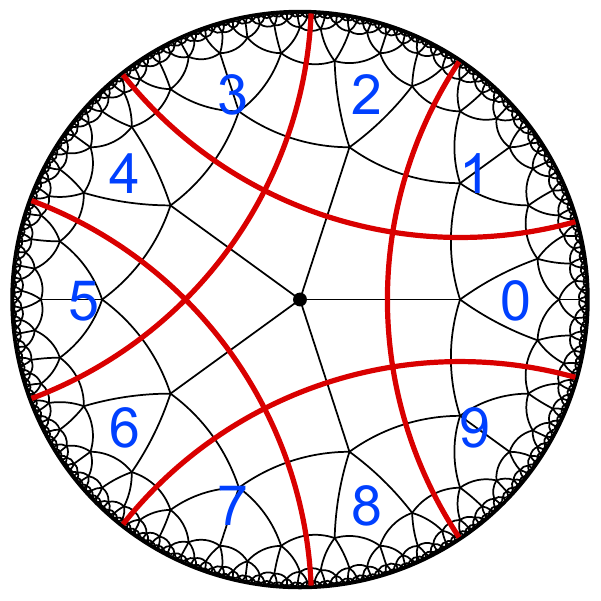}
\caption{The atoms of~$\A(B_1)$.}
\label{fig:SquareTilingAtoms1}
\end{figure}%
We begin by classifying the atoms in~$\Gamma$.  As in any graph, the first atom is the whole graph~$\Gamma$, which is the only atom of $\A_0(\Gamma)$, and is the root of the tree of atoms.  We will refer to this atom as having \newword{type~A}.

Figure~\ref{fig:SquareTilingAtoms1} shows the atoms of $\A(B_1)$.  There is one finite atom, namely the singleton set $\{x_0\}$, as well as ten infinite atoms, which we have labeled with the digits $0,\ldots,9$.  As suggested by the shapes, there are two types of infinite atoms here:
\begin{itemize}
\item Atoms $0$, $2$, $4$, $6$, and $8$ are bounded by two geodesic rays and a geodesic segment.  These atoms all have the same type, which we will refer to as \newword{type~B}.
\item Atoms $1$, $3$, $5$, $7$, and $9$ are bounded by two geodesic rays. These atoms all have the same type, which we will refer to as \newword{type~C}.
\end{itemize}
Thus the root of the tree of atoms has type~A, with ten children in~$\A_1(\Gamma)$ of types B and~C:
\[
\newcommand{\tA}{\mathrm{A}}
\newcommand{\tB}{\mathrm{B}}
\newcommand{\tC}{\mathrm{C}}
\xymatrix@C=0in@R=0in@M=0in@W=0.2in@H=0pt{
&&&&&&&&& \tA &&&&&&&&& \\
&&&&&&&&&
\ar@{-}[ddlllllllll]
\ar@{-}[ddlllllll]
\ar@{-}[ddlllll]
\ar@{-}[ddlll]
\ar@{-}[ddl]
\ar@{-}[ddr]
\ar@{-}[ddrrr]
\ar@{-}[ddrrrrr]
\ar@{-}[ddrrrrrrr]
\ar@{-}[ddrrrrrrrrr]
&&&&&&&&& \\ 
\rule{0pt}{36pt} &&&&&&&&&&&&&&&&&& \\
&&&&&&&&&&&&&&&&&& \\
\rule{0pt}{15pt}\tB && \tC && \tB && \tC && \tB && \tC && \tB && \tC && \tB && \tC
}
\]
\begin{figure}
\centering
$\underset{\textstyle \rule{0pt}{12pt}\{x_0\} }{\includegraphics{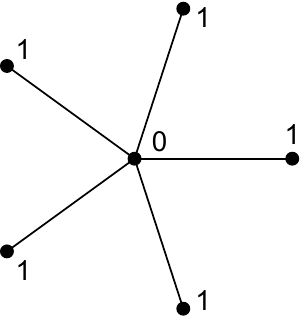}}$
\hfill
$\underset{\textstyle \rule{0pt}{12pt}\text{Atom~0} }{\includegraphics{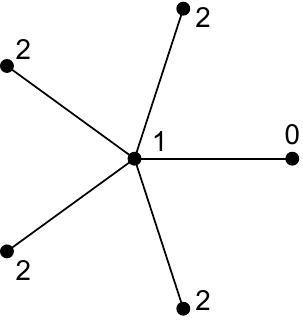}}$
\hfill
$\underset{\textstyle \rule{0pt}{12pt}\text{Atom~5} }{\includegraphics{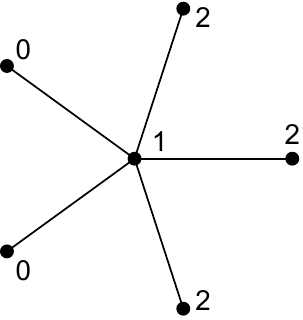}}$
\caption{Distance functions $\od_A$ on $B_1$ for three different atoms from~$\A(B_1)$.}
\label{fig:DistanceFunctions}
\end{figure}%
Figure~\ref{fig:DistanceFunctions} shows the distance functions $\od_A$ associated with the singleton atom $\{x_0\}$ and the atoms $0$ and $5$.  In each case, the additive constant has been chosen so that $\od_A$ has a minimum value of $0$ on~$B_1$.

\begin{figure}
\centering
\includegraphics{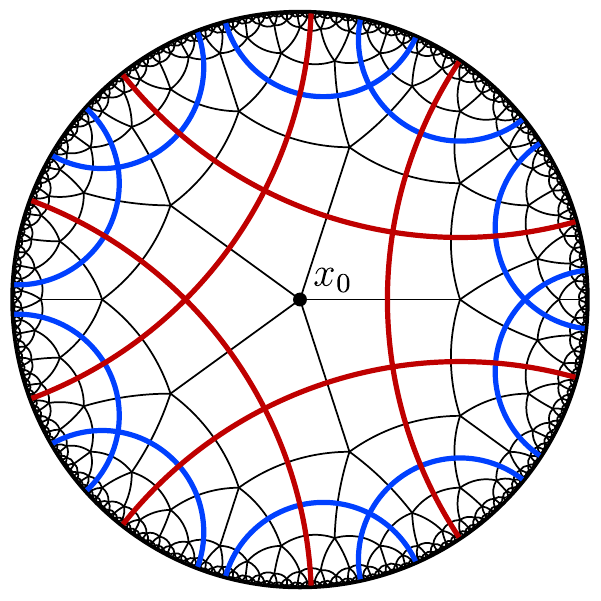}
\caption{The atoms of~$\A(B_2)$.}
\label{fig:SquareTilingAtoms2}
\end{figure}%
Of course, the atoms of $\A(B_1)$ are subdivided further in~$\A(B_2)$.
Figure~\ref{fig:SquareTilingAtoms2} shows the atoms of $\A(B_2)$, with the new subdivisions indicated in blue.  As you can see, each atom of type~B from $\A(B_1)$ has been subdivided into four atoms in $\A(B_2)$, and each atom of type~C from $\A(B_1)$ has been subdivided into three atoms in~$\A(B_2)$.  Thus $\A(B_2)$ has a total of $36$ atoms.  Of these, only $30$ are infinite, and therefore $\A_2(\Gamma)$ has $30$ elements.

\begin{figure}
\centering
\includegraphics{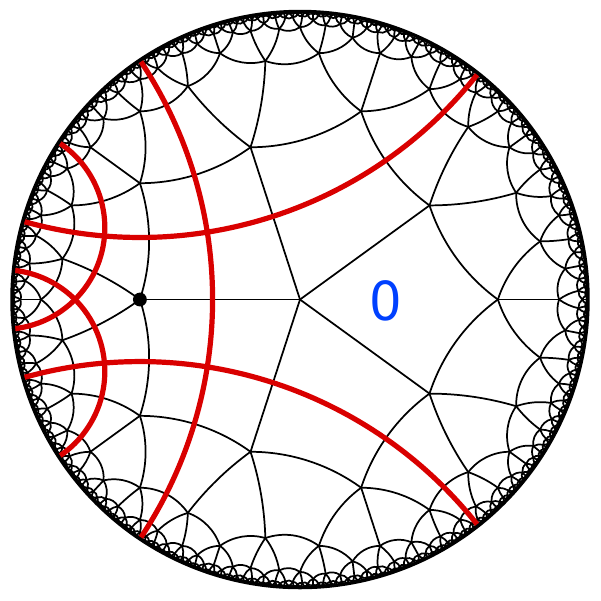}
\hfill
\includegraphics{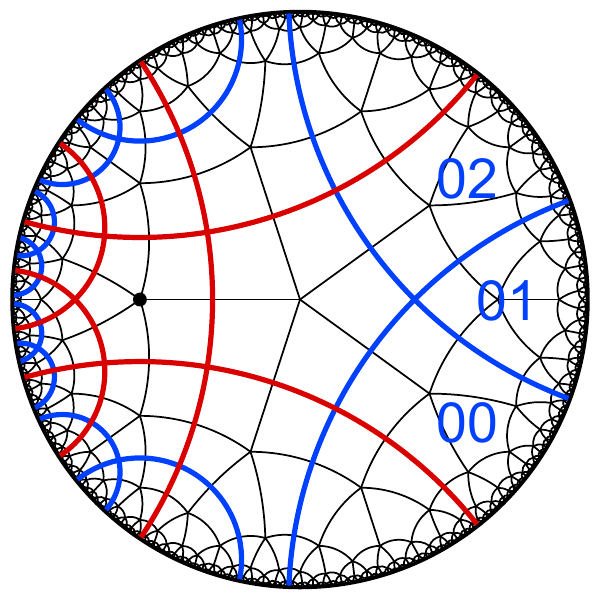}
\caption{Subdividing an atom of type~B.}
\label{fig:SquareTilingAtomB}
\end{figure}%
Figure~\ref{fig:SquareTilingAtomB} shows a close-up of atom~0 from $\A(B_1)$, which has type~B, as well as the subdivision of this atom in~$\A(B_2)$.  As the figure suggests, a type~B atom is subdivided into one singleton atom, two atoms of type~B, and one atom of type~C.  Thus every type~B node in the tree of atoms has three children of types B, C, and B:
\[
\newcommand{\tB}{\mathrm{B}}
\newcommand{\tC}{\mathrm{C}}
\xymatrix@C=0in@R=0in@M=0in@W=0.2in@H=0pt{
&& \tB && \\
&&
\ar@{-}[ddll]
\ar@{-}[dd]
\ar@{-}[ddrr]
&& \\ 
\rule{0pt}{24pt} &&&& \\
&&&& \\
\rule{0pt}{15pt}\tB && \tC && \tB
}
\]

\begin{figure}
\centering
\includegraphics{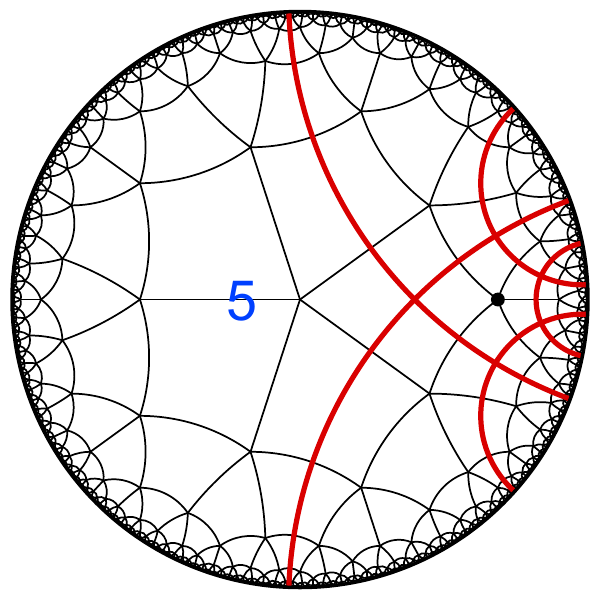}
\hfill
\includegraphics{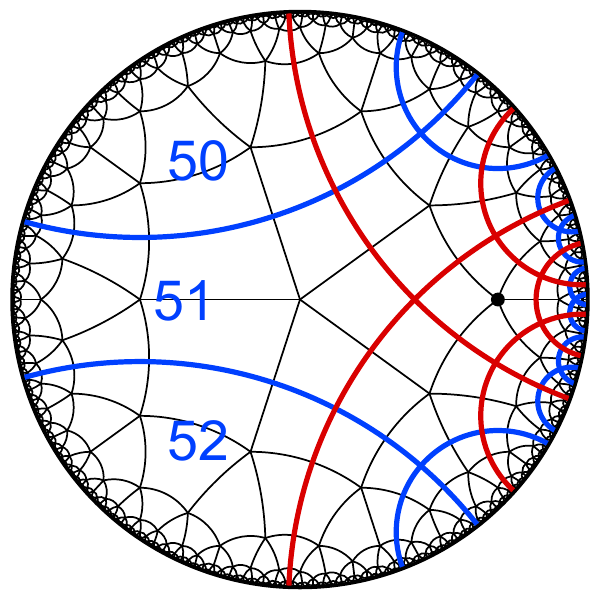}
\caption{Subdividing an atom of type~C.}
\label{fig:SquareTilingAtomC}
\end{figure}%
Figure~\ref{fig:SquareTilingAtomC} shows a close-up of the type~C atom in $\A(B_1)$ immediately to the left of the base vertex, as well as its subdivision in~$\A(B_2)$.  As the figure suggests, a type~C atom is subdivided into two atoms of type~C and one atom of a new type, which we refer to as \newword{type~D}.  Type~D atoms have a ``pentagon'' shape, and are bounded by two geodesic rays and two geodesic segments in the hyperbolic plane.  Thus every type~C node in the tree of atoms has three children of types C, D, and C:
\[
\newcommand{\tC}{\mathrm{C}}
\newcommand{\tD}{\mathrm{D}}
\xymatrix@C=0in@R=0in@M=0in@W=0.2in@H=0pt{
&& \tC && \\
&&
\ar@{-}[ddll]
\ar@{-}[dd]
\ar@{-}[ddrr]
&& \\ 
\rule{0pt}{24pt} &&&& \\
&&&& \\
\rule{0pt}{15pt}\tC && \tD && \tC
}
\]
Figure~\ref{fig:DistanceFunctionsLevel2} shows the distance functions $\od_A$ for the atoms 01 and~51.
\begin{figure}
\centering
$\underset{\textstyle\text{(a)}}{\includegraphics{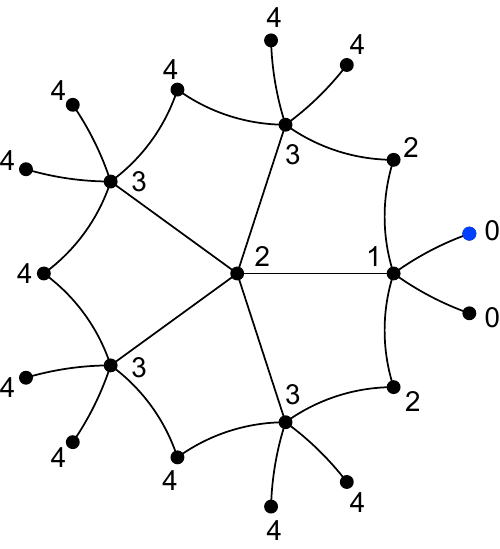}}$
\qquad\quad
$\underset{\textstyle\text{(b)}}{\includegraphics{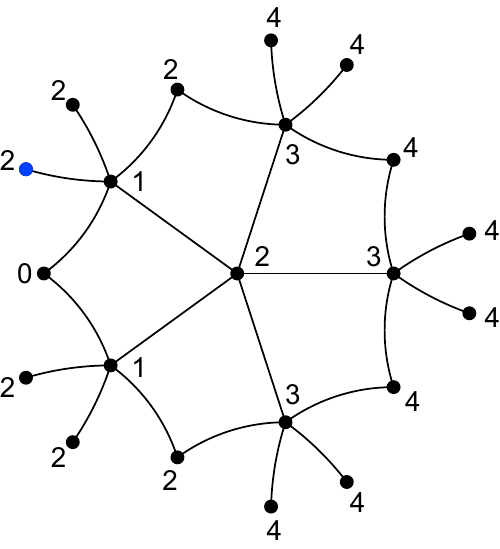}}$
\caption{(a) The distance function $\od_A$ for atom~01.  The one for atom~00 is the same except that the value for the blue vertex changes to~2. (b) The distance function $\od_A$ for atom~51.  The one for atom~50 is the same except that the value for the blue vertex changes to~0.}
\label{fig:DistanceFunctionsLevel2}
\end{figure}%

\begin{figure}
\centering
\includegraphics{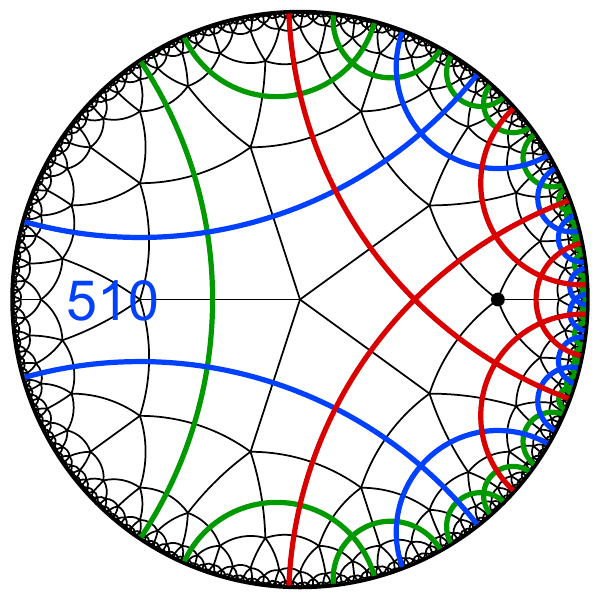}
\caption{Subdividing an atom of type~D.}
\label{fig:SquareTilingAtomD}
\end{figure}%
Finally, Figure~\ref{fig:SquareTilingAtomD} shows the subdivision in $\A(B_3)$ of this same atom of type~C.  As the picture suggests, the type~D child atom is subdivided in the next level into a singleton atom and an atom of type~B.  Thus every type~D node in the tree of atoms has exactly one child of type~B:
\[
\newcommand{\tB}{\mathrm{B}}
\newcommand{\tD}{\mathrm{D}}
\xymatrix@C=0in@R=0in@M=0in@W=0.2in@H=0pt{
\tD \\
\ar@{-}[dd]
\\ 
\rule{0pt}{24pt} \\
 \\
\rule{0pt}{15pt} \tB
}
\]
The type graph for the full tree of atoms is shown in Figure~\ref{fig:TransitionGraph}.  By Proposition~\ref{prop:IsomorphicToPathLanguage}, the tree $\A(\Gamma)$ is isomorphic to the set of all finite directed paths in this graph starting at~A, and $\hb\Gamma$ is naturally homeomorphic to the space of all infinite directed paths in this graph starting at~A.
\begin{figure}
\[
\entrymodifiers={++[o][F-]}
\xymatrix@C=0.7in@R=0.7in{
\text{A}\ar[r]^{0,2,4,6,8}\ar[d]_{1,3,5,7,9} & \text{B}\ar@(r,u)[]_{0,2}\ar[dl]_{1} \\
\text{C}\ar@(l,d)[]_{0,2}\ar[r]_{1} & \text{D}\ar[u]_{0}
}
\]
\caption{The type graph for the tree of atoms. Directed edges with multiple labels represent multiple edges.}
\label{fig:TransitionGraph}
\end{figure}
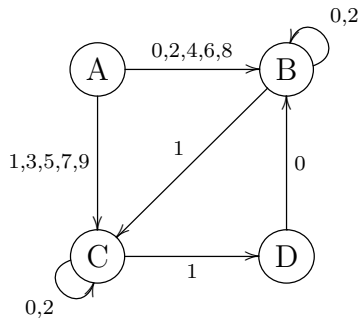

\begin{note}\label{note:NotGeometricallyEquivalent}Not every pair of atoms of the same type in $\A(\Gamma)$ are geometrically equivalent in the sense of Definition~\ref{def:GeometricEquivalence}.  In particular, since $\Gamma$ is not a tree the hyperbolic constant $\delta$ is at least~$1$, so $2\delta+1 \geq 3$.  Then it follows easily from the definition of proximal points (Definition~\ref{def:Proximal}) that $P(x,S_n)=S_n$ whenever $n\leq 3$, and in particular $P(A,S_1) = S_1$ for all $A\in\A_1(\Gamma)$ and $P(A,S_2) = S_2$ for all $A\in\A_2(\Gamma)$.  Since $S_1$ has five vertices and $S_2$ has fifteen vertices, it follows that no atom in $\A_1(\Gamma)$ is geometrically equivalent to an atom in~$\A_2(\Gamma)$, even though both $\A_1(\Gamma)$ and $\A_2(\Gamma)$ contain atoms of types~B and~C.\qed
\end{note}

\subsection{The Group}Let $G$ be the group of orientation-preserving isometries of~$\Gamma$.  Because we are restricting to orientation-preserving isometries, there is a unique morphism between any two atoms of the same type, so the tree of atoms is rigid.

The group $G$ has presentation
\[
\big\langle r,s \;\bigl|\; r^5,s^2,(rs)^4\big\rangle
\]
where
\begin{enumerate}
\item $r$ is a counterclockwise rotation by $2\pi/5$ at the base vertex~$x_0$, and 
\item $s$ is a rotation by $\pi$ at the point~$p$ shown in Figure~\ref{fig:SquareTilingGroupGenerators}.
\end{enumerate}
\begin{figure}
\centering
\includegraphics{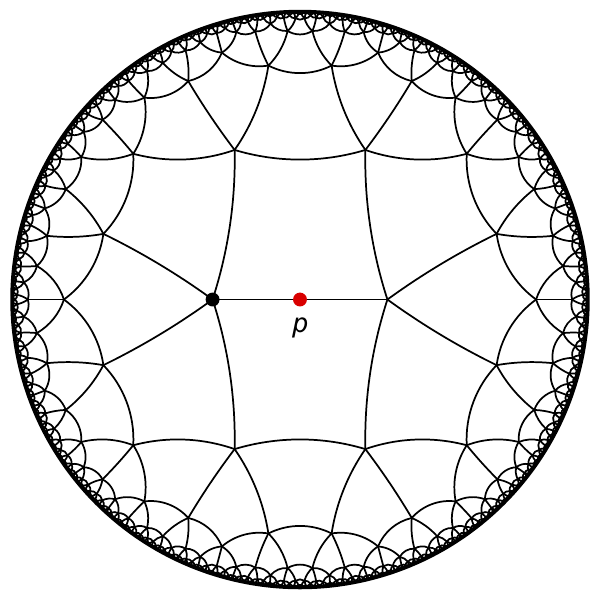}
\caption{The group $G$ is generated by the stabilizer of $x_0$ and a $180^\circ$ rotation at the point~$p$.}
\label{fig:SquareTilingGroupGenerators}
\end{figure}%
The rational homeomorphism for $r$ is given by the formulas
\[
\begin{array}{lllll}
r(0\beta) = 2\beta,&
r(1\gamma) = 3\gamma,&
r(2\beta) = 4\beta,&
r(3\gamma) = 5\gamma,&
r(4\beta) = 6\beta,\\[6pt]
r(5\gamma) = 7\gamma,&
r(6\beta) = 8\beta,&
r(7\gamma) = 9\gamma,&
r(8\beta) = 0\beta,&
r(9\gamma) = 1\gamma,
\end{array}
\]
where $\beta$ can be any infinite path in the type graph starting at~B, and $\gamma$ can be any valid infinite path in the type graph starting at~C.

\begin{figure}
\centering
\subfloat[\label{subfig:SquareTilingGenerator2}]{\includegraphics{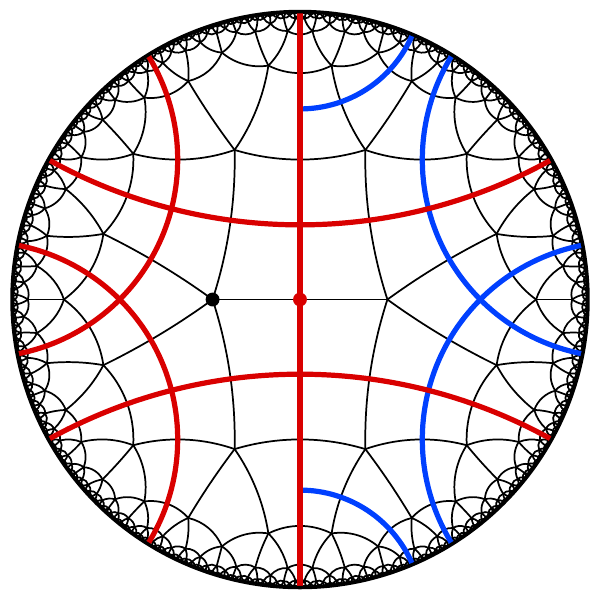}}
\hfill
\subfloat[\label{subfig:SquareTilingGenerator3}]{\includegraphics{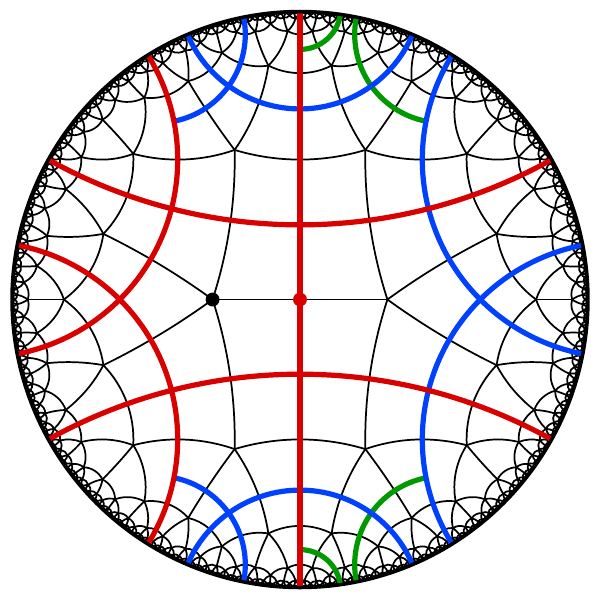}}
\caption{Two subdivisions of $\hb \Gamma$ into atoms.}
\end{figure}%
The formula for $s$ is more complicated. Figure~\subref*{subfig:SquareTilingGenerator2} shows a subdivision of~$\Gamma$ into certain atoms of $\A(B_1)$ and $\A(B_2)$.  From this subdivision, we can see that
\[
\begin{array}{l@{\quad\quad}l@{\quad\quad}l@{\quad\quad}l}
s(0\beta) = s_0(\beta)&
s(1\gamma) = s_1(\gamma)&
s(2\beta) = 9f(\beta) &
s(3\gamma) = 92\gamma \\[6pt]
s(4\beta) = 00\beta &
s(5\gamma) = 01\gamma &
s(6\beta) = 02\beta &
s(7\gamma) = 10\gamma \\[6pt]
s(8\beta) = 1\of(\beta) &
s(9\gamma) = s_9(\gamma)
\end{array}
\]
where $\beta$, $\gamma$, and $\delta$ represent any infinite paths in the type graph starting at B, C, and D, respectively, and
\[
\begin{array}{l@{\quad\quad}l@{\quad\quad}l}
s_0(0\beta) = 4\beta&
s_0(1\gamma) = 5\gamma&
s_0(2\beta) = 6\beta \\[6pt]
s_1(0\gamma) = 7\gamma &
s_1(1\delta)= 8h(\delta) &
s_1(2\gamma) = 8g(\gamma) \\[6pt] 
s_9(0\gamma) = 2\og(\gamma) &
s_9(1\delta) = 2\oh(\delta) &
s_9(2\gamma) = 3\gamma \\[6pt]
\end{array}
\]

Figure~\subref*{subfig:SquareTilingGenerator3} shows a refinement of this subdivision that can be used to determine $f$, $\of$, $g$, $\og$, $h$, and~$\oh$:
\[
\begin{array}{l@{\quad\quad}l@{\quad\quad}l}
f(0\beta) = 0f(\beta) &
f(1\gamma) = 02\gamma &
f(2\beta) = 10\beta \\[6pt]
\of(0\beta) = 10\beta &
\of(1\gamma) = 20\gamma &
\of(2\beta) = 2\of(\beta) \\[6pt]
g(0\gamma) = 1\gamma &
g(1\delta) = 2h(\delta) &
g(2\gamma) = 2g(\gamma) \\[6pt]
\og(0\gamma) = 0\og(\gamma) &
\og(1\delta) = 0\oh(\delta) &
\og(2\gamma) = 1\gamma \\[6pt]
h(0\beta) = 0\beta &
\oh(0\beta) = 2\beta
\end{array}
\]
Each of the letters $s$, $s_0$, $s_1$, $s_9$, $f$, $\of$, $g$, $\og$, $h$, and~$\oh$ represents an equivalence class of atoms on which~$s$ has equivalent restrictions, and there are also classes for the identity functions $\mathrm{id}_B$, $\mathrm{id}_C$, and $\mathrm{id}_D$ corresponding to types B, C, and~D, respectively.  The atoms corresponding to each restriction type are as follows:
\begin{itemize}
\item $s$: the root atom only
\item $s_0$: atom $0$ only
\item $s_1$: atom $1$ only
\item $s_9$: atom $9$ only
\item $f$: $20^n$ for any $n\geq 0$
\item $\of$: $12^n$ for any $n\geq 0$
\item $g$: $12^n$ for any $n\geq 0$
\item $\og$: $90^n$ for any $n\geq 0$
\item $h$: $12^n1$ for any $n\geq 0$
\item $\oh$: $90^n1$ for any $n\geq 0$ 
\item $\mathrm{id}_B$: all other type B atoms 
\item $\mathrm{id}_C$: all other type C atoms 
\item $\mathrm{id}_D$: all other type D atoms
\end{itemize}

\bibliography{bibHypEmbedding}{}
\bibliographystyle{amsplain}

\end{document}